\documentclass[11pt,a4paper]{article}
\usepackage[body={154mm,229mm},top=34mm]{geometry}
\usepackage{amscd}
\usepackage{mathrsfs}
\usepackage{epsfig}
\usepackage{amsmath}
\usepackage{amssymb}
\usepackage{amsthm}
\usepackage{epsfig}
\usepackage{verbatim}
\usepackage{url}
\usepackage{color}
\usepackage{hyperref}
\include{pspicture}

\newcommand{\ninsepsc}[3]{
\begin{figure}[h]
\begin{center}
 \scalebox{#3}{\includegraphics{#1}}
\end{center}

\vspace{-0.65cm}
\caption{\hspace{0.25cm}#2\label{f:#1}}
\end{figure}
}

\newtheorem{lemma}{Lemma}

\newtheorem{definition}{Definition}[section]
\newtheorem{remark}{Remark}

\newtheorem{proposition}{Proposition}[section]

\usepackage{makeidx}
\usepackage{mathabx}
\usepackage{mathrsfs} 
\begin{document}
\title{Exit Probabilities and Balayage of Constrained Random Walks}
\author{Ali Devin Sezer\footnote{Middle East Technical University, Institute of Applied Mathematics, Ankara, Turkey}}
\maketitle
\begin{abstract}
Let $X$ be the constrained random walk on ${\mathbb Z}_+^d$, $d \in \{2,3,4,...\}$,
representing the queue lengths of a stable Jackson network
and let $x \in {\mathbb Z}_+^d $ be its initial position ($X$ is a random walk
with independent and identically distributed increments except that its dynamics are 
constrained on the boundaries of ${\mathbb Z}_+^d$  so that $X$ remains 
in ${\mathbb Z}_+^d$;
stability means that $X$ has a nonzero drift pushing it to the origin).
Let $\tau_n$ be the first time when the sum of the components of $X$ equals 
$n$.  The probability $p_n \doteq P_x(\tau_n < \tau_0)$ is one of the key performance
measures for the queueing system represented by $X$ and its analysis/computation
received considerable attention over the last several decades.
The stability of $X$ implies that $p_n$ decays exponentially in $n$.
Currently the only analytic method available to approximate $p_n$ is large
deviations analysis, which gives the exponential decay rate of $p_n$. Finer approximations
are available via rare event simulation.  The present article develops a new
method to approximate $p_n$ and related expectations.
The method has two steps: 1) with an affine transformation,
move the origin to a point on the exit boundary associated with $\tau_n$;
let $n\rightarrow \infty$ to remove some of the constraints on the dynamics
of the walk; the first step gives a limit {\em unstable} /{\em transient} 
constrained random walk $Y$
2) construct a basis of harmonic functions of $Y$ and
use this basis to apply the classical superposition principle of linear analysis (the basis
functions can be seen as perturbations of the classical Fourier basis).
The basis functions
are linear combinations of $\log$-linear functions and come from solutions of
{\em harmonic systems}; these are graphs with labeled edges whose vertices represent
points on the interior {\em characteristic surface} ${\mathcal H}$ of $Y$;
the edges between the vertices represent conjugacy relations between the points
on the characteristic surface, the loops (edges from a vertex to itself) represent
membership in the boundary characteristic surfaces.
Characteristic surfaces are algebraic varieties determined by the distribution of
the unconstrained increments of $X$ and the boundaries of ${\mathbb Z}_+^d$.
Each point on ${\mathcal H}$ defines a harmonic function of the unconsrained version
of $Y$.
Using our method we derive explicit, simple and almost exact formulas 
for $P_x(\tau_n < \tau_0)$ for $X$ representing $d$-tandem queues, 
similar to the product form formulas for
the stationary distribution of $X$.  The same method allows us to approximate the
Balayage operator mapping $f$ to
$x \rightarrow {\mathbb E}_x
\left[ f(X_{\tau_n}) 1_{\{ \tau_n < \tau_0\}} 
\right]$ for a range of stable constrained random walks representing the queue lengths
of a queueing system with two nodes (i.e., $d=2$).
We provide two convergence theorems; one using the coordinates of the limit
process and one using the scaled coordinates of the original process.
The latter is given for two tandem queues (i.e., when the set of possible 
increments of $X$ is $\{ (0,1), (-1,1)(0,-1)\}$)
and uses a sequence of subsolutions of a related Hamilton Jacobi Bellman 
equation {\em on a manifold}; the manifold consists
of three copies of ${\mathbb R}_+^2$, the zeroth
glued to the first along $\{x:x(1)=0\}$ and the first to the second
along $\{x:x(2) =0\}.$
We indicate how the ideas of the paper relate to more general processes
and exit boundaries.
\end{abstract}

\newpage
\tableofcontents
\newpage

\section{Introduction}
Constrained random walks arise naturally 
as models of queueing networks and this paper treats only
walks associated with Jackson 
networks. But the approach of the paper applies
more generally, see subsection \ref{ss:possiblegen}.

Let $X$ denote the number of customers in 
the queues of a $d$-node Jackson network
at arrival and service completion times ($X_k(i)$ is the number of 
customers waiting
in queue $i$ of the network right after the $k^{th}$ arrival/service completion); 
mathematically,
$X$ is a constrained random walk
on ${\mathbb Z}^d_+$, i.e., it has independent increments except that on 
the boundaries of ${\mathbb Z}^d_+$ the process is constrained to 
remain on ${\mathbb Z}^d_+$ (see \eqref{d:constX}
for the precise definition of the constrained random walk $X$).
Define
\begin{equation}\label{e:defA}
A_n = \left\{ x \in {\mathbb Z}_+^d: \sum_{i=1}^d x(i) \le n \right\}
\end{equation}
\index{$A_n$}
and its boundary
\begin{equation}\label{e:defBoundaryA}
\partial A_n =
  \left\{ x \in {\mathbb Z}_+^d: \sum_{i=1}^d x(i) = n \right\}.
\end{equation}
Let $\tau_n$ be the first time $X$ hits $\partial A_n$.
One of the ``exit probabilities'' that the title refers to is
$p_n \doteq P_x( \tau_n < \tau_0)$,
the probability that starting from an initial state $x \in A_n$
the number of
customers in the system reaches $n$ before the system empties. One of our
primary aims in this paper  will be the approximation of this probability.
The set $A_n$
models a systemwide shared buffer of size $n$ 
(for example, if the queueing system models a set of computer 
programs running
on a computer, the shared buffer may be the computer's memory)
and
 $\tau_n$ represents the first time this buffer overflows. 
If we measure
time in  the number of independent cycles that restart each time $X$ hits $0$, $p_n$ is 
the probability that the current cycle finishes successfully (i.e., without a buffer
overflow). 

One can change the domain $A_n$ to model other buffer structures, e.g.,
$ \{ x \in {\mathbb Z}_+^d: x(i) \le n\}$
models separate buffers of size $n$ for each queue in the system. 
The present work focuses on 
the domain $A_n$. The basic ideas of the paper
apply to other domains, and we comment on this in the
conclusion.

For a set $a$ and $\tau_a \doteq \{k : X_k \in a\}$, 
the distribution ${\mathcal T}_a$ of 
$X_{\tau_a}$ on $a$ is called the Balayage operator.  ${\mathcal T}_a$
maps bounded measurable functions on $a$ to
harmonic functions on $a^c$:
\[
{\mathcal T}_a: f \rightarrow g,  
g(x) = {\mathbb E}_x\left[ f\left(X_{\tau_a}\right) 1_{\{ \tau_a < \infty \}} \right].
\]
The computation of $p_n$ is a special case of the computation of 
(the image of a given function
under) the Balayage operator:
for $a = A_n^c \cup \partial A_n \cup \{0\}$,
$\tau_a$ becomes $\tau_n \wedge \tau_0$
and if we set
\[
f = 1_{ \{ \partial A_n\} },
\]
$({\mathcal T}_a f)(x) 
= {\mathbb E}_x\left[ f(X_{\tau_a}) 1_{\{ \tau_a < \infty \}} \right]$,
$x \in A_n$, equals $P_x( \tau_n < \tau_0)$.

We assume that $X$ is stable, i.e., 
the total arrival rate $\nu_i$ is less than the total service rate
$\mu_i$ for all nodes $i$ of the queueing system that $X$ represents 
($\mu$ and $\nu$ are linear functions of the distribution 
of the increments of $X$, see
\eqref{d:lambdamu} and \eqref{d:defnu} for their definitions).
For a stable $X$, the event $\{ \tau_n <\tau_0\}$ rarely happens
and its probability $p_n$ decays 
exponentially with buffer size $n$. 
The problem of approximating $p_n$ has a long history and an extensive literature; 
let us mention two of the main approaches here. The first is
large deviations (LD) analysis  
\cite{dupell2, MR1619036, MR1335456}
which gives the exponential decay rate of
$p_n$ as the value function of a limit deterministic optimal control problem (see below).
If one would like to obtain more precise
estimates than what LD analysis gives, the popular method has so far been 
simulation with variance reduction such as importance sampling, see, 
\cite[Chapter VI]{asmussen2007stochastic}
and \cite{ GlassKou, DSW}; 
the use of IS for similar problems in a single
dimension goes back to \cite{Siegmund}.
The goal of this paper is to offer a new alternative,  which, in particular,
allows to approximate the Balayage operator
for a wide class of two dimensional systems
and gives an almost exact formula for the probability
$P_x( \tau_n < \tau_0)$ for tandem networks in any dimension.
We explain its elements in the following paragraphs.

One way to think of the LD analysis is as follows. 
$p_n$ itself
decays to $0$, which is trivial. To get a nontrivial limit 
transform $p_n$ to $V_n \doteq -\frac{1}{n}\log p_n$;
using convex duality, one can write
the $-\log$ of an expectation as an optimization
problem involving the relative entropy \cite{dupell2} and thus
$V_n$ can be interpreted 
as the value function of a discrete time stochastic optimal control 
problem. 
The LD analysis
consists of the law of large numbers limit analysis of this control problem;
the limit problem is a deterministic optimal control problem whose value function
satisfies a first order Hamilton Jacobi Bellman equation (see
 \eqref{e:HJB0} of Section \ref{s:convergence2}).
Thus, LD analysis amounts to the computation of the
limit of a {\em convex  transformation} of the problem.

We will use another, an {\em affine}, transformation of $X$ for the limit analysis.
The proposed transformation is extremely simple:  
{\em observe $X$ from the exit boundary.}
The most natural vantage points
on the exit boundary $\partial A_n$
are the corners $\{ ne_i, i=1,2,3,...,d\}$, where $e_i$ are the standard basis
elements of ${\mathbb Z}^n$:
\begin{equation}\label{e:defTn}
Y^n \doteq T_{n}(X), ~~~ T_{n}: {\mathbb R}^d\rightarrow {\mathbb R}^d, 
T_{n}(x) \doteq y, ~~~ y(j) = \begin{cases} n - x(j), &\text{ if } j  = i, \\
				x(j) &\text{ otherwise},
\end{cases}
\end{equation}
$j =1,2,\cdots, d.$
\index{T@$T_n$ the affine transformation mapping $X$ to $Y^n$}
$T_{n}$ is affine and its inverse equals itself.
$Y^n$, i.e., the process $X$ as observed from the corner $n e_i$,
 is a constrained process on the domain $
\Omega_Y^n\doteq {\mathbb Z}_+^d \times (n - {\mathbb Z}_+)
\times {\mathbb Z}_+^d$; it is the same process
as $X$, except that $Y^n$ represents
the state of the $i^{th}$ queue not by the number of customers waiting in queue $i$
but by the number of spots in the buffer not occupied by the customers
in queue $i$. 
\index{Y@$Y^n$}
$T_{n}$ maps
the set $A_n$ to
$B_n \subset \Omega_Y^n$, $B_n \doteq T_{n} (A_n)$;
the corner $n e_i$
to the origin of $\Omega_Y^n$; 
the exit boundary  $\partial A_n$
to 
$\partial B_n \doteq \{ y \in \Omega_Y^n, y(i) = \sum_{j=1, j \neq i }^d y(j) \}$; 
finally the constraining
boundary $\{z \in {\mathbb Z}_+^d, z(i)=0\}$ to
\[
\{ y \in {\mathbb Z}_+^{i-1}\times {\mathbb Z} \times {\mathbb Z}_+^{d-i}:
y(i) = n \}.
\]
As $n\rightarrow \infty$ 
the last boundary vanishes and $Y^n$ converges to
the 
limit process $Y$ on the domain
$\Omega_Y \doteq {\mathbb Z}_+^{i-1} \times {\mathbb Z} \times {\mathbb Z}_+^{d-i}$
and
the set $B_n$ to 
\begin{equation}\label{e:defB}
B \doteq \left\{ y \in \Omega_Y, y(i) \ge \sum_{j=1, j \neq i }^d y(j) \right\}.
\end{equation}
Figure \ref{f:Tn2tex} sketches these transformations 
for the case of $X$ representing lengths of two tandem queues and for $i=1$
(the random walk $X$ represents tandem queues if its set of possible jumps
are $e_1$, $-e_i + e_{i+1}$ , $i=1,2,3,...,d-1$ and $-e_d$, if $X$ is of this
form we will call it 
a ``tandem walk;'' for the exact definition, see \eqref{e:poftandem}).

\begin{figure}[h]
\begin{center}
\scalebox{0.7}{
\centerline{\input{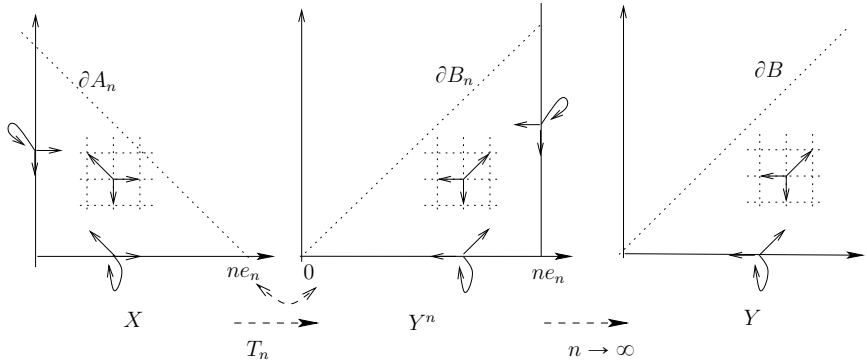}}}
\end{center}

\vspace{-0.65cm}
\caption{\hspace{0.25cm}The transformation $T_n$\label{f:Tn2tex}}
\end{figure}

The boundary of $B$ is 
\begin{equation}\label{e:defBb}
\partial B = \left\{ y \in \Omega_Y, y(i) = \sum_{j=1, j \neq i }^d y(j) \right\};
\end{equation}
the limit stopping time
\begin{equation}\label{e:deftau}
\tau \doteq \inf \{k: Y_k \in \partial B\}
\end{equation}
is the first time $Y$ hits $\partial B$.
{\em 
The stability of $X$ and the vanishing of the boundary constraint on $i$ implies
that $Y$ is {\em unstable} / {\em transient}, i.e., with probability $1$ it wanders off to $\infty$}.
Therefore, in our formulation, the limit process is an {\em unstable} constrained random walk
in the same space and time scale as the original process
but with less number of constraints.

Fix an initial point $y \in B$ in the new coordinates; our first
convergence result is Proposition \ref{t:c1} which says
\begin{equation}\label{e:conv1}
p_n = P_{x_n}( \tau_n < \tau_0 )
\rightarrow  P_y( \tau < \infty),
\end{equation}
where $x_n = T_n(y)$.
The proof uses the
law of large numbers and LD lowerbounds to show that the difference between the two
sides of \eqref{e:conv1} vanishes with $n$.
With \eqref{e:conv1} we see that 
the limit problem in our formulation is to compute the hitting probability of
the unstable $Y$ to 
the boundary $\partial B$.

The convergence statement \eqref{e:conv1} involves a fixed initial condition for the
process $Y$. In classical LD analysis, one specifies the initial point
in scaled coordinates as follows: 
$x_n = \lfloor n x \rfloor \in A_n $ for $x \in {\mathbb R}_+^d$.
Then the initial
condition for the $Y^n$ process will be
$y_n = T_n(x_n)$
(thus we fix not the $y$ coordinate but the
scaled $x$ coordinate). 
When $x_n$ is defined in this way, \eqref{e:conv1}
becomes a trivial statement because its both sides decay to $0$.
For this reason, Section \ref{s:convergence2} studies
the relative error
\begin{equation}\label{e:conv2}
\frac{| P_{x_n}( \tau_n < \tau_0 ) - P_{y_n}( \tau < \infty)|}{P_{x_n}(\tau_n < \tau_0)};
\end{equation}
Proposition \ref{t:guzel} says that this error converges exponentially to $0$
for the case of two dimensional tandem walk (i.e., the process $X$ shown in
Figure \ref{f:Tn2tex}).
The proof rests on
showing that the probability of the intersection of the events 
$\{\tau_n < \tau_0\}$ and $\{ \tau < \infty\}$ 
dominate the probabilities of both as $n\rightarrow \infty$. 
For this we calculate bounds in Proposition \ref{t:thm1}
on the LD decay rates of the probability
of the differences between these events using
a sequence of subsolutions of a Hamilton Jacobi Bellman equation {\em on a manifold};
the manifold consists of three copies of ${\mathbb R}_+^2$,
zeroth copy glued to the first along $\partial_1$, and the first to the second
along $\partial_2$, where $\partial_i = \{x\in {\mathbb R}_+^2: x(i) = 0 \}.$
Extension of this argument to more complex processes and
domains remains for future work.

For a process $X$ in $d$ dimensions, each affine transformation $T_n^i$, $i=1,2,3,...,d$,
gives a possible approximation of $P_x(\tau_n < \tau_0)$. A key question is:
which of these best approximates $P_x(\tau_n < \tau_0)$ for a given $x$?
Proposition \ref{t:guzel} says that, for the two dimensional tandem walk, $i=1$ works well 
for all points $x = \lfloor nx \rfloor$ as long as $x(1) > 0$.
In general this
will not be the case (i.e., depending on $x$, removing one constraint of the
process may give better approximations than removing another); subsection 
\ref{ss:combineapp} comments on this problem.

The convergence results
\eqref{e:conv1} and \eqref{e:conv2} reduce the problem of calculation
of $P_x(\tau_n < \tau_0)$ to that of $P_y(\tau < \infty)$. 
This constitutes the first
step of our analysis and we expect it to apply more generally;
see subsection \ref{ss:possiblegen}. 
Computation of $P_y(\tau < \infty)$ 
is a static linear problem and can be attacked
with a range of ideas and methods.

Sections \ref{s:twodim},
\ref{s:dg2} and \ref{s:hstq} apply the {\em principle of superposition}
of classical linear analysis to the computation of $P_y(\tau < \infty)$ and
related expectations.
The key for its application is to construct
the right class of efficiently computable basis functions to be superposed.
The construction of our basis functions goes as follows:
the distribution
of the increments of $Y$ is used to define the {\em characteristic polynomial}
${\mathbf p}:{\mathbb C}^n \rightarrow {\mathbb C}$.
${\mathbf p}$ can be represented both as a rational
function and as a polynomial; to simplify our analysis we use the polynomial representation
in two dimensions and the rational one in $d$ dimensions. 
In the rest of this paragraph we will
only refer to the higher dimensional definitions; the definitions for the
case of two dimensions are given in subsection \ref{ss:charpolsurf}.
We call the $1$ level set of ${\mathbf p}$, the {\em characteristic surface} of $Y$
and denote it with ${\mathcal H}$, see \eqref{d:DefHcal}.
${\mathcal H}$ is, more precisely, a $d-1$ dimensional complex affine algebraic variety of degree $d+1$.
Each point on the characteristic surface ${\mathcal H}$
defines a $\log$-linear function (see \eqref{d:basicloglinear})
that satisfies the interior harmonicity condition of $Y$ (i.e., defines a harmonic
function of the completely unconstrained version of $Y$); 
similarly, each boundary of the state
space of $Y$ has an associated characteristic polynomial and surface.
${\mathbf p}$ can be written as a second order polynomial in each of its arguments;
this implies that most points on ${\mathcal H}$ come in conjugate pairs, there
are $d-1$ different conjugacy relations, one for each constraining boundary of $Y$.
The keystone of the approach developed in these sections is the following
observation:
{\em $\log$-linear functions defined by two points 
on ${\mathcal H}$ satisfying a given type of conjugacy relation
can be linearly combined to get nontrivial functions which satisfy
the corresponding boundary harmonicity condition
(as well as the interior one)}; see Proposition \ref{p:harmonicYtwoterms}. 
Based on this
observation
we introduce the concept of a {\em harmonic system} (Definition \ref{d:Yharmonic})
which is an
{\em edge-complete} graph with labeled edges representing a system of variables
and equations:
the vertices represent the variables constrained to be on ${\mathcal H}$,
the edges between distinct vertices represent the conjugacy relations between the
variables that the edges connect (the label of the edge determines the type of 
the conjugacy relation)
and its loops (an edge from a vertex to itself) represent membership on a boundary
characteristic surface (the label of the loop determines which boundary characteristic
surface).
We show that any solution to a
harmonic system gives a harmonic function for $Y$ in the form of linear combinations
of $\log$-linear functions (each vertex defines a $\log$-linear function). 
The computational
complexity of the evaluation of the resulting harmonic function is
essentially determined by the size of the graph.

In two dimensions (Section \ref{s:twodim}) edge-complete graphs have $1$ or $2$
vertices and the 
above construction gives a rich enough basis of harmonic functions of $Y$
to approximate the image of any function on $\partial B$ under the Balayage operator;
with the use of these basis functions
the approximation of ${\mathbb E}_y[f(Y_{\tau}) 1_{\{\tau < \infty\}}]$
for any given bounded $f$
reduces to the solution of a linear equation in $K$ dimensions, where $K$ is the 
number of basis functions used in the approximation (Section \ref{ss:gtdw} gives
an example with $K= 12$).
Once the approximation is computed,
the error made in the approximation is simple to bound when
$f$ is constant outside of a bounded support and satisfies $f>0$ or $f<0$.
The restrictions of the basis functions  on ${\partial B}$
are perturbed versions of the restriction of the
 ordinary Fourier basis on ${\mathbb Z}$ to ${\mathbb Z}_+$; 
for this reason we call the constructed basis a ``perturbed'' Fourier basis. 

Section \ref{s:dg2} gives the definition of a harmonic system for $d$-dimensional
constrained walks and prove that any solution to a harmonic system defines
a $Y$-harmonic function  (Proposition \ref{p:simpleharmonicfunctions}).
In Section \ref{s:hstq} we compute
explicit solutions to a particular class of harmonic systems for the $d$ dimensional
tandem walk;
the span of the class
contains
$P_y( \tau < \infty)$ exactly. Hence we obtain explicit formulas, similar to
the product form formulas for the stationary distribution of Jackson networks
\cite{kelly2011reversibility},
for $P_y(\tau < \infty)$ in the case of constrained walks representing
tandem queues in arbitrary dimension (see 
Proposition \ref{p:exactformulaDtandem}; the two dimensional version
of the same formula is written out more explicitly in display \eqref{e:nicerep}).
If we take exponentiation and algebraic operations
to be atomic, the complexity of evaluating the
formula is independent of $y$ (and hence of the buffer size 
$n$ if $P_y(\tau < \infty)$ is used
to approximate $P_x(\tau_n < \tau_0)$ ) and depends only on the dimension $d$
of $X$.

Section \ref{s:examples} gives example computations using the approach developed in
the paper. Three examples are considered: the tandem walk in two dimensions, 
a non tandem walk in two dimensions, tandem walk in $4$ and $14$ dimensions.
The conclusion (Section \ref{s:conclusion}) discusses several directions for future
research. Among these are the application of the approach of the present paper
to constrained diffusion processes (subsection \ref{ss:diffusionswithdrift}),
 the study of nonlinear perturbed second order 
HJB equations which arise when one would like to sharpen large deviations estimates
(subsection \ref{ss:perturbedPDE}) and the sizes of boundary layers which arise
in subsolution based IS algorithms applied to constrained random walks (subsection
\ref{ss:boundarylayers}).

\section{Definitions}
This section sets the notation of the paper, defines the domains, the processes
and the stopping times we will study and states some elementary facts about them.

We will denote  components of a vector using
parentheses, e.g., for $x\in {\mathbb Z}_+^d$, $x(i)$, $i=1,2,3,...,d$
denotes the $i^{th}$ component of $x$.

For two sets $a$ and $b$, $a^b$ denotes the set of functions from
$b$ to $a$. For a function $f$ on a set $a$, we will write
$f|_c$ to mean $f$'s restriction to a subset $c \subset a$.
For a finite set $a$, $|a|$ denotes the number of its elements. We will
assume that elements of sets are written in a certain order and we will
index sets as we index vectors, e.g., $a(1)$ denotes the first element of $a$
and $a(|a|)$ \index{a@$a(\cdot) a("|a"|)$} the last.

Our analysis will involve several types of boundaries:
the coordinate hyperplanes of ${\mathbb Z}^d$, the constraining 
boundaries of constrained 
processes and the boundaries of exit sets.
To keep our notation short and manageable, we will make use of
the symbol $\partial$ to indicate that a set is a boundary of some type.

Define 
${\mathcal N}_0 \doteq \{0,1,2,...,d\} $ and ${\mathcal N}_+ \doteq {\mathcal N}_0-\{0\}$; ${\mathcal N}_+$
is the set of nodes of $X$.
\index{N@${\mathcal N}_0, {\mathcal N}_0"_+$}
For $a \subset {\mathcal N}_+$ the coordinate hyperplanes of ${\mathbb Z}^d$ are
\[
\partial_a \doteq \left\{ z \in {\mathbb Z}^d,~ z(j) =0~ \forall j \in a \right\}.
\]
\index{p@$\partial_a$ The coordinate hyperplanes}
 We will use the letter $\sigma$ to denote hitting times to these sets;
for any process ${\mathcal P}$ on ${\mathbb Z}^d$ define
\begin{equation}\label{e:defsigma}
\sigma^{\mathcal P}_a \doteq \inf \{k: {\mathcal P}_k \in \partial_a\};
\end{equation}
\index{s@$\sigma_a$ first hitting times to coordinate hyperplanes}
in what follows the process ${\mathcal P}$ will always be clear from context
and we will omit the superscript of $\sigma$.
If $a = \{j\}$ for some $j \in {\mathcal N}_+$,
we will write $\sigma_j$ rather than $\sigma_{\{j\}}$; the same convention applies
to $\partial_{\{j\}}.$

We will denote the domain of a process ${\mathcal P}$ 
by $\Omega_{\mathcal P}$; $\partial \Omega_{\mathcal P}$ will denote its
constraining boundary if it has one; $\Omega_{\mathcal P}^o$ will denote
$\Omega_{\mathcal P} - \partial \Omega_{\mathcal P}$.
\index{p@$\partial \Omega_{\mathcal P}$ constraining boundary of a process ${\mathcal P}$}

We will often express constraints of constrained random walks using the 
constraining maps $\pi_a$, $a \subset {\mathcal N}_+$,
defined as follows:
\[
\pi_a(x,v) = \begin{cases}
		x + v, &\text{ if $x(j) + v(j) \ge 0~~ \forall j \in a^c$},\\
		x,     &\text{ otherwise,}
		\end{cases}
\]
where $x \in {\mathbb Z}^d$ and $v \in {\mathbb Z}^d$.
If $a = \emptyset$, we will write $\pi$ instead
of $\pi_\emptyset$; $\pi$ constrains to ${\mathbb Z}_+^d$ 
any process to which it is applied (see the definition of $X$ and $Y$ below to see
how $\pi_a$ is used to define constrained random walks).
Other than $a=\emptyset$ this paper will only use $\pi_{\{i\}}$, 
$i \in {\mathcal N}_+$
and for most of the paper we will assume $i=1$. To ease notation we will write
$\pi_i$ rather than $\pi_{\{i\}}$; $\pi_i$ constrains a given process to be positive
in its ${\mathcal N}_+ - \{i\}$ coordinates.

If ${\mathcal P}$ is a random walk with increments ${\mathcal V}({\mathcal P})$,
\index{${\mathcal V}({\mathcal P})$ the set of increments of a random walk ${\mathcal P}$}
constrained to
stay in $\Omega_\mathcal P \subset{\mathbb Z}^d$,
and $S \subset \Omega_{\mathcal P}$, define
\begin{equation} \label{e:defintbound}
S^o \doteq \{ s \in S: 
\Omega_\mathcal P \cap ( s + {\mathcal V}({\mathcal P}) ) \subset S \},
\partial S \doteq S- S^o.
\end{equation}
\index{p@$\partial S$ exit boundary of a given set $S$}

The notation doesn't state explicitly the ${\mathcal P}$-dependence
of these terms; but
whenever we use them below,
the underlying process ${\mathcal P}$ will always be clear from context.
In what follows ${\mathcal P}$ will be either $X$, $Y^n$,
$Y$ or $Z$, all of which are defined below.

We want to compute certain probabilities/ expectations
associated with a constrained random walk.
This will involve three transformations of the original process:
an affine change of variables, taking a limit (this will drop one
of the boundary constraints of the process) and removing all constraints
which makes the process an ordinary random walk with independent
and identically distributed (iid)  increments.
We will show the original process with $X$, 
the result of the affine transformation with $Y^n$, 
the limit process with $Y$
and the completely unconstrained process with $Z$.

$X$ will denote a constrained random walk on
\index{$X$}
$\Omega_X \doteq {\mathbb Z}_+^d$ with independent increments
\index{O@$\Omega_X$}
$I_k \in \{e_i - e_j, i \neq j \in {\mathcal N}_0\}$
\index{I@$I_k$ increments of $X$}
where $e_0$ is the zero
vector in ${\mathbb Z}^d$ and $e_i \in {\mathbb Z}^d$, $ i \neq 0$ is the
\index{$e_i$ the unit vectors and the zero vector in ${\mathbb Z}^d$}
unit vector in the direction $i$. To keep $X$ in its domain, the increments are
constrained on the boundaries of ${\mathbb Z}_+^d$:
\begin{align}\label{d:constX}
X_0 &=x \in {\mathbb Z}_+^d,\notag \\
X_{k+1} &\doteq X_k + \pi ( X_k,I_k )\\
\pi(x,e_i-e_j) &\doteq \begin{cases} e_i - e_j, &\text{ if } x(j) > 0 \\
				  0      , &\text{otherwise.}\notag
\end{cases}, 
\end{align}
where, by convention, ``$x(0)$'' means ``$1$'' (or some other positive quantity).
The constraining boundary of $X$ is 
$\partial \Omega_X \doteq \Omega_X \cap 
\left( \cup_{j \in {\mathcal N}_+} \partial_j\right).$
\index{p@$\partial \Omega_X$}

We denote the common distribution of the increments $I_k$ by the matrix $p$,
\index{p@$p$ the distribution of the iid increments of $X$}
i.e., $p(i,j)$ is the probability that $I_k$ equals $e_i -e_j$, $i\neq j \in {\mathcal N}_0$;
$p(i,i)=0$ for $i\in{\mathcal N}_0$.
${\mathcal V}(X)$ will denote 
the set of increments of $I_k$ with nonzero probability:
\index{${\mathcal V(X)}$ the increments of $X$}
\[
{\mathcal V}(X) \doteq \{e_i-e_j: p(i,j) > 0 \}.
\]
\begin{remark}
\label{r:conventionpv}
{\em
For $v =e_i -e_j$, $i,j \in {\mathcal N}_0 \times {\mathcal N}_0$, 
``$p(v)$'' will denote $p(i,j)$.
}
\end{remark}

For our probability space we will take ${\mathcal V}(X)^{\mathbb N}$.
$\{I_k\}$ are the coordinate maps on 
${\mathcal V}(X)^{\mathbb N}$, ${\mathscr F}$ is the
 $\sigma$-algebra
generated by $\{I_k\}$ and $P$ is the product measure on 
${\mathcal V}(X)^{\mathbb N}$
under which $I_k$ are an
iid sequence with common distribution $p$.

$X$ describes the dynamics of 
the number of customers in the queues of a Jackson network, i.e., a queueing system
with exponentially and identically distributed and independent 
interarrival and service times.
In this interpretation, $0 \in {\mathcal N}_0$ represents the outside of the queueing
system and for $ i \in {\mathcal N}_+$,
$X_k(i)$ is
the number of customers in queue $i$ at the $k^{th}$ jump
of the system (an arrival or a service completion).
$X$ is on the boundary $\partial_i$
 when the $i^{th}$
queue is empty and the constrained dynamics
on the boundary
means that the server at node $i$ cannot serve when its queue is empty.
The increment $e_i -e_j$, $i,j \neq 0$, $i \neq j$, represents a customer leaving
queue $i$ after service completion at node $i$ and joining queue $j$; 
$e_i$, $i \neq 0$, represents an arrival from outside to queue $i$ and
$-e_j$ a customer leaving the system after a server completion at queue $j$.
We will assume that the Markov chain defined by the matrix $p$ on ${\mathcal N}_0$
is irreducible. 
One can also represent $p$ as arrival, service and routing probabilities:
\begin{equation}\label{d:lambdamu}
\lambda_j \doteq p(0,j),~~ \mu_j \doteq \sum_{j \in {\mathcal N}_0 } p(j,k),
 j \in {\mathcal N}_+, ~~
r(i,j) \doteq  p(i,j)/\mu_i.
\end{equation}
\index{l@$\lambda$ arrival rates}
\index{m@$\mu$ service rates}

The irreducibility of $p$ implies that 
\begin{equation}\label{d:defnu}
\nu = \lambda + r \nu
\end{equation}
\index{n@$\nu$ total arrival rates}
has a unique solution. $\nu_i$ is the total arrival rate to node $i$ when
system is in equilibrium. We assume that the network corresponding to $X$ is stable,
i.e,
\begin{equation}\label{e:stable}
\rho_i \doteq \frac{\nu_i}{\mu_i} < 1,~~ i \in {\mathcal N}_+.
\end{equation}
\index{r@$\rho$ utilization rates}

We will pay particular attention to tandem networks, i.e., a number of queues in tandem;
these are Jackson networks whose
$p$ matrix is of the form
\begin{equation}\label{e:poftandem}
p(0,1)  > 0, p(0,j) = 0, j \neq 1,
p(d,0) = \mu_d, p(j,j+1) = \mu_j, j \in {\mathcal N}_0 - \{0,d\};
\end{equation}
for tandem queues $\lambda$ will denote the only nonzero arrival rate $p(0,1)$;
then
$\nu_i = p(0,1)= \lambda$  for all $ i \in {\mathcal N}_+$.
We will call the random walk $X$ a {\em tandem walk} if it represents
the queue lengths of a tandem network.
\index{tandem networks and tandem walks}

The domain $A_n$ is defined as in \eqref{e:defA}.
We will assume 
\begin{equation}\label{e:someonecomes}
 \lambda_j > 0 \text{ at least for some } j \in {\mathcal N}_+.
\end{equation}
With this and \eqref{e:defintbound}, $\partial A_n$ indeed equals the right side
of \eqref{e:defBoundaryA}. Define the
stopping times
\[
\tau_n \doteq \{k: X_k \in \partial A_n \}
\]
\index{t@$\tau_n$}
and 
\[
p_n \doteq P( \tau_n < \tau_0).
\]
\index{pn@$p_n$ the overflow probability of interest}
If \eqref{e:someonecomes} fails, $p_n$ becomes trivial.

Define the input/output ratio of the system as
\begin{equation}\label{e:defr}
r \doteq
\frac{
\sum_{j \in {\mathcal N}_+ } p(0,j)
}{
\sum_{j \in {\mathcal N}_+} p(j,0)}.
\end{equation}
\index{r@$r$ the input/output ratio}

Stability of $X$ implies 
\begin{proposition}\label{p:unstableZ}
\begin{equation}\label{e:unstableZ}
r < 1.
\end{equation}
\end{proposition}
\begin{proof}
By definition
\[
p(0,j) = 
\nu_j - 
\sum_{k \in \{ j\}^c } 
r(k,j) \nu_k,
\]
for $j \in {\mathcal N}_+$.
Sum both sides  over $j$:
\begin{align*}
\sum_{j \in {\mathcal N}_+ } p(0,j) 
&= \sum_{j \in {\mathcal N}_+ } \nu_j( 1 - \sum_{k \in \{j\}^c} r(j,k))\\
&= \sum_{j \in {\mathcal N}_+ } \nu_j\frac{ p(j,0)}{\mu_j}
= \sum_{j \in {\mathcal N}_+ } \rho_j p(j,0).
\end{align*}
Then
\[
r= \frac{
\sum_{j \in {\mathcal N}_+ } p(0,j)
}{
\sum_{j \in {\mathcal N}_+} p(j,0)}  = \sum_{j \in {\mathcal N}_+} \rho_j 
\frac{p(j,0)}{ \sum_{j \in {\mathcal N}_+} p(j,0) }
\]
which is an average of the utilization rates $\rho_i$ which are 
by assumption all less than $1$; \eqref{e:unstableZ} follows.
\end{proof}

$T_n$ and $Y^n$ are defined as in \eqref{e:defTn}. $T_n$ depends on $i$;
when we need to make this dependence explicit
we will write $T^i_n$; for most of the analysis of the paper $i$ will be fixed
and therefore can be assumed constant, and, unless otherwise noted, in the following
sections we will take $i=1$.
\index{T@$T_n^i$ the affine transformation mapping $X$ to $Y^n$}

Define
\begin{equation}\label{e:DefIi}
{\mathcal I}_i \in {\mathbb R}^{d \times d},~~
{\mathcal I}_i(j,k) = 0, j \neq k,~~ {\mathcal I}_i(j,j) = 1, i \neq j,~~
{\mathcal I}_i(i,i) = -1.
\end{equation}
${\mathcal I}_i$ is the identity operator except that its $i^{th}$ diagonal term
is $-1$ rather than $1$. Then
\begin{equation}\label{e:explicitTn}
T_n = n e_i + {\mathcal I}_i.
\end{equation}
Define the sequence of transformed increments
\begin{equation}\label{e:defJ}
J_k \doteq {\mathcal I}_i (I_k);
\end{equation}
\index{J@$J_k$ increments of $Y$}
$J_k$ and $I_k$ take the same values except that $J_k = e_j + e_i$ whenever
$I_k = e_j -e_i$ and $J_k = -e_i$ whenever $I_k = e_i$.
Define 
\begin{align*}
{\mathcal V}(Y) \doteq \{ {\mathcal I}_i v, v \in {\mathcal V}(X) \}.
\end{align*}
\index{V@${\mathcal V}(Y)$ set of possible increments of $Y$}
\begin{remark}\label{r:simplenotationforp}{\em
For $v \in {\mathcal V}(Y)$ we will shorten $p({\mathcal I}_i v)$  to $p(v)$
(remember, per Remark \ref{r:conventionpv}, for $e_i-e_j$, $i,j \in {\mathcal N}_0$,
$p(v)$ denotes $p(i,j)$).
}
\end{remark}
The limit unstable constrained process $Y$ is
\begin{equation}\label{e:defY}
Y_0 \doteq y \in{\mathbb Z}_+^d,~~ Y_{k+1} \doteq Y_k + \pi_i(Y_k, J_k).
\end{equation}
\index{Y@$Y$ definition of the limit process $Y$}
$Y$ has the same dynamics as $Y^n$ except that $Y$ has no constraining
boundary on its $i^{th}$ coordinate; therefore its state space is
$\Omega_Y \doteq {\mathbb Z}_+^{i-1} \times {\mathbb Z} \times {\mathbb Z}_+^{d-i}.$
The domain $B$ for $Y$ is defined as in \eqref{e:defB} and its boundary
$\partial B$ is defined from $B$ using \eqref{e:defintbound} and coincides
with the right side of \eqref{e:defBb}. Let $\tau$ be as in \eqref{e:deftau}.
Define
\begin{equation}\label{e:defofSigman}
\zeta_n \doteq \inf\left\{k: Y_k(i) = \sum_{j\neq i } Y_k(j) + n\right\};
\end{equation}
note $\tau = \zeta_0.$

$X$, $Y^n$ and $Y$ are all defined on the same probability space 
$\left({\mathcal V}(X)^{\mathbb N}, {\mathscr F}, P\right)$; the measure
$P$ and their initial positions determine their distributions.
We will use a subscript on $P$
to denote the initial positions, e.g.,
$P_x( \tau_n < \tau_0)$ is the same as $P(\tau_n < \tau_0)$ 
with $X_0 = x$ and $P_y( \tau < \infty)$ means $P( \tau < \infty)$ with
$Y_0= y$.

We note a basic fact about $Y$ here:
\begin{proposition}\label{p:cantwanderforever}
For $y \in {\mathbb Z}_+^d$, $\sum_{i=1}^d y(i) < n$
\begin{equation}\label{e:toprove0}
P_y( \zeta_n \wedge \zeta_0 = \infty) = 0.
\end{equation}
\end{proposition}
\begin{proof}
Set
\[
c = \sum_{j \in {\mathcal N}_0} p(0,j) > 0.
\]
For $y \in {\mathbb Z}_+^d$, $y:\sum_{i=1}^d y(i) < n$, 
\begin{equation}\label{e:bound0}
P_y ( \zeta_n \wedge \zeta_0 \le n) > c' \doteq c^n > 0
\end{equation}
because, at least the sample paths whose increments consist only
of $\{-e_i, i \in {\mathcal N}_+, p(0,i)  > 0 \}$ push $Y$ to $\partial B$
in $n$ steps and the probability of this event is $c^n$.

$\hat{Y}_k \doteq Y_{nk \wedge \zeta_n \wedge \zeta_0}$
is Markov on 
$\hat{B}_n \doteq \left\{ y \in {\mathbb Z}_+^d, 0 \le \sum_{i=1}^d y(i)  \le n 
\right\}$
(because $Y$ is Markov). 
The boundary of the last set is
\[
\partial \hat{B}_n =  \left\{ y \in {\mathbb Z}_+^d, 0 =  \sum_{i=1}^d y(i)  \text{ or }
  \sum_{i=1}^d y(i) = n \right\}.
\]
By definition
\begin{equation}\label{e:bound1s2}
P_y( \zeta_0 \wedge \zeta_n = \infty ) \le P_y( \hat{Y}_k \in \hat{B}_n - \partial \hat{B}_n).
\end{equation}
The bound \eqref{e:bound0} implies
$P_y\left( \hat{Y}_1  \in \hat{B}_n - \partial \hat{B}_n \right) \le 1-c'.$
This and that $\hat{Y}$ is Markov give\\
$P_y\left( \hat{Y}_k  \in \hat{B}_n - \partial \hat{B}_n \right) \le (1-c')^k.$
This and \eqref{e:bound1s2} imply
\begin{equation}\label{e:bound1}
P_y( \zeta_0 \wedge \zeta_n = \infty) \le (1-c')^k.
\end{equation}
Letting $k\rightarrow \infty$ gives \eqref{e:toprove0}.
\end{proof}

\section{Convergence - initial condition set for $Y$ }\label{s:convergence1}
This section shows that the affine transformation of observing
the process from the exit boundary really 
gives approximations of the exit probabilities
we seek to compute. 
The present convergence 
result specifies the initial point for the $Y$ process.
This allows a simple argument that works for general
stable $X$ and uses LD results only roughly to prove
that certain probabilities decay to $0$. 
In Section \ref{s:convergence2} we will prove a second convergence result (for the case of
two dimensional tandem walk) where the initial
point is given for the $X$ process; this will require a
finer use of large deviations decay rates.

Denote by ${\mathcal X}$ the law of large numbers limit of $X$ , i.e., the deterministic
function which satisfies
\begin{equation}\label{e:lln}
\lim_n P_{x_n}\left ( \sup_{k \le t_0 n} |X_k/n -  {\mathcal X}_{k/n} | > \delta \right) = 0
\end{equation}
for any $\delta > 0$ and $t_0 > 0$
where  $x_n \in {\mathbb Z}_+^d$ is a sequence of initial positions
satisfying $\frac{x_n}{n} \rightarrow \chi \in {\mathbb R}_+^d$ 
(see, e.g., \cite[Proposition 9.5]{robert2003stochastic} or 
\cite[Theorem 7.23]{chen2013fundamentals}).
The limit process starts from
${\mathcal X}_0 = \chi$,  is piecewise affine and takes values in ${\mathbb R}_+^d$;
then
 $s_t \doteq \sum_{i=1}^d {\mathcal X}_t(i)$ starts from $\sum_i \chi(i)$
is also piecewise linear and continuous (and therefore differentiable except
for a finite number of points) with 
values in ${\mathbb R}_+$. The stability and bounded iid increments of
$X$ imply that $s$ is strictly decreasing 
and 
\begin{equation}\label{e:sdecrease}
 c_1 > -\dot{s} > c_0 > 0
\end{equation}
for two constants $c_1$ and $c_0$.
These imply that ${\mathcal X}$ goes in finite time  $t_1$ to
$ 0 \in {\mathbb R}_+^d$ and remains there afterward.

Fix an initial point $y \in \Omega_Y$ for the process $Y$ and set
$x_n = T_n(y)$; \eqref{e:explicitTn} implies
\begin{equation}\label{e:llninit}
\frac{x_n}{n} \rightarrow e_i.
\end{equation}

\begin{proposition}\label{t:c1}
Let $y$ and $x_n$ be as above. Then
\[
\lim_{n\rightarrow \infty} P_{x_n}(\tau_n < \tau_0) =  P_y(\tau < \infty ).
\]
\end{proposition}
\begin{proof}
Note that for $n > y(i)$, $x_n \in A_n$.
Define
\[
M_k = \max_{l \le k } Y_l(i), ~~M^X_k = \min_{l \le k } X_l(i).
\]
$M$ is an increasing process and $M_\tau$ is the greatest that the $i^{th}$
component of $Y$ gets before hitting $\partial B$ (if this happens in finite time).
The monotone convergence theorem implies
\[
P_y( \tau < \infty) = \lim_{n\nearrow \infty} 
P_y( \tau < \infty, M_\tau <  n ).
\]
Thus 
\begin{equation}\label{e:decompose1}
P_y( \tau < \infty) = 
P_y( \tau < \infty, M_\tau < n )
+
P_y( \tau < \infty, M_\tau \ge n )
\end{equation}
and the second term goes to $0$ with $n$.
Decompose $P_{x_n}( \tau_n < \tau_0)$ similarly using $M^X$:
\begin{align*}
P_{x_n}( \tau_n < \tau_0 ) &= 
P_{x_n}\left( \tau_n < \tau_0, M^X_{\tau_n}  > 0  \right)
+
P_{x_n}\left( \tau_n < \tau_0, M^X_{\tau_n}  = 0 \right).\\
\intertext{ 
On the set $\{M^X_{\tau_n} >0 \}$, the process $X$ cannot reach the boundary
$\partial_i$ before $\tau_n$, therefore over this set 
1) the events $\{\tau_n < \tau_0\}$ and
$\{\tau < \infty\}$ coincide  (remember that $X$ and $Y$ are defined
on the same probability space)
2) the distribution of $(T_n(X),n-M^X)$ is the same as that of
$(Y,M)$ upto time $\tau_n.$
Therefore,}
&= P_y( \tau < \infty, M_\tau < n) + 
P_{x_n}\left( \tau_n < \tau_0, M^X_{\tau_n} =  0 \right).
\end{align*}
The first term on the right equals the first term on the right side of 
\eqref{e:decompose1}. We know that the second term in \eqref{e:decompose1}
goes to $0$ with $n$. Then  to finish our proof, it suffices to show
\begin{equation}\label{e:toproveconv1}
\lim_n 
P_{x_n}\left( \tau_n < \tau_0, M^X_{\tau_n} = 0\right)  = 0.
\end{equation}
$M^X_{\tau_n} = 0$ means that $X$ has hit $\partial_i$ before $\tau_n$.
Then the last probability equals
\begin{equation}\label{e:lastbit}
 P_{x_n}\left (  \sigma_i < \tau_n < \tau_0\right ),
\end{equation}
which, we will now argue, goes to $0$ ($\sigma_i$ is the first time $X$
hits $\partial_i$; see \eqref{e:defsigma}).
\eqref{e:llninit} implies ${\mathcal X}_0 = e_i.$ 
Define $t^i \doteq \inf\{t: {\mathcal X}_t(i)  = 0 \}$
and $t^0 \doteq \inf\{t: {\mathcal X}_t = 0 \in {\mathbb R}^d\}.$
By definition
$t^i \le t^0 < \infty$
Now choose $t_0$ in \eqref{e:lln} to be equal to $t^0$, define
${\mathcal C}_n \doteq
\left\{ \sup_{k \le t^0 n} \in |X_k/n -  {\mathcal X}_{k/n} | > \delta \right \}$
and partition \eqref{e:lastbit} with ${\mathcal C}_n$:
\begin{align}\label{e:decomposewithCn}
P_{x_n}\left (  \sigma_i < \tau_n < \tau_0\right) = 
P_{x_n}\left ( \{  \sigma_i < \tau_n < \tau_0 \} \cap  {\mathcal C}_n\right)
+P_{x_n}\left( \{  \sigma_i < \tau_n < \tau_0 \} \cap  {\mathcal C}_n^c\right).
\end{align}
The first of these goes to $0$ by \eqref{e:lln}.
The event in the second term is the following: $X$ remains at most $n\delta$ 
distance away $n {\mathcal X}$ until its $nt^0$ step, hits $\partial_i$
then  $\partial A_n$ and then $0$.
These and \eqref{e:sdecrease} imply that,
for $n$ large enough,
any sample path lying in this event can hit
$\partial A_n$ only after time $nt^0$. Thus, the  second probability on the
right side of \eqref{e:decomposewithCn} is bounded above by
\[
P_{x_n}(\{nt^0 < \tau_n < \tau_0\} \cap {\mathcal C}_n^c).
\]
The Markov property of $X$,
$ \{  \sigma_i < \tau_n < \tau_0 \} \subset
 \{   \tau_n < \tau_0 \}$
and \eqref{e:lln}
imply that the last probability is less than
\[
\sum_{x: |x| \le n\delta} P_x(   \tau_n < \tau_0 ) P_{x_n} (X_{nt^0}= x).
\]
For $|x| \le n \delta$, 
the probability $P_x( \tau_n < \tau_0)$ decays exponentially in $n$
\cite[Theorem 2.3]{GlassKou};
then, the above sum goes to $0$. This establishes \eqref{e:toproveconv1} 
and finishes the proof of the proposition.

\end{proof}

\section{Analysis of $Y$, $d=2$}\label{s:twodim}
Let us begin with several definitions for the general dimension $d$;
because we will almost exclusively work with the $Y$ process from here on,
we will shorten ${\mathcal V}(Y)$ to ${\mathcal V}$.
A function $V: {\mathbb Z}^d \rightarrow {\mathbb C}$ is said to be
a harmonic function of the process $Y$ (or {\em $Y$-harmonic}) 
on a set $O \subset \Omega_Y$ if
\begin{equation}\label{e:linear}
V(y) = {\mathbb E}_y \left[ V(Y_1) \right] = 
\sum_{v \in {\mathcal V} } V(y + \pi_i(y, v)) p(v), y \in O,
\end{equation}
where we use the convention set in Remark \ref{r:simplenotationforp}.
Throughout the paper $O$ will be either $B^o$ or $\Omega_Y$, and the choice
will always be clear from context; for this reason we
will often write ``...is $Y$-harmonic'' 
without specifying the set $O$.
$V$ is said to be $Z$-harmonic on $O \subset {\mathbb Z}^d$ if
\begin{equation*}
V(z) = {\mathbb E}_z \left[ V(Z_1) \right] = 
\sum_{v \in {\mathcal V} } V(z + v)) p(v), z \in O.
\end{equation*}
$Z$-harmonicity and $Y$-harmonicity coincide on $\Omega_Y^o$.

Above we have assumed the domain of $V$ to be ${\mathbb Z}^d.$
If $V$ is defined only on a subset $a \subset {\mathbb Z}^d$, it can be
trivially extended to all ${\mathbb Z}^d$ by setting it to $0$ on ${\mathbb Z}^d -a$.
Thus, the above definitions can be applied to any function
defined on any subset of ${\mathbb Z}^d$; we will use a similar convention for most
of the definitions below.

The Markov property of $Y$ implies that
\begin{equation}\label{e:defPartialCdet}
h:y \rightarrow {\mathbb E}_y \left[ 
f\left(Y_\tau\right) 1_{\{\tau < \infty\}}\right], y \in B,
\end{equation}
is a harmonic function of $Y$
whenever the right side is well defined for all $y \in B^o$. Note that
$h$ is the image of the function $f$ 
under the Balayage operator ${\mathcal T}_{(B^o)^c}$.
The dynamics of $Y$ and the definition of $B$ 
imply that ${\mathcal T}_{(B^o)^c}$ is a 
distribution on $\partial B$; for this reason 
we will call harmonic functions of the
form \eqref{e:defPartialCdet}
$\partial B$-determined. If a function $f$ is defined over a
domain larger than $B$, we will call $f$ $\partial B$-determined, if its
restriction to $B$ is so.

The analysis of the previous section suggests that we approximate
\[
P_x( \tau_n < \tau_0)
\]
with $W(T_n(x))$ where
\begin{equation}\label{e:simpler}
W(y) \doteq P_y ( \tau < \infty) = {\mathbb E}_y\left[ 1_{\{ \tau < \infty\}}\right]
\end{equation}
for any stable Jackson network $X$. 
$W$ is a $\partial B$ determined harmonic function of $Y$, in particular it
solves \eqref{e:linear} with $O= B^o$.
That
$W(y) = P_y( \tau < \infty) = 1$ for $ y \in \partial B$
implies that $W$ also satisfies the boundary condition
\begin{equation}\label{e:boundary}
V|_{\partial B} = 1.
\end{equation}

Then $W$ is a solution of (\ref{e:linear},\ref{e:boundary}) with $O=B^o.$
Large deviations analysis of $W$
is an asymptotic analysis
of the system (\ref{e:linear},\ref{e:boundary})
that scales $V$ to $-\frac{1}{n}\log V$ 
and uses a law of large numbers scaling for space and time.
With the $y$ coordinates, we no longer need to scale $V$, time or space
and can
directly attempt to solve  (\ref{e:linear},\ref{e:boundary})- perhaps approximately.
We have assumed that $X$ is stable; 
this implies that $Y_{\tau \wedge k}$, $k=1,2,3,...$, is unstable and therefore,
the Martin boundary of this process has points at infinity. Then
one cannot expect all harmonic functions of $Y$ to be
$\partial B$-determined and
in particular the system (\ref{e:linear},\ref{e:boundary}) will not
have a unique solution; hence, once we find a 
solution of (\ref{e:linear}, \ref{e:boundary}) that we believe
(approximately) equal to $P_y(\tau < \infty)$, 
we will have to prove that it is $\partial B$-determined.

Define 
\[
B_Z \doteq 
\left\{ z \in {\mathbb Z}^d: z(1) \ge  \sum_{j=2}^d z(j)
\right\}.\]
The unconstrained version of \eqref{e:linear} is
\begin{equation}\label{e:linear1}
V(z) = {\mathbb E}_z \left[ V(Z_1) \right] = 
\sum_{v \in {\mathcal V} } V(z +v ) p(v), 
\end{equation}
$z \in O \subset {\mathbb Z}^d$
and that of \eqref{e:boundary} is
\begin{equation}\label{e:boundary1}
V|_{\partial B_Z} = 1.
\end{equation}
A function is said to be a harmonic function of the unconstrained random
walk $Z$ on $O$
if it satisfies \eqref{e:linear1}.

Introduce also the boundary condition
\begin{equation}\label{e:boundary2constrained}
V|_{\partial B} = f,~~~ f: \partial B \rightarrow {\mathbb C},
\end{equation}
for $Y$ 
which generalizes \eqref{e:boundary}.

\noindent
Our idea to [approximately] solve
(\ref{e:linear},\ref{e:boundary}) is this:
\noindent
\begin{enumerate}
\item Construct a class ${\mathcal F}_Y$ 
of ``simple'' harmonic functions for the process $Y$
(i.e., a class of solutions to \eqref{e:linear})
For this
	\begin{enumerate}
	\item
	Construct a class ${\mathcal F}_Z$ of harmonic functions for the
	unconstrained process $Z$ 
(i.e., a class of solutions to \eqref{e:linear1} with $O= {\mathbb Z}^d$),
	\item Use linear combinations of elements of ${\mathcal F}_Z$ to find
		solutions to \eqref{e:linear}.
	\end{enumerate}
	\item 
Represent [or approximate] the boundary condition \eqref{e:boundary}
		by linear combinations of the boundary values of the 
		$\partial B$-determined members of the class ${\mathcal F}_Y$.
\end{enumerate}
The definition of the class ${\mathcal F}_Z$ is given in
\eqref{d:classFZ} and that of ${\mathcal F}_Y$ is given in \eqref{d:classFY}.

This section treats the case of two dimensions,
where this program yields, for a wide class of processes,
approximate solutions of \eqref{e:linear} not
just with $f=1$, i.e, the boundary condition \eqref{e:boundary},
but with any bounded $f$, i.e., the boundary condition 
 \eqref{e:boundary2constrained}.

Stability of $X$ implies
$p(2,0) \wedge p(1,0) > 0$ and we will assume 
\begin{equation}\label{e:assumptionp20}
p(2,0) > 0;
\end{equation}
otherwise one can switch the labels of the nodes to call $2$
the node for which $p(i,0) > 0$.

\subsection{The characteristic polynomial and surface}\label{ss:charpolsurf}
Let us call
\begin{equation}\label{e:hamiltonian}
{\mathbf p}(\beta,\alpha)
\doteq \beta\alpha \sum_{ v\in {\mathcal V} } p(v) \beta^{v(1)-v(2)}
\alpha^{v(2)}, (\beta,\alpha) \in {\mathbb C}^2,
\end{equation}
the {\em characteristic polynomial} of the process $Z$ for
$B_Z$,
\begin{equation}\label{e:chareqZ}
{\mathbf p}(\beta,\alpha) -\beta\alpha = 0
\end{equation}
the {\em characteristic equation} of $Z$ for $B_Z$ and 
\[
{\mathcal H} \doteq \{ (\beta,\alpha): 
{\mathbf p}(\beta,\alpha) -\beta\alpha =0  \}
\]
the {\em characteristic surface} of $Z$ for $B_Z$.  
We borrow the adjective ``characteristic''
from the classical theory of linear ordinary differential equations; the
development below parallels that theory. 
Figure \ref{f:charsurfexample} depicts the real section of the characteristic
surface of the walk whose $p$ matrix equals
\begin{equation}\label{e:pexample}
p = \left( \begin{matrix}
0  &  0.05 & 0.1 \\
0.35 &  0 &  0.12\\
 0.3 &  0.08 & 0
\end{matrix}\right).
\end{equation}

${\mathcal H}$ is an affine algebraic curve
of degree $3$ \cite[Definition 8.1, page 32]{griffiths}; its $d$ dimensional version
in Section \ref{s:dg2} will be an affine algebraic variety of degree $d+1$.
We will need, for the purposes of the present paper, only that points on these varieties
come in conjugate pairs (see below). 
A thorough study/ description of the geometry of
these varieties (and their projective counterparts) and
its implications for constrained random walks will have to be taken up in future work.

${\mathbf p}$ 
is a second order polynomial  in $\alpha$ [$\beta]$ with second and first order
coefficients in $\beta$ $[\alpha]$:
\begin{align}
{\mathbf p}(\beta,\alpha) &= ( p(1,0) \alpha + p(2,0)) \beta^2
+ ( p(1,2)\alpha^2 + p(2,1) - \alpha) \beta \label{e:polinbeta}
+ (p(0,2) \alpha^2 + p(0,1) \alpha) \\
~\notag\\
{\mathbf p}(\beta,\alpha) &= ( p(0,2) + \beta p(1,2) ) \alpha^2
+ ( p(0,1) + p(1,0)\beta^2 - \beta) \alpha
+ ( p(2,0) \beta^2 + p(2,1)\beta). \label{e:polinalpha}
\end{align}

\paragraph{A singularity}
For $ \beta p(1,2) + p(0,2) = 0$
\eqref{e:polinalpha} becomes  affine. If $p(1,2) = p(0,2) = 0$, \eqref{e:polinalpha}
is affine for all values of $\beta$ and the method developed below is not applicable.
But such walks are essentially one dimensional
(only their
first component can freely move and their second component decreases
to $0$
and stay there upon hitting it)
and yield to simpler methods.
In what follows the $\beta$ we will work with will always satisfy
\[
\beta p(1,2) + p(0,2) \neq 0.
\]

\ninsepsc{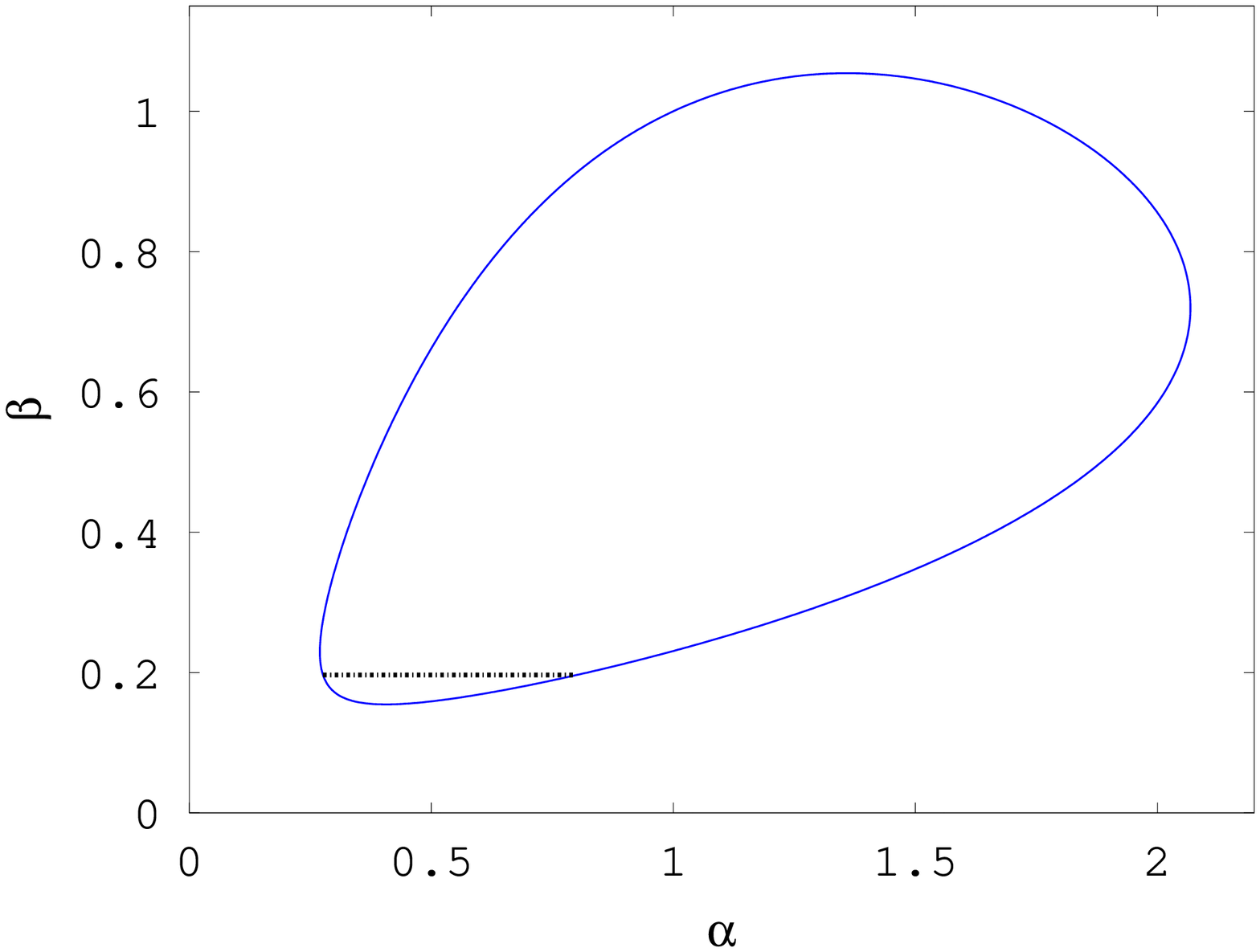}{The real section of the characteristic surface ${\mathcal H}$ 
of the walk defined by $p$
of \eqref{e:pexample}; the end points of the dashed line are two conjugate
points, see \eqref{e:conjugacy}}{0.4}

\subsection{$\log$-linear harmonic functions of $Z$}
The $Z$-version of the random times $\zeta_n$ of \eqref{e:defofSigman} and $\tau$ of
\eqref{e:deftau} are
\begin{align*}
\tau^Z &\doteq  \inf \left\{ k: Z_k \in \partial B_Z \right\}\\
\zeta^Z_n &\doteq \inf\left\{k: Z_k(i) = \sum_{j\neq i } Z_k(j) + n\right\}.
\end{align*}
We will omit the $Z$ superscript below
because the underlying process will always be clear from context.

For $\tau < \infty$,
$Z_\tau$
takes values in $\partial B_Z$ and $Z_\tau(1) = Z_\tau(2)$. Therefore, the
distribution of $Z_\tau$ on $\partial B_Z$ 
is equivalent to the distribution of $Z_\tau(1)$ on ${\mathbb Z}$
whose characteristic function is
\[
\theta \rightarrow {\mathbb E}_z 
\left[ e^{i \theta Z_\tau(1)} 1_{\{\tau < \infty \}}
\right], \theta \in {\mathbb R}.
\]

That $Z_\tau(1)$ is integer valued 
makes the above characteristic function
periodic with period $2\pi$ therefore we can
restrict $\theta \in [0,2\pi)$;
setting $\alpha = e^{i\theta}$ we rewrite the last display as
\begin{equation}\label{e:charfunZ}
\alpha \rightarrow {\mathbb E}_z 
\left[ \alpha^{ Z_\tau(1)} 1_{\{ \tau < \infty\}}
\right], 
\alpha \in S^1,
\end{equation}
where $S^1 \doteq \{ u \in {\mathbb C}: |u| = 1\}$
is the unit circle in ${\mathbb C}$.
For each fixed $\alpha \in  S^1$
the right side of \eqref{e:charfunZ} defines a harmonic function of  the process
$Z$ on $B_Z^o$ as $z$ varies in this set. Our collection of harmonic functions
${\mathcal F}_Z$
for the process $Z$
will consist of these and its generalizations when we allow $\alpha$ to vary
in ${\mathbb C}$.
For $\alpha \in {\mathbb C}$
the function $z \rightarrow \alpha^{z(1)}$
 is an eigenfunction
of the translation operator on ${\mathbb Z}^2$ and $Z$ is a random walk on the
same group. These imply
\begin{proposition}\label{p:structureofeiZ}
Suppose
\[
 {\mathbb E}_{(1,0)}
\left[ |\alpha|^{Z_\tau(1)} 1_{\{\tau < \infty\}} \right]  < \infty
\]
for $\alpha \in {\mathbb C}$. Then
\begin{equation}\label{e:explicitZ}
{\mathbb E}_z \left[ \alpha^{Z_1(\tau)} 1_{\{\tau < \infty\}} \right] 
= U^{z(1)-z(2)} \alpha^{z(2)}
\end{equation}
for $z \in C$
where
\[
U \doteq {\mathbb E}_{(1,0)}
\left[ \alpha^{Z_\tau(1)} 1_{\{\tau < \infty\}}\right].
\]
Furthermore,
$(U,\alpha)$ is on the characteristic
surface ${\mathcal H}$. 
\end{proposition}

\begin{proof}
The proof will be by induction on $z(1)-z(2)$. 
\eqref{e:explicitZ} is true by definition for $z(1) -z(2) = 0$. 
Assume now that \eqref{e:explicitZ} holds for $z(1) - z(2) = k-1 \ge 0$
and fix  $z$ with $z(1) - z(2) = k$. 
The invariance of $Z$ under translations implies
\begin{equation}\label{e:distZsigma}
P_{z} ( Z_{\zeta_{k-1}} = z+(j-1,j), \zeta_{k-1} < \infty) = 
P_{(1,0)}( Z_\tau = (j,j) ,\tau < \infty)
\end{equation}
for $j \in {\mathbb Z}$.
The strong Markov property of $Z$ and
$\zeta_{k-1} < \tau$ imply
\begin{align*}
{\mathbb E}_z\left[ \alpha^{Z_\tau(2)} 1_{\{\tau < \infty\}}\right]
 &= 
{\mathbb E}_z\left[ 1_{ \{ \zeta_{k-1} < \infty \} } 
{\mathbb E}_z\left[\alpha^{Z_\tau(2)} 1_{\{\tau < \infty\}} | {\mathscr F}_{\zeta_{k-1}}
\right]\right]
\\
&= 
{\mathbb E}_z\left[ 1_{ \{ \zeta_{k-1} < \infty \} } 
{\mathbb E}_z\left[\alpha^{Z_\tau(2)} 1_{\{\tau < \infty\}} | {Z}_{\zeta_{k-1}}\right]\right]
\intertext{The random variable $Z_{\zeta_{k-1}}$ is discrete;
then, one can write the last expectation explicitly as the sum}
&=
\sum_{j=-\infty}^\infty
{\mathbb E}_{z+(j-1,j)} \left[1_{\{\tau< \infty\} }\alpha^{Z_\tau(2)}\right] 
P_z\left(Z_{\zeta_{k-1}} = z + (j-1,j) \right).
\intertext{ $z' = z + (j-1,j)$ satisfies $z'(1)-z'(2) = k-1$; this,
the induction hypothesis and \eqref{e:distZsigma}
give}
&= \sum_{j=-\infty}^\infty
U^{z(1) + j -1 - z(2)} \alpha^{j+z(2)} 
P_{(1,0)}\left( Z_\tau = (j,j), 1_{\{\tau< \infty\}}\right)\\
&=
U^{z(1) -1 - z(2)} \alpha^{z(2)} 
\sum_{j=-\infty}^\infty \alpha^{j}
P_{(1,0)}\left( Z_\tau = (j,j), 1_{\{\tau < \infty\}} \right)
\intertext{By definition, the last sum equals ${\mathbb E}_{(1,0)}
\left[\alpha^{Z_\tau(1)}1_{\{\tau< \infty\}}\right] = U$ and therefore}
&= 
U^{z(1) -1 - z(2)} \alpha^{z(2)} U = U^{z(1)-z(2)} \alpha^{z(2)},
\end{align*}
i.e., \eqref{e:explicitZ} holds also for $z$ with $z(1) - z(2) = k$.
This finishes the induction and the proof of the first part of the proposition.

The Markov property of $Z$ and the first part of the proposition
imply that $g:
z \rightarrow U^{z(1) -z(2)} \alpha^{z(2)}$ is a harmonic function of $Z$
on $B_Z^o$, i.e., it satisfies \eqref{e:linear1}. 
Substituting $g$ in 
\eqref{e:linear1} implies that 
$(U,\alpha)$ is on the characteristic
surface ${\mathcal H}$. 
\end{proof}

Conversely, any point on ${\mathcal H}$ defines
a harmonic function of $Z$:
\begin{proposition}
For any $(\beta,\alpha) \in {\mathcal H}$,
$z \rightarrow \beta^{z(1) - z(2)}\alpha^{z(2)}$, $z \in B_Z$,
is a harmonic function of $Z$.
\end{proposition}
\begin{proof}
Condition $Z$ on its first step and use ${\mathbf p}(\beta,\alpha) = \beta\alpha$.
\end{proof}

For $(\beta,\alpha) \in {\mathbb C}^2$, define 
\[
[(\beta,\alpha),\cdot]: {\mathbb Z}^2 \rightarrow {\mathbb C},~~
[(\beta,\alpha),z ]\doteq \beta^{z(1)-z(2)} \alpha^{z(2)} 
\]
The last proposition gives us the class of harmonic functions 
\begin{equation}\label{d:classFZ}
{\mathcal F}_Z \doteq 
\left\{ [(\beta,\alpha),\cdot],~~
(\beta,\alpha) \in {\mathcal H}
\right\}
\end{equation}
for $Z$.

\subsection{$\partial B_Z$-determined harmonic functions of $Z$}
A harmonic function $h$ of $Z$ on $B_Z^o$
is said to be $\partial B_Z$-determined if 
\[
h(z) = {\mathbb E}_z \left[ f\left(Z_\tau\right) 1_{\{\tau < \infty\}}\right],
z \in B_Z^o,
\]
for some $f:\partial B_Z\rightarrow {\mathbb C}$ 
for which the right side is well defined for all $z \in B_Z^o$. 
The above display defines  the Balayage operator ${\mathcal T}_{(B_Z^o)^c}$
of $Z$ on $\partial B_Z$, mapping $f$ to $h$; thus, $h$ is $\partial B_Z$-determined
if and only if it is the image of a function $f$ under the Balayage operator
${\mathcal T}_{(B_Z^o)^c}$.
In our analysis of $Y$ and its harmonic functions
we will find it useful to be able to differentiate between harmonic functions
of $Z$ which are $\partial B_Z$-determined, and those which are not.
The reader can skip this subsection for now and can return to it when we
refer to its results in subsection \ref{ss:partialBdet}.

For each
$\alpha \in {\mathbb C}$ satisfying
\begin{equation}\label{e:condalpha0}
\alpha p(1,0) + p(2,0) \neq 0
\end{equation}
\eqref{e:chareqZ} is 
a second order polynomial 
equation in $\beta$ (see \eqref{e:polinbeta})
with roots
\begin{align}
\beta_1 \doteq 
\frac{ \alpha - p(1,2) \alpha^2 -p(2,1) - \sqrt{\Delta}}{2 (p(2,0) + p(1,0)\alpha)},~~
\beta_2 
\doteq \frac{ \alpha - p(1,2) \alpha^2 -p(2,1) + \sqrt{\Delta}}{2 (p(2,0) + p(1,0)\alpha)},
\label{e:formulaforbeta}
\end{align}
where
\[
\Delta(\alpha) \doteq \left(p(1,2)\alpha^2 + p(2,1) - 1\right)^2 - 
4 (p(2,0) + p(1,0) \alpha)(p(0,1)\alpha + p(0,2) \alpha^2),
\]
 and $\sqrt{{\mathrm z}}$ denotes the complex number with nonnegative real part
whose square equals ${\mathrm z}$.

\begin{remark}
{\em
Unless otherwise noted, we will assume 
\eqref{e:condalpha0}. The stability condition 
\eqref{e:stable} rules out $p(1,0) = p(2,0) = 0$;
if $p(1,0) = 0$, \eqref{e:condalpha0} always holds.
If $p(1,0) \neq 0$ and $p(2,0) = 0$,  \eqref{e:condalpha0} fails
exactly when $\alpha$ equals $0$, a value which represents a trivial
situation (Balayage of the zero function on $\partial B_Z$). 
When $p(1,0), p(2,0) \neq 0$, \eqref{e:condalpha0} fails
only for $\alpha = \frac{-p(2,0)}{p(1,0)} < 0.$ This value may be of interest
to us in the next subsection and we comment on it there in
Remark \ref{r:annoyingremark}.
}
\end{remark}

Our next step is to identify a set of $\alpha$'s for which one of the roots above gives a
$\partial B_Z$-determined harmonic function.
In this, we will use
\cite[Exercise 2.12, Chapter 2, page 54]{revuz1984markov} (rewritten
for the present setup):
\begin{proposition}\label{p:niceresult}
For a function $f:\partial B_Z \rightarrow {\mathbb R}_+$
$z \rightarrow {\mathbb E}_z
\left[f(Z_\tau) 1_{\{\tau < \infty\}}\right]$ , $z \in B_Z$,
is the smallest function equal to $f$ on $\partial B_Z$ and harmonic on $B_Z^o$.
\end{proposition}

We begin with $\alpha=1$.
\begin{proposition}\label{p:balayageZalpha1}
\begin{equation}\label{e:casealpha1}
{\mathbb E}_z \left[ 1^{Z_\tau(2)} 1_{\{ \tau < \infty\}} \right]=
P_z ( \tau < \infty) = \beta_1(1)^{z(1)} = r^{z(1)},
\end{equation}
where $r$ is the input/output ratio \eqref{e:defr} of $X$ .
\end{proposition}
\begin{proof}
\begin{equation}\label{e:Delta1}
\Delta(1) = ( (p(1,0)  + p(2,0) ) - (p(0,1) + p(0,2)) ) ^2
\end{equation}
Proposition \ref{p:unstableZ} implies
\[
\sqrt{\Delta(1)} = (p(1,0)  + p(2,0) ) - (p(0,1) + p(0,2)) > 0.
\]
Then $\beta_1= r < \beta_2 = 1$ and
$(1,r)$ and $(1,1)$ are points on the characteristic curve ${\mathcal H}$.
Proposition \ref{p:structureofeiZ} implies $P_{(1,0)} (\tau < \infty) = \beta_1 = r$
or $P_{(1,0)}(\tau  <  \infty ) = \beta_2 = 1$. Proposition \ref{p:niceresult} 
implies that the former
must hold. Proposition \ref{p:structureofeiZ} now implies \eqref{e:casealpha1}.
\end{proof}

For a complex number ${\mathrm z}$ let $\Re({\mathrm z})$ [$\Im({\mathrm z})$] 
denote its real (imaginary) part.
If we write 
$\alpha = e^{i \theta}$,$ \theta \in (0,2\pi)$ and set ${\mathrm x} = \cos(\theta)$
then
\begin{align*}
\Re(\Delta) &= 2 {\mathrm x}^2 ( p(2,1)^2 + p(1,2)^2)
-2{\mathrm x} ( p(2,1) + p(1,2) + 2p(1,0)p(0,2)  + 2p(2,0)p(0,1) )\\
&~~~~~~~+ 1  - 4 ( p (0,1) p(1,0) + p(2,0) p(0,2) )
- (p(2,1) - p(1,2))^2\\
~\\
\Im(\Delta) &= \sin(\theta)(
  2{\mathrm x}(p(1,2)^2 -p(2,1)^{2}) +4 p(0,1) p(2,0)
 -4p(0,2) p(1,0) +2  p(2,1)- 2 p(1,2) ).
\end{align*}
$\Im(\Delta)/\sin(\theta)$ is affine in ${\mathrm x}$; to simplify exposition,
we will assume that this function has a unique root lying
outside of the interval $(-1,1)$:
\begin{equation}\label{e:simplifying0}
\frac{2p(0,2) p(1,0) - 2 p(0,1) p(2,0) + p(1,2)-  p(2,1) }{
p(1,2)^2 -p(2,1)^{2}} \notin (-1,1).
\end{equation}
See the end of this subsection for comments on \eqref{e:simplifying0}.
This 
and $\sin(\theta) \neq 0$ imply
\begin{equation}\label{e:nonzeroImDelta}
\Im(\Delta(\alpha)) \neq 0,
\theta \in (0,\pi).
\end{equation}

\begin{proposition}\label{p:betaunequal}
\begin{equation}\label{e:betaunequal}
\beta_1(\alpha)\neq  \beta_2(\alpha)
\end{equation}
for $\alpha = e^{i\theta}$.
\end{proposition}
\begin{proof}
$\beta_1 - \beta_2 = -\frac{\sqrt{\Delta}}{p(2,0) + p(1,0) \alpha}$
and \eqref{e:betaunequal} will follow from
\begin{equation}\label{e:deltanotzero}
\Delta(e^{i\theta}) \neq 0, 
\theta \in [0,2\pi).
\end{equation}
$\Delta$ is a polynomial with real coefficients. Then
$\Delta(\exp(-i\theta)) = \widebar{\Delta}(\exp(i\theta))$ and
hence, it suffices to prove \eqref{e:deltanotzero}
for $\theta \in [0,\pi]$.
\eqref{e:nonzeroImDelta}
implies \eqref{e:deltanotzero} for
$\theta \in [0,\pi)$. For $\theta = 0$: $\alpha=1$,
\eqref{e:Delta1} and Proposition \ref{p:unstableZ} imply $\Delta(1) > 0$.
For $\theta = \pi$: $\alpha=-1$ and 
\begin{equation}\label{e:DeltaEnd}
\Delta(-1) > \Delta(1).
\end{equation}
 These
prove \eqref{e:deltanotzero} for $\theta \in [0,\pi]$ and complete the proof.
\end{proof}
\begin{proposition}
\begin{equation}\label{e:b1funoftheta}
\theta \rightarrow \beta_1\left(e^{i\theta}\right),~~\theta \in (0,2\pi]~~
\theta \rightarrow \beta_2\left(e^{i\theta}\right),~~\theta \in (0,2\pi]
\end{equation}
are continuous.
\end{proposition}
\begin{proof}
We have defined $\sqrt{{\mathrm z}}$ to mean the complex number with positive real part
whose square equals ${\mathrm z}$; this definition leads to a discontinuity 
only when 
${\mathrm z}$
passes the negative side of the real axis on the complex plane. Then
the only possibility of discontinuity for the functions
$\beta_1$ and $\beta_2$ (as functions of $\theta$, as in \eqref{e:b1funoftheta})
is
if $\Delta(\exp(i\theta))$ crosses this half line as $\theta$ varies in $[0,2\pi)$.
But \eqref{e:Delta1}, \eqref{e:nonzeroImDelta} and
\eqref{e:DeltaEnd} imply that as $\theta$ varies in $[0,\pi)$,
$\Delta$ defines a curve ${\mathcal C}$ starting from and ending at the positive real line
and lying on either on the positive or the negative complex half plane.
That $\Delta$ is a polynomial with real coefficients implies that $\Delta(\exp(i\theta))$,
$\theta \in [\pi,2\pi)$ is the mirror image $\bar{\mathcal C}$ of ${\mathcal C}$ with respect to the
real line; ${\mathcal C}$ and $\bar{\mathcal C}$ together define a closed loop
that crosses the real line twice on its positive side. These imply that 
$\sqrt{\Delta}$ defines a continuous closed loop in $\{{\mathrm z} \in {\mathbb C}: \Re({\mathrm z}) > 0 \}$, from
which  the statement of the theorem follows.
\end{proof}

Figure \ref{f:deltab1b2} depicts the curves traced
on the ${\mathbb C}$-plane
by $\beta_1$, $\sqrt{\Delta}$
and $\beta_2$ as $\theta$ varies in $[0,2\pi)$ for the $p$ matrix of \eqref{e:pexample}.

\ninsepsc{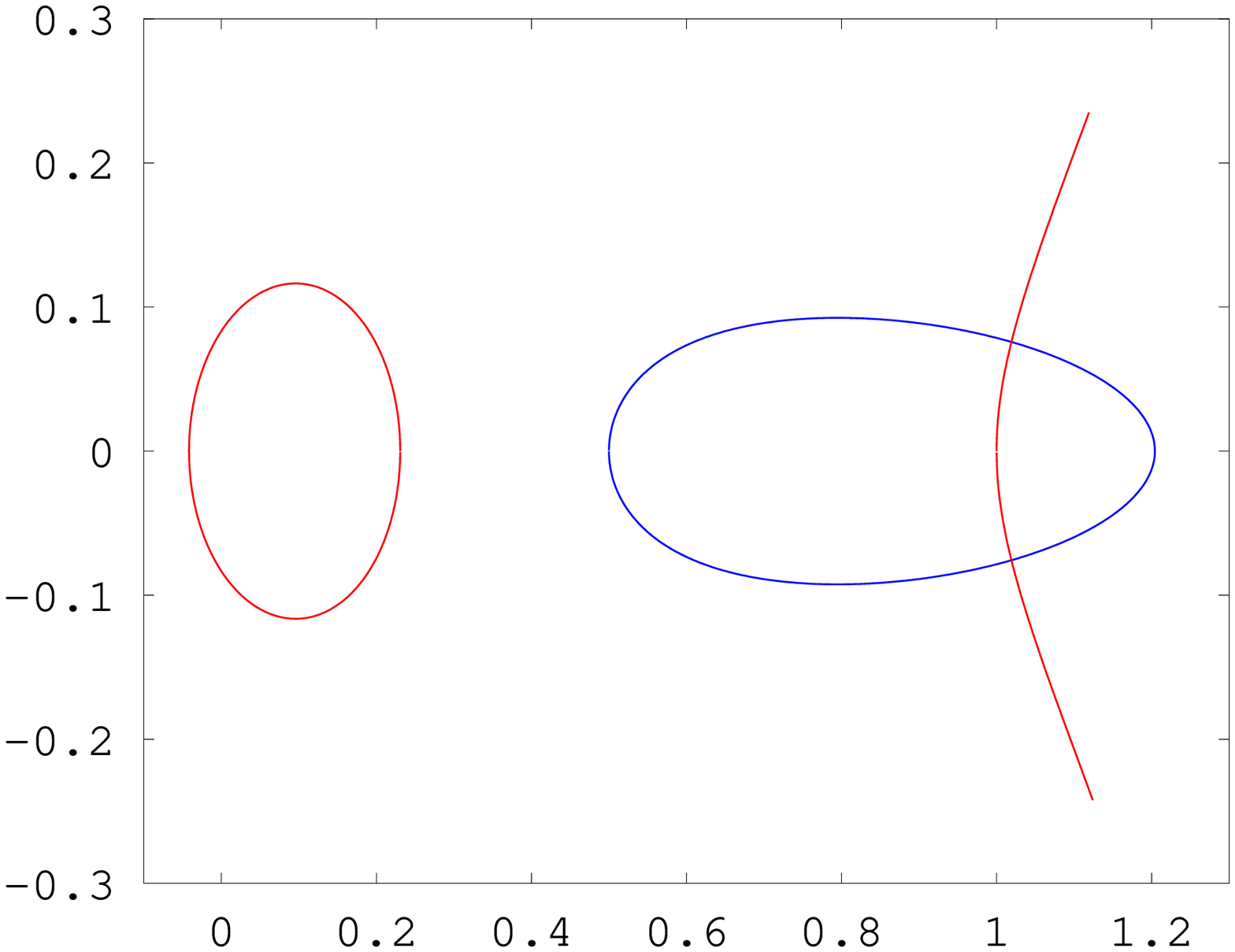}{$\beta_1(e^{i\theta})$, $\sqrt{\Delta(e^{i\theta})}$ and a section
of $\beta_2(e^{i\theta})$ (intersecting $\sqrt{\Delta}$), $\theta \in [0,2\pi)$,
 on the ${\mathbb C}$-plane for the $p$ listed in \eqref{e:pexample}
}
{0.4}

Proposition \ref{p:balayageZalpha1} extends to $|\alpha|=1$ as follows:
\begin{proposition}
\begin{equation}\label{e:casealphaabs1}
{\mathbb E}_z \left[ \exp(i\theta {Z_\tau(2) )} 1_{\{ \tau < \infty\}} \right]=
\exp(i\theta z(2))\left(\beta_1(e^{i\theta})\right)^{z(1) - z(2)}.
\end{equation}
and
\begin{equation}\label{e:boundonbeta1}
\left|\beta_1(e^{i\theta})\right|
=
\left|{\mathbb E}_{(1,0)} 
\left[ \exp(i\theta {Z_\tau(2) )} 1_{\{ \tau < \infty\}} \right]\right| \le r.
\end{equation}
\end{proposition}
\begin{proof}
The dominated convergence theorem implies that
\begin{equation}\label{e:funcoftheta}
\theta \rightarrow 
{\mathbb E}_z \left[ \exp(i\theta {Z_\tau(2) } 1_{\{ \tau < \infty\}} \right]
\end{equation}
is continuous in $\theta$. 
Proposition \ref{p:structureofeiZ} implies that for each fixed $\theta$
the value of this function equals either
\begin{equation}\label{e:alternative1}
\exp(i\theta z(2))\left(\beta_1\left(e^{i\theta}\right)\right)^{z(1)- z(2)}
\end{equation}
or
\[
\exp(i\theta z(2))\left(\beta_2\left(e^{i\theta}\right)\right)^{z(1)- z(2)}.
\]
\eqref{e:betaunequal} and \eqref{e:b1funoftheta} imply that these two expressions
are continuous as functions of $\theta$ and they are never equal. Then
\eqref{e:funcoftheta} must equal one or the other for all $\theta$ and therefore if one
can verify that \eqref{e:casealphaabs1} equals \eqref{e:alternative1} for a single $\theta$
then the equality must hold for all $\theta$; \eqref{e:casealpha1} asserts the
desired equality at $\theta=1$; \eqref{e:casealphaabs1} follows.
Set $z= (1,0)$ in \eqref{e:casealphaabs1} and take absolute values of both sides:
\begin{align*}
\left|\beta_1(e^{i\theta})\right|
&=
\left|{\mathbb E}_{(1,0)} 
\left[ \exp(i\theta {Z_\tau(2) )} 1_{\{ \tau < \infty\}} \right]\right|
\intertext{$|\cdot|$ is convex; then Jensen's inequality gives}
&\le 
{\mathbb E}_{(1,0)} 
\left[ 1_{\{ \tau < \infty\}} \right]= r,
\end{align*}
where the last equality is \eqref{e:casealpha1} with $z(1) = 1$.
\end{proof}
For two tandem queues, the narrow shaded region
of Figure \ref{f:Mu1Mu2DeltaIntersects} (let's call it ${\mathcal R}$)
shows the set of parameter values that violate \eqref{e:simplifying0} 
(the shaded triangle containing ${\mathcal R}$ 
shows the parameter values of stable tandem walks, i.e.,
the region $\lambda < \mu_1,\mu_2$ where $\lambda = 1-\mu_1-\mu_2$).
\ninsepsc{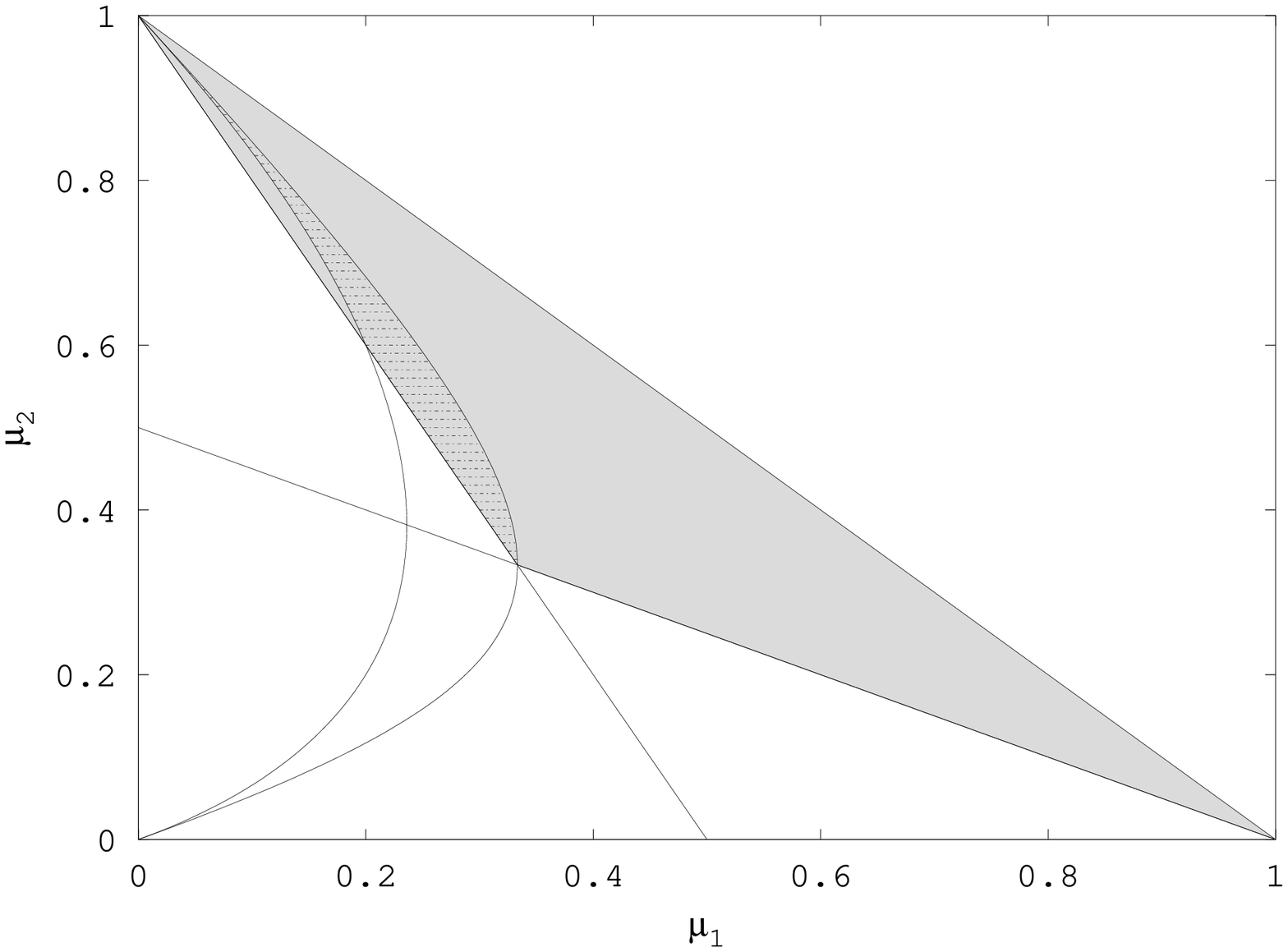}{The set of parameter values violating \eqref{e:simplifying0}
for the two dimensional tandem walk (the narrow shaded region on the left)}{0.4} 
The image of 
$\theta \rightarrow \Delta(e^{i\theta})$, $\theta \in [0,2\pi)$,
becomes a self intersecting curve on ${\mathbb C}$
when \eqref{e:simplifying0} fails and, e.g., the proof of Proposition \ref{p:betaunequal}
(or that of its adaptation to the parameter values in ${\mathcal R}$)
will require a study of $\Re(\Delta)$ over ${\mathcal R}$; we leave this
analysis to future work.

\subsection{$\log$-linear harmonic functions of $Y$}
Let us rewrite \eqref{e:linear} separately for the boundary $\partial_2$
and the interior $B^o -\partial_2$:
\begin{align}
V(y) &= \sum_{v \in {\mathcal V} } V(y + v) p(v), y \in B^o -\partial_2,
\label{e:linearint}\\
V(y) &= 
V(y)\mu_2 + 
\sum_{v \in {\mathcal V}, v(2)\neq -1 } V(y + v) p(v), y \in \partial_2 \cap B^o.
\label{e:linearboundary}
\end{align}
Any $ g\in {\mathcal F}_Z$ satisfies \eqref{e:linearint} (because \eqref{e:linearint}
is the restriction of \eqref{e:linear1} to $B^o - \partial_2$);
\eqref{e:linearint} is linear
and
so any finite linear combination of members of ${\mathcal F}_Z$ continues to satisfy
\eqref{e:linearint}. In the next two subsections we will show that
appropriate linear combinations of members of ${\mathcal F}_Z$ 
will also satisfy the
boundary condition \eqref{e:linearboundary} and define harmonic functions
of $Y$.

Parallel to the definitions in the previous section, we will define characteristic
polynomials and surfaces for $Y$. The constrained process will have a pair of these,
one set for the interior and one set for the boundary $\partial_2$. The interior
characteristic polynomial for $Y$ is by definition that of $Z$, i.e., ${\mathbf p}$
and its interior characteristic surface is ${\mathcal H}$.
The characteristic polynomial of $Y$ and its characteristic surface on the boundary
$\partial_2$ are defined below.

\subsubsection{A single term}\label{ss:asingleterm}
Remember that members of ${\mathcal F}_Z$
are of the form $[(\beta,\alpha),\cdot]:z\rightarrow\beta^{z(1)-z(2)} 
\alpha^{z(2)}$ 
and $(\beta,\alpha) \in {\mathcal H}$; these define harmonic functions for $Z$
and they therefore satisfy \eqref{e:linearint}.
The simplest approach of constructing a $Y$-harmonic function
is to look for $[(\beta,\alpha),\cdot]$ 
which satisfies \eqref{e:linear}, i.e., which satisfies
\eqref{e:linearint} {\em and} \eqref{e:linearboundary} at the same time. 
Substituting $[(\beta,\alpha),\cdot]$ in \eqref{e:linearboundary}
 we see that it solves \eqref{e:linearboundary}
if and only if $(\beta,\alpha) \in {\mathcal H}$ also satisfies
\begin{equation}\label{e:chareqYb2}
{\mathbf p}_{2}(\beta,\alpha) -\beta\alpha = 0
\end{equation}
where
\begin{align}\label{e:hamiltonianb2}
{\mathbf p}_{2}(\beta,\alpha)
&\doteq \beta\alpha 
\left( \sum_{ v\in {\mathcal V}, v(2)\neq -1 } p(v) \beta^{v(1)-v(2)}
\alpha^{v(2)} + \mu_2\right).
\end{align}

We will call
\eqref{e:chareqYb2} ``the characteristic equation of $Y$ on $\partial_2$''
and ${\mathbf p}_{2}$ its characteristic polynomial on the same boundary.
${\mathbf p}_2$ can be expressed in terms of ${\mathbf p}$ as follows:
\begin{align}
{\mathbf p}_2 &= {\mathbf p}(\beta,\alpha) -
 \beta\alpha\left( \sum_{ v\in {\mathcal V}, v(2)= -1 } p(v) 
\beta^{v(1)-v(2)}\alpha^{v(2)} - \mu_2\right)\notag\\
&= {\mathbf p}(\beta,\alpha) -
 \beta\alpha\left( p(2,0)\frac{\beta}{\alpha} + p(2,1)\frac{1}{\alpha}
- \mu_2\right).
 \label{e:intermsofPY}
\end{align}

Define the boundary characteristic surface of $Y$ for $\partial_2$ as
${\mathcal H}_{2} \doteq \{ (\beta,\alpha) \in {\mathbb C}^2: 
{\mathbf p}_{2}(\beta,\alpha) = 0
\}$.

Then,
$(\beta,\alpha)$
must lie on ${\mathcal H} \cap{\mathcal H}_{2}$
for $[(\beta,\alpha),\cdot]$ to be a harmonic function of $Y$.
Suppose $(\beta,\alpha) \in {\mathcal H} \cap{\mathcal H}_{2}$, i.e.,
${\mathbf p}_{2}(\beta,\alpha) = {\mathbf p}(\beta,\alpha) = \beta\alpha$;
and $\beta,\alpha \neq 0$. 
Then, by \eqref{e:intermsofPY} 
\begin{align}
0 &= p(2,1)\frac{1}{\alpha} + p(2,0) \frac{\beta}{\alpha} - \mu_2, \notag \\
\beta &= \frac{1}{p(2,0)}\left( \mu_2\alpha - p(2,1)\right).\label{e:simpleb2}
\end{align}
Substituting this back in \eqref{e:chareqZ} implies that $\alpha$ 
must solve
\[
{\mathbf p}^{r}_{2}(\alpha)
=0,
\]
where
\begin{align*}
&{\mathbf p}^{r}_{2}(\alpha)
\doteq \alpha \Bigg(
\frac{\mu_2}{p(2,0)}
\left( \mu_2 \frac{p(1,0)}{p(2,0)} + p(1,2)\right)\alpha^2\\
&~~+
\left(
p(0,2) + \frac{\mu_2}{p(2,0)}
\left( \mu_2 - 2\frac{p(1,0)p(2,1)}{p(2,0)} -1\right)
- \frac{p(1,2)p(2,1)}{p(2,0)}\right) \alpha\\
&~~~+
p(0,1) +\frac{p(2,1)}{p(2,0)}\left(
1 + \frac{p(1,0) p(2,1)}{p(2,0)} - \mu_2\right)\Bigg);
\end{align*}
(the superscript $r$ stands for ``reduced.'')
Then $[(\beta,\alpha),\cdot]$ is a harmonic function of $Y$ if and only if $\alpha$ is a root
of ${\mathbf p}^r_{2}$ and $\beta$ is defined by \eqref{e:simpleb2}.
The functions $z \rightarrow 1$  and $z\rightarrow 0$ 
are harmonic functions of $Y$  of
the form $z \rightarrow \beta^{z(1)-z(2)} \alpha^{z(1)}$; then
two of the roots of ${\mathbf p}^r_{2}$ are $0$ and $1$ (that $0$ is a root
also directly follows from the form of ${\mathbf p}^r_{2}$).
It follows that the third root is
\[
r_1 \doteq
\frac{
p(0,1) +\frac{p(2,1)}{p(2,0)}\left(
1 + \frac{p(1,0) p(2,1)}{p(2,0)} - \mu_2\right)}
{
\frac{\mu_2}{p(2,0)}
\left( \mu_2 \frac{p(1,0)}{p(2,0)} + p(1,2)\right)
}.
\]
This quantity is always less than $1$ if the first queue is stable:
\begin{lemma}
$\nu_1 < \mu_1$ if and only if $r_1 < 1$.
\end{lemma}
\begin{proof}
$r_1 <1 $ is equivalent to
\begin{align}\label{e:boundonr_1}
p(0,1) + \frac{p(2,1)}{p(2,0)}  ( 1-\mu_2)
&< p(1,0) + \frac{2p(2,1)p(1,0)}{p(2,0)} 
+ p(1,2) + \frac{p(2,1)p(1,2)}{p(2,0)} \notag
\intertext{substitute $\mu_1 + p(0,1) + p(0,2)$ for $1-\mu_2$, multiply
both sides by $p(2,0)$ and cancel out equal terms from both sides:}
\mu_2 p(0,1) + p(2,1)p(0,2) &< \mu_1 p(2,0) + p(2,1)p(1,0)\\
\mu_2 p(0,1) + p(2,1)p(0,2) &< p(1,2) p(2,0) + \mu_2p(1,0)\notag\\
p(0,1) + \frac{p(2,1)}{\mu_2}p(0,2) &< p(1,2)\frac{p(2,0)}{\mu_2} + p(1,0);
\notag
\end{align}
divide both sides by $1 - \frac{p(1,2)p(2,1)}{\mu_1 \mu_2}$
to get $\nu_1 < \mu_1$.
This establishes the ``only if'' part of the statement of the lemma.
The last sequence of inequalities in reverse gives the ``if'' part.
\end{proof}
And thus we get our first nontrivial harmonic function for $Y$:
\begin{proposition}\label{p:singletermharmY}
The function
\begin{equation}\label{e:singletermharmfY}
z\rightarrow
\beta(r_1)^{z(1) - z(2)} r_1^{z(1)} 
\end{equation}
is a harmonic function of $Y$ where 
\begin{equation*}
\beta(\alpha) \doteq \frac{1}{p(2,0)}\left( \mu_2\alpha - p(2,1)\right)
\end{equation*} is the right side
of \eqref{e:simpleb2}.  Furthermore $0 < \beta(r_1) < 1$.
\end{proposition}
\begin{proof}
It remains only to prove the last part of the proposition's statement.
$r_1 < 1$ implies
$r_1 \mu_2 < \mu_2 = p(2,0) + p(2,1)$ or, what is the same,
\[
\frac{1}{p(2,0)} \left( r_1 \mu_2 - p(2,1)\right) < 1;
\]
this is $\beta(r_1) < 1$. The inequality
$\beta(r_1) > 0$ 
turns out to be true
for all $p$ as long as $p(2,0) >0$ 
and follows from a sequence of inequalities similar to \eqref{e:boundonr_1}.
\end{proof}

\subsubsection{Two terms}\label{ss:twoterms}
Define the boundary operator $D_2$ acting on functions on ${\mathbb Z}^2$ and giving
functions on $\partial_2$:
\begin{align*}
&D_2V = g,~~~~ V: {\mathbb Z}^2 \rightarrow {\mathbb C},\\
& g( {\mathrm y},0) \doteq
\left(\mu_2  +  \sum_{ v\in {\mathcal V}, v(2)\neq -1 } p(v) V( ({\mathrm y},0) + v) \right)
- V({\mathrm y},0),~~ {\mathrm y} \in {\mathbb Z};
\end{align*}
(if $V$ is defined on a subset of ${\mathbb Z}^2$ one may extend it trivially to
all of ${\mathbb Z}^2$ to apply $D_2$).
$D_2$ is the difference between the left and the right sides of \eqref{e:linearboundary}
and gives how much $V$ deviates from being $Y$-harmonic
along the boundary $\partial_2$:
\begin{lemma}\label{l:D2eq0}
 $D_2 V = 0$ 
if and only if $V$ is a harmonic function of $Y$ on $\partial_2$.
\end{lemma}
The proof follows from the definitions involved.
For $(\beta,\alpha ) \in {\mathbb C}^2$ and
$\beta,\alpha \neq 0$
\[
\left[D_2 \left( [({\beta,\alpha}),\cdot] \right)\right]({\mathrm y},0) =  
\left( \frac{1}{\beta\alpha} {\mathbf p}_2(\beta,\alpha) - 1\right) \beta^{{\mathrm y}}.
\]
where the left side denotes the value of the function 
$D_2 \left( [(\beta,\alpha),\cdot] \right)$ at $({\mathrm y},0)$, 
${\mathrm y} \in {\mathbb Z}$.
For $(\beta,\alpha) \in {\mathcal H}$, ${\mathbf p}(\beta,\alpha)=\beta\alpha$;
this, the last display and \eqref{e:intermsofPY}
imply
\begin{equation}\label{e:imagofD2}
\left[D_2 \left( [(\beta,\alpha),\cdot] \right)\right]({\mathrm y},0) =  
\left(\mu_2- \left( p(2,0)\frac{\beta}{\alpha} + p(2,1)\frac{1}{\alpha}\right)\right)
\beta^{{\mathrm y}}
\end{equation}
if $(\beta,\alpha) \in {\mathcal H}$.
One can write the function $({\mathrm y},0) \rightarrow \beta^{{\mathrm y}}$
as $[(\beta,\alpha),\cdot] |_{\partial_2} = [(\beta,1),\cdot] |_{\partial_2}$;
in addition, define
\begin{equation}\label{e:defC2d}
C(\beta,\alpha) \doteq \mu_2-  \left( p(2,0)\frac{\beta}{\alpha} + p(2,1)\frac{1}{\alpha}
\right).
\end{equation}
With these, rewrite \eqref{e:imagofD2}
as
\begin{equation}\label{e:imageofD2short}
D_2 \left( [(\beta,\alpha),\cdot] \right)  = 
C(\beta,\alpha) [(\beta,1),\cdot]|_{\partial_2}.
\end{equation}
The key observation here
is this: $D_2\left( [(\beta,\alpha),\cdot]\right)$ is a constant
multiple of $[(\beta,1),\cdot] |_{\partial_2}$.
This and the linearity of $D_2$ imply that for
\begin{equation}\label{e:conjugacy}
\alpha_1 \neq \alpha_2, ~~(\beta,\alpha_1), (\beta,\alpha_2)  \in {\mathcal H},
\end{equation}
$[(\beta,\alpha_1),\cdot]$ and $[(\beta,\alpha_2),\cdot]$
can be linearly combined to cancel out
each other's value under $D_2$. 
We will call $(\beta,\alpha_1)$ and $(\beta,\alpha_2)$  {\em conjugate}
if they satisfy \eqref{e:conjugacy}.
An example:
the two end points of the dashed line in Figure \ref{f:charsurfexample}
are conjugate to each other.

Because the characteristic equation \eqref{e:chareqZ}
is quadratic in $\alpha$, 
fixing $\beta$ in \eqref{e:chareqZ}
and solving for $\alpha$ will give a conjugate pair 
$(\beta,\alpha_1)$ and $(\beta,\alpha_2)$
satisfying
\begin{equation}\label{e:conjalphas}
\alpha_1 + \alpha_2 = \frac{ \beta - p(1,0) \beta^2 - p(0,1)}{p(0,2) + \beta p(1,2)}
\end{equation}
for most $\beta \in {\mathbb C}$; the next proposition uses these
conjugate pairs and the above observation to define harmonic functions of $Y$:
\begin{proposition}\label{p:harmonicYtwoterms2d}
Suppose that for $\beta \in {\mathbb C}$,
$\beta \neq 0$,
$p(0,2) + \beta p(1,2) \neq 0$.
Then
\begin{equation}\label{e:harmonicY2d}
h_\beta \doteq  C(\beta,\alpha_2) [(\beta,\alpha_1),\cdot]
- C(\beta,\alpha_1) [(\beta,\alpha_2),\cdot]
\end{equation}
is a harmonic function of $Y$.
\end{proposition}
\begin{proof}
By assumption
$(\beta,\alpha_1), (\beta,\alpha_2)$ are both on ${\mathcal H}$ and
therefore $[(\beta,\alpha_1),\cdot]$ and $[(\beta,\alpha_2),\cdot]$ are
 harmonic functions of $Z$ and in particular, they both satisfy
\eqref{e:linearint}. Then their linear combination $h_\beta$ also
satisfies \eqref{e:linearint}, because \eqref{e:linearint}  is linear in $V$.
It remains to show that $h_\beta$ solves \eqref{e:linearboundary} as well.
$\beta \neq 0$ implies $\alpha_1,\alpha_2 \neq 0,1$. Then 
\eqref{e:imageofD2short} implies
\begin{align*}
D_2(h_\beta) &= 
 C(\beta,\alpha_2) 
D_2( [(\beta,\alpha_1),\cdot]) -  C(\beta,\alpha_1) D_2([\beta,\alpha_2,\cdot])\\
&= 
C(\beta,\alpha_2) 
C(\beta,\alpha_1)  [(\beta,1),\cdot]|_{\partial_2}
-
C(\beta,\alpha_1)
C(\beta,\alpha_2)  [(\beta,1),\cdot]|_{\partial_2}\\
& = 0
\end{align*}
and
Lemma \ref{l:D2eq0} implies that $h_\beta$ satisfies \eqref{e:linearboundary}.
\end{proof}
\begin{remark}\label{r:specialcase}{\em
If we set $\beta = \beta(r_1)$
in the last proposition $h_\beta$ reduces to a constant
multiple of \eqref{e:singletermharmfY}.
}
\end{remark}

With Proposition \ref{p:harmonicYtwoterms2d} 
 we define our basic class of harmonic functions of $Y$:
\begin{equation}\label{d:classFY}
{\mathcal F}_Y \doteq \{ h_\beta, \beta \neq 0, p(0,2) + \beta p(1,2) \neq 0 \}.
\end{equation}

Members of ${\mathcal F}_Y$ consist of linear combinations of $\log$-linear functions;
with a slight abuse of language, we will also refer to such functions as
$\log$-linear.

\begin{lemma}
Suppose $p(0,2) + \beta p(1,2) \neq 0$
so that \eqref{e:conjalphas} makes sense.
Suppose further that
$(\beta,\alpha_1) \in {\mathcal H}$,
$(\beta,\alpha_2) \in {\mathcal H}$
are conjugate with 
$\alpha_1, \alpha_2 \neq 0$.
Then \eqref{e:conjalphas} is equivalent to
\begin{equation}\label{e:conjugator0}
\alpha_i = \frac{1}{\alpha_{3-i}}\frac{p(2,0)\beta^2 + p(2,1) \beta}{p(0,2)+ \beta p(1,2)},
i \in \{1,2\}.
\end{equation}
\end{lemma}
\begin{proof}
By \eqref{e:conjalphas}
\begin{align*}
\alpha_i
&= \frac{\alpha_{3-i}(\beta - p(1,0) \beta^2  -p(0,1))}{\alpha_{3-i}(p(0,2) + \beta p(1,2))} -\alpha_{3-i}\\
&= \frac{\alpha_{3-i}\beta - p(1,0) \beta^2  -p(0,1) -\alpha_1^2 p(0,2) - \beta\alpha_{3-i}^2p(1,2)}
{\alpha_{3-i}(p(0,2) + \beta p(1,2))} 
\intertext{ $(\alpha_{3-i},\beta) \in {\mathcal H}$ implies}
&=\frac{1}{\alpha_{3-i}}\frac{p(2,0)\beta^2 + p(2,1) \beta}{p(0,2)+ \beta p(1,2)}.
\end{align*}
\end{proof}

Define
\begin{equation}\label{e:conjugator}
\boldsymbol \alpha(\beta,\alpha) \doteq 
\frac{1}{\alpha}\frac{p(2,0) \beta^2+ p(2,1) \beta}{p(0,2)+ \beta p(1,2)}
\end{equation}
We can write \eqref{e:conjugator0} as
\[
\alpha_{i} = \boldsymbol \alpha (\alpha_{3-i}).
\]
The map ${\boldsymbol\alpha}$
is invertable (it is a multiple of $\alpha^{-1}$)
 and its inverse equals itself. Thus, conjugacy is symmetric:
if $(\beta,\alpha_2)$
is conjugate to $(\beta,\alpha_2)$, then $(\beta,\alpha_1)$ is conjugate to
$(\beta,\alpha_2)$. We will sometimes refer to $\boldsymbol \alpha$ as 
{\em conjugator.}

\subsection{Graph representation of $\log$-linear harmonic functions of $Y$}
Figure \ref{f:graphsharmtex}
gives a graph representation of the harmonic functions developed
in the last subsection.

\begin{figure}[h]
\begin{center}
\scalebox{0.8}{
\centerline{\input{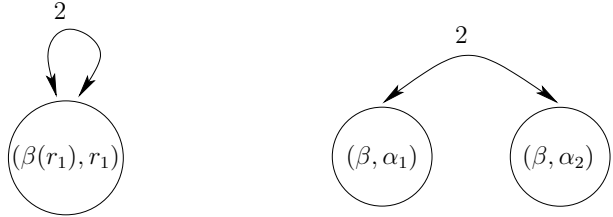}}}
\end{center}

\vspace{-0.65cm}
\caption{\hspace{0.25cm}The harmonic functions of $Y$\label{f:graphsharmtex}}
\end{figure}

Each node in this figure represents a member of ${\mathcal F}_Z$. 
The edges represent the boundary conditions; in this case there is only
one, \eqref{e:linearboundary} of $\partial_2$, and the edge label ``$2$''
 refers to $\partial_2$.
A self connected vertex represents a member of ${\mathcal F}_Z$ that also
satisfies the $\partial_2$ boundary condition \eqref{e:linearboundary}, i.e.,
$z \rightarrow \beta(r_1)^{z(1)} r_1^{z(1)-z(2)}$ 
of Proposition \ref{p:singletermharmY};
the graph on the left represents exactly this function. The ``$2$'' labeled
edge on the right represents the conjugacy relation
\eqref{e:conjugator0} between $\alpha_1$ and $\alpha_2$, which allows these functions
to be linearly combined to satisfy the harmonicity condition of $Y$ on $\partial_2$.
\subsection{$\partial B$-determined harmonic functions of $Y$}
\label{ss:partialBdet}

Our task now is to distinguish a collection of $\partial B$-determined
members of ${\mathcal F}_Y.$
This  collection will form a basis
of harmonic functions with which we will approximate/ represent
the rest of the $\partial B$-determined functions of $Y$.

\begin{proposition}\label{p:balayagesimple}
Let $\alpha_1$, $\alpha_2$ and $\beta$ be as in Proposition
\ref{p:harmonicYtwoterms2d}. 
If 
\begin{equation}\label{e:simpleconditions}
|\beta| < 1,~~~|\alpha_1|,|\alpha_2|  \le 1
\end{equation}
then $h_\beta$ of \eqref{e:harmonicY2d} is $\partial B$-determined.
\end{proposition}
\begin{proof}
By Proposition \ref{p:harmonicYtwoterms2d}  $h_\beta$ is  a
$Y$-harmonic function;
\eqref{e:simpleconditions} and its definition \eqref{e:harmonicY2d}
imply that $h_\beta$ is also bounded on $B^o$. 
Then
$M_k  =  h_\beta(Y_{\tau \wedge \zeta_n \wedge k})$ is a bounded martingale. 
This, Proposition \ref{p:cantwanderforever} and the optional
sampling theorem imply
\begin{align}\label{e:resofopsamp}
h_\beta(y) &= 
{\mathbb E}_y
\left[
h_\beta(Y_{\tau}) 1_{\{\tau < \zeta_n\}} \right]
+ {\mathbb E}_y\left[
h_\beta(Y_{\zeta_n}) 
1_{\{ \zeta_n \le \tau \}} \right], y \in B^o.
\end{align}
$Y_{\zeta_n}(1) = n$ for $\tau > \zeta_n.$ This and \eqref{e:simpleconditions}
imply 
\[
\lim_{n\rightarrow \infty}
 {\mathbb E}_y\left[
h_\beta(Y_{\zeta_n}) 
1_{\{ \zeta_n \le \tau \}} \right]
\le \lim_{n\rightarrow \infty} \beta^n = 0.
\]
This, $\lim_n \zeta_n = \infty$ and letting $n\rightarrow \infty$
in \eqref{e:resofopsamp}
give
\[
h_\beta(y) ={\mathbb E}_y \left[ h_\beta(Y_{\tau}) 1_{\{\tau <\infty \}} 
\right],
\]
i.e, $h_\beta$ is $\partial B$-determined.
\end{proof}

Remark \ref{r:specialcase}, Proposition \ref{p:singletermharmY} and the last
proposition imply
\begin{proposition}\label{p:partialBdetermined}
The harmonic function \eqref{e:singletermharmfY} is $\partial B$-determined.
\end{proposition}

What Proposition \ref{p:balayagesimple} does is it gives us a collection of 
{\em basis functions} for which the Balayage operator ${\mathcal T}_{(B^o)^c}$ 
is extremely simple to compute; these functions play the same role for the current problem
as the one which exponential functions do
in the solution of linear ordinary differential equations
or the trigonometric functions in the solution
of the heat and the Laplace equations.
Let us rewrite Proposition \ref{p:balayagesimple} more explicitly.
Suppose $\alpha_1$, $\alpha_2$ and $\beta$ 
are as in Proposition \ref{p:balayagesimple}.
Define 
\[
f(y) = 
 C(\beta,\alpha_2)\alpha_1^{y(2)} - 
 C(\beta,\alpha_1)\alpha_2^{y(2)}, y \in {\mathbb Z}^2.
\]
Then, Proposition \ref{p:balayagesimple} says
\begin{align*}
{\mathbb E}_y\left[ 
f(Y_\tau)1_{\{\tau < \infty\} }\right] &= 
\beta ^{y(1)-y(2)}\left(
C(\beta,\alpha_2)
\alpha_1^{y(2)}
-
C(\beta,\alpha_1)\alpha_2^{y(2)} \right)\\
	&= 
\beta ^{y(1)-y(2)} f(y),
~~~y \in B.
\end{align*}
Proposition \ref{p:balayagesimple} rests on the condition
 \eqref{e:simpleconditions};
 Proposition \ref{p:sufficientconditions} below identifies a set of
conjugate pairs $(\beta,\alpha_1)$ and $(\beta,\alpha_2)$ on ${\mathcal
H}$ satisfying \eqref{e:simpleconditions}.

\subsection{A modified Fourier basis for ${\mathcal T}_{(B^o)^c}$}\label{ss:modifiedfourier}
Let's go back for a moment to the problem of evaluating the
Balayage operator of the unconstrained process $Z$ 
for the set $(B_Z^o)^c$.
For a bounded function $f$
on $\partial B_{Z}$, this is the operator mapping $f$ to the harmonic function
\[
{\mathcal T}_{(B_Z^o)^c}(f)= z\rightarrow {\mathbb E}_z 
\left[ f(Z_\tau) 1_{\{\tau < \infty\}}\right], z \in B_{Z};
\]
in the present subsection we will write
${\mathcal T}$ for
${\mathcal T}_{(B_Z^o)^c}$
For the Fourier basis functions
\begin{equation}\label{e:fourierbasis}
f_\alpha:\partial B_Z\rightarrow {\mathbb C},  
f_\alpha ({\mathrm y},{\mathrm y}) = \alpha^y, {\mathrm y} \in {\mathbb Z}, |\alpha|=1,
\end{equation}
we already know how to compute ${\mathcal T}(f_\alpha)$ (given by \eqref{e:casealphaabs1})
and ${\mathcal T}$ is linear. 
One can use these to evaluate ${\mathcal T}$ more generally
in three related ways. First,
Fourier series theory tells us that if $f$ is $l_2$, i.e., if
$\sum_{{\mathrm y} \in {\mathbb Z}} |f({\mathrm y}, {\mathrm y})|^2 < \infty$, 
it can be written in terms
of the Fourier basis functions thus:
\begin{equation}\label{e:defhatf}
f({\mathrm y},{\mathrm y}) = 
\frac{1}{2\pi} \int_{-\pi}^{\pi} \hat{f}(\theta) e^{i {\mathrm y} \theta} d\theta,
\hat{f}(\theta) \doteq \sum_{{\mathrm y}=-\infty}^{\infty} 
f({\mathrm y},{\mathrm y}) e^{-i{\mathrm y} \theta}
\end{equation}
Fubini's theorem
now implies
${\mathcal T}(f) = \frac{1}{2\pi} \int_{-\pi}^{\pi} \hat{f}(\theta) {\mathcal T}(f_{e^{i\theta}}) d\theta$,
and one can construct approximating sequences 
by truncating the sum in \eqref{e:defhatf}.

Second, when interpreted as a function of $\alpha$,
the formula \eqref{e:casealphaabs1} gives the characteristic function of the
distribution ${\mathcal T}$. Its inversion would give the distribution 
${\mathcal T}$ itself.

Third, we can first replace $f$ with its periodic approximation defined
as follows
\[
f^p({\mathrm y},{\mathrm y}) = \begin{cases} 
	f({\mathrm y},{\mathrm y}), \text{ if } |{\mathrm y}| \le N,\\
	f( {\mathrm y} \mod N, {\mathrm y} \mod N), \text{ if } |{\mathrm y}| \ge N.
	\end{cases}
\]
As $N$ increases, $|{\mathcal T}(f^p) - {\mathcal T}(f)|$ will converge to $0$.
Because it is periodic, $f^p$ has a unique Fourier representation of
the form
$f^p = \sum_{k=0}^{2N+1} c_k f_{\alpha_k}$
where $ \alpha_k =  e^{  i \left(\frac{k 2\pi}{2N+1}\right)}.$
Then
${\mathcal T}(f^p) = \sum_{k=0}^{2N+1} c_k {\mathcal T}(f_{\alpha_k}).$

How should one proceed to build a parallel theory for the constrained process $Y$?
The first obstacle to the above development in the case of $Y$
is that the Balayage of the Fourier basis functions is not simple to compute,
i.e., we don't know a simple way to compute ${\mathcal T}_{(B^o)^c}(f_\alpha)$. 
But Proposition \ref{p:balayagesimple} 
says that if 
\begin{equation}\label{e:conditionsperturbed}
|\alpha' = {\boldsymbol \alpha}(\beta_1(\alpha),\alpha)| < 1,~~
C(\beta_1(\alpha),\alpha')\neq 0
\end{equation}
then 
the Balayage of the {\em perturbed}
Fourier basis function
\begin{equation}\label{e:perturbedfourier}
({\mathrm y}, {\mathrm y}) \rightarrow \alpha^{\mathrm y} - 
\frac{C(\beta_1(\alpha),\alpha)}{C(\beta_1(\alpha),\alpha')}
(\alpha')^{\mathrm y}, {\mathrm y} \in {\mathbb Z}_+,
\end{equation}
is simple to compute and is given by
\begin{equation}\label{e:simpleimagebal}
y \rightarrow \beta_1(\alpha)^{y(1)-y(2)} \left(
\alpha^{y(2)} -
\frac{C(\beta_1(\alpha),\alpha) }{C(\beta_1(\alpha),\alpha')} (\alpha')^{y(2)}\right),
y \in B.
\end{equation}
One can interpret 
\eqref{e:perturbedfourier} as a perturbation of the restriction of
\eqref{e:fourierbasis} to $\partial B$,
because $|\alpha'| < 1$ implies that these two functions equal each other for
${\mathrm y}$ large.
Below in Proposition \ref{p:sufficientconditions} 
we identify a set of $\alpha$'s for which \eqref{e:conditionsperturbed}
holds and, therefore, for which the image of \eqref{e:perturbedfourier} 
under the Balayage operator is given by 
\eqref{e:simpleimagebal}.
This will require further assumptions on $p$;
in particular, we will {\em assume} \eqref{e:conditionsperturbed} for $\alpha=1$:
\begin{equation}\label{e:stability2}
{\boldsymbol \alpha}(r,1)< 1,
\end{equation}
\begin{equation}\label{e:conjnotharmonic}
C(r,{\boldsymbol \alpha}(r,1)) \neq 0,
\end{equation}
where, as before, $r$ is the input/output ratio \eqref{e:defr}.

Comments on these assumptions:
\begin{enumerate}
\item
If $p(2,1)$ is large enough, ${\boldsymbol \alpha}(r,1)$  can exceed
$1$ even when the stability assumption \eqref{e:stable} holds. For such networks,
we cannot check whether $h_r$ is $\partial B$-determined by an application
of Proposition \ref{p:balayagesimple}. Furthermore, even if
$h_r$ is $\partial B$-determined, $[(r,{\boldsymbol \alpha}(r,1)),\cdot]$
will dominate $h_r$ for large $y(2)$ and one can no longer think
of $h_r/C(r,{\boldsymbol \alpha}(r,1))$ as a perturbation of the function $1$ on $\partial B$.
\item
The condition
\eqref{e:conjnotharmonic} implies that 
the $\log$-linear function $[(r,{\boldsymbol \alpha}(r,1)),\cdot]$ 
is not $Y$-harmonic.
For two dimensional tandem walk,
\eqref{e:conjnotharmonic} reduces to $\mu_1 \neq \mu_2$. One can treat
the case
$\mu_1 = \mu_2$ by writing it as the limit $\mu_1 \rightarrow \mu_2$.
This limiting procedure introduces harmonic functions with polynomial terms, see
subsection \ref{ss:harmpoly}.
\end{enumerate}
With \eqref{e:stability2} and \eqref{e:conjnotharmonic}
we are able to take $\alpha=1$ in \eqref{e:perturbedfourier}.
Subsection \ref{ss:asingleterm} 
implies that for $\alpha \in S^1$ and $\alpha \neq 1$,
$C(\beta_1(\alpha),\alpha) \neq 0$. Thus, for such $\alpha$ only the first
condition in \eqref{e:conditionsperturbed} is nontrivial.
In the next proposition we will show that 
\eqref{e:stability2} implies  $|\boldsymbol\alpha(\beta_1(\alpha),\alpha)| <1$
for all $|\alpha| =1$.
To simplify its proof, we will
further assume $p(0,2) = 0.$
Treating $p(0,2) \neq 0$ seems to require 
a more refined analysis of $\beta_1$
as a function of $p$ and $\alpha$, a task we defer to future work.
\begin{proposition}\label{p:alphaconjalwayslessthan1}
Suppose $p(0,2) = 0$,
\eqref{e:stability2} and \eqref{e:condalpha0} hold;
for $p(1,2) = 0$ the system becomes trivial, so we will also assume $p(1,2) \neq 0.$
Then 
\begin{equation}\label{e:alphaconjalwayslessthan1}
|{\boldsymbol \alpha}(\beta_1(\alpha),\alpha)| < 1
\end{equation}
for all $|\alpha| = 1$.
\end{proposition}
\begin{proof}
The definition of ${\boldsymbol \alpha}$ and $p(0,2) = 0$ imply
${\boldsymbol \alpha}(\beta,\alpha) = {\bf c}(\alpha)/\alpha$
where \[{\bf c}(\alpha) \doteq ( p(2,0) \beta_1(\alpha) + p(2,1))/p(1,2).\]
Then
\begin{equation}
|{\boldsymbol \alpha}(\beta_1(\alpha),\alpha)| = |{\bf c}(\alpha)| \label{e:alphaconjlength}
\end{equation}
because $|\alpha|=1$.
As $\alpha$ varies on $S^1$, ${\bf c}(\alpha)$ defines a closed curve in
${\mathbb C}$ that is symmetric around the real axis. \eqref{e:boundonbeta1} 
implies that this curves is contained in a circle centered
at the point $p(2,1)/p(1,2)$ 
with radius $r p(2,0)/p(1,2)$, which in turn is contained
in the circle centered at the origin and with radius
$(p(2,1) + rp(2,0))/p(1,2)= {\boldsymbol \alpha}(r,1)$ which by assumption
\eqref{e:stability2} is less than $1$; then $|{\bf c}(\alpha)| < 1$.
This and \eqref{e:alphaconjlength} imply \eqref{e:alphaconjalwayslessthan1}.
\end{proof}
We have assumed $p(0,2) =0$ only to simplify the above proof;
Figure \ref{f:alphaconj} shows an example
with $p(2,0) \neq 0$ where again $|\boldsymbol\alpha| < 1$.

\ninsepsc{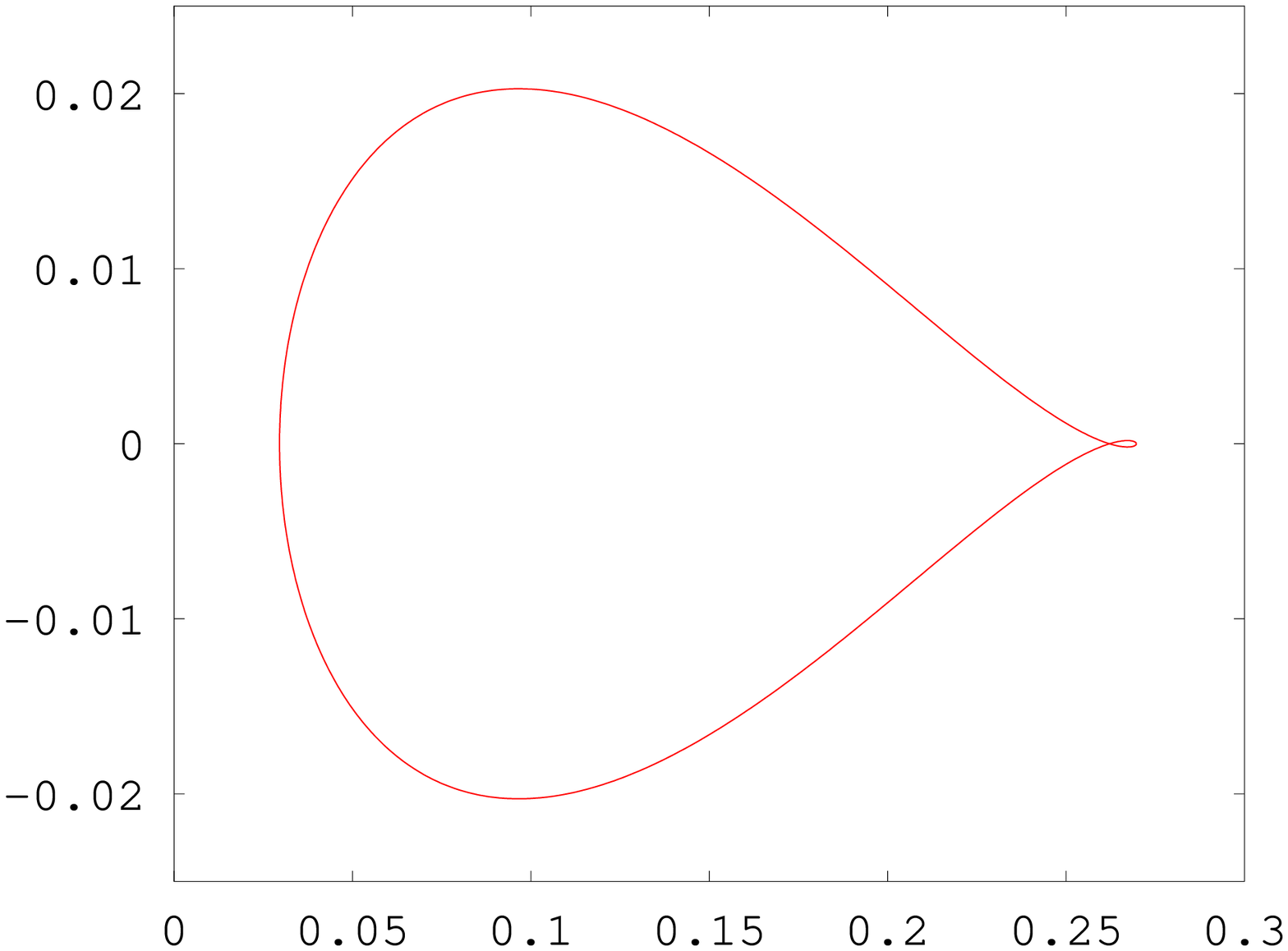}{The graph of 
${\boldsymbol \alpha}(\beta,\alpha)$, $\alpha = \exp(i\theta)$ on the
${\mathbb C}$-plane as $\theta$ varies in $[0,2\pi)$ for the $p$ of \eqref{e:pexample}}{0.4}

\begin{remark}\label{r:annoyingremark}
{\em 
If 
$p(2,0)/p(1,0)= 1$, $\alpha=-1$ violates
\eqref{e:condalpha0} and the last proposition is not applicable
at $\alpha=-1$, because
$\beta_1$ is not well defined there. But in that case the single 
root of the affine \eqref{e:polinbeta} 
will take the place of $\beta_1$ above. With this modification,
the above argument works verbatim for 
$p(2,0)/p(1,0)= 1$
as when 
$p(2,0)/p(1,0) \neq 1$
and from here on we assume that a similar
modification is made when a violation of \eqref{e:condalpha0} occurs.
}
\end{remark}

\begin{proposition}\label{p:sufficientconditions}
Assume \eqref{e:stable}, \eqref{e:stability2}, \eqref{e:conjnotharmonic} and
\eqref{e:simplifying0}. Then there exists $0 < R < 1$ such that for all
$\alpha \in {\mathbb C}$ with $R < |\alpha| \le 1$,
$h_{\beta_1(\alpha)}$
is a $\partial B$-determined harmonic function of $Y$.
\end{proposition}
\begin{proof}
Proposition \ref{p:balayagesimple} and Proposition \ref{p:alphaconjalwayslessthan1} 
imply that $h_{\beta_1(\alpha)}$ 
is  $\partial B$-determined harmonic function of $Y$ for $\alpha \in S^1.$
$\boldsymbol \alpha$ and $\beta_1$ are continuous functions. This, compactness
of $S^1$ and \eqref{e:boundonbeta1} imply
\[
|\beta_1(\alpha)|,|\boldsymbol\alpha(\beta_1(\alpha),\alpha)| < 1
\]
for $|\alpha| \in (R,1]$ where $R$ is sufficiently close to $1$. This and 
Proposition \ref{p:balayagesimple} now imply the statement of the proposition.
\end{proof}

\section{Analysis of $Y$, $d>2$}\label{s:dg2}
Now we would like to extend some of the ideas of the previous section to $d$ dimensions.
Our path will be this: each of the graphs shown in Figure
\ref{f:graphsharmtex} and its corresponding equations define a harmonic function of $Y$.
We will develop the graph representation in $d$ dimensions
and show that any solution of the equations represented by a certain class of
graphs defines a harmonic function of $Y$.

For $\alpha \in {\mathbb C}^{d-1}$ we will index the components
of the vector $\alpha$ with the set 
${\mathcal N} \doteq {\mathcal N}_0 - \{0,1\}$, i.e.,
$\alpha = (\alpha(2),\alpha(3),...,\alpha(d)).$
The members of
${\mathcal N}$ are exactly the constrained coordinates of $Y$.
For $a \subset {\mathcal N}$
the constraining map $\pi$ sets the following set of increments 
of $Y$ to $0$ on $\partial_a$:
\[
{\mathcal V}_a \doteq \{v: v\in {\mathcal V}, v(j) \neq -1, \text{ for } j \in a\}.
\]
Rewrite \eqref{e:linear} as
\begin{equation}\label{e:allconditions}
V(y) = 
V(y) \sum_{j \in a} \mu_j + 
\sum_{v \in {\mathcal V}_a} V(y + v) p(v), 
y \in \partial_a \cap O, a \subset {\mathcal N}.
\end{equation}

Set 
\begin{equation}\label{d:basicloglinear}
[(\beta,\alpha),z] \doteq \beta^{z(1)- \sum_{j \in {\mathcal N}}z(j) } 
\prod_{j \in {\mathcal N}} \alpha(j)^{z(j)}.
\end{equation}
$[(\beta,\alpha),z]$ is $\log$-linear in $z$, i.e., $\log([(\beta,\alpha),z])$ is 
linear in $z$; our goal is to construct $Y$-harmonic functions out of
linear combinations of these functions.

Define the characteristic polynomial
\begin{equation}\label{e:hamiltonianba}
{\mathbf p}_a(\beta,\alpha)
\doteq 
\left( \sum_{ v\in {\mathcal V}_a  } p(v) 
[(\beta,\alpha) ,v ]
+ 
\sum_{j \in a} \mu_j\right)
\end{equation}
the characteristic equation
\begin{equation}\label{e:chareqgenerald}
{\mathbf p}_a(\beta,\alpha) = 1,
\end{equation}
and the characteristic surface
\begin{equation}\label{d:DefHcal}
{\mathcal H}_a \doteq \{ (\beta,\alpha) \in {\mathbb C}^d: 
{\mathbf p}_a(\beta,\alpha) = 1\}
\end{equation}
of the boundary $\partial_a$, $ a \subset{\mathcal N}$.
We will write ${\mathbf p}$ and ${\mathcal H}$ 
instead of ${\mathbf p}_\emptyset$ and ${\mathcal H}_\emptyset.$

Generalize \eqref{e:intermsofPY} to the current setup as
\begin{equation}\label{e:intermsofPyd0}
{\mathbf p}_a(\beta,\alpha)  = 
{\mathbf p}(\beta,\alpha) - 
\left( \sum_{ v \in {\mathcal V}_a^c } p(v) [(\beta,\alpha),v]
- \sum_{j \in a} \mu_j\right).
\end{equation}

${\mathbf p}_a$ is not a polynomial but a rational function;
to make it a polynomial one must multiply it by $\beta \prod_{j \in {\mathcal N}} \alpha(j)$; this is what we did when we gave the two dimensional versions
of these definitions
in \eqref{e:hamiltonian}
and \eqref{e:hamiltonianb2}. For $d >2 $, the
$\beta \prod_{j \in {\mathcal N}}\alpha(j)$ multiplier complicates notation; 
for this reason we omit it but continue to refer to the rational
\eqref{e:hamiltonianba} as the ``characteristic polynomial.''

The argument in subsection \ref{ss:asingleterm} continues to work verbatim
for general $d$ and gives
\begin{lemma}\label{l:singletermgen}
Suppose $(\beta,\alpha) \in {\mathcal H} \cap {\mathcal H}_i$.
Then $[ (\beta,\alpha),\cdot]$ is $Z$-harmonic (i.e., $Y$-harmonic on
$\Omega_Y^o$) and $Y$-harmonic on $\partial_i.$
\end{lemma}

Our next step is to extend the content of subsection \ref{ss:twoterms} to the current
setup.
Begin with the operator $D_a$ acting on functions on 
${\mathbb Z}^d$ and giving functions on $\partial_a$:
\begin{align*}
&D_aV = g,~~~~ V: {\mathbb Z}^d \rightarrow {\mathbb C},\\
& g( z)
\doteq
\left(\sum_{j \in a} \mu_a  +  \sum_{ v\in {\mathcal V_a} } p(v) V( z + v) \right)
- V(z).
\end{align*}
\begin{lemma}\label{l:Daeq0}
 $D_a V = 0$ 
if and only if $V$ is $Y$-harmonic on $\partial_a$.
\end{lemma}
The proof follows from the definitions.
Next generalize $C$ of \eqref{e:defC2d} to
\begin{equation}\label{e:Cjbetaalpha}
C(j,\beta,\alpha) \doteq 
\mu_j
-\sum_{v \in{\mathcal V}, v(j)=-1} p(v) [ (\beta,\alpha),v ].
\end{equation}

For $\alpha \in {\mathbb C}^{{\mathcal N}}$ and
$a \subset {\mathcal N}$ define
$\alpha\{a\} \in {\mathbb C}^{{\mathcal N}}$ 
as follows: $\alpha\{a\}|_{a^c} = \alpha|_{a^c}$ and
$\alpha\{a\}|_{a} = 1$;
If $a = \{i\}$, we will write $\alpha\{i\}$, instead of
$\alpha\{ \{i\} \}.$
 For example, for ${\mathcal N}= \{2,3,4\}$,
$a = \{4\}$
and $\alpha = (0.2,0.3,0.4)$, $\alpha\{a\} = (0.2,0.3,1).$ 
We will use this notation in the next paragraph, where we define
the conjugacy of points on ${\mathcal H}$ in $d$ dimensions and
in the next section where we apply the results of the present section
to $d$-tandem
queues.

For $i \in {\mathcal N}$, fix $\alpha|_{{\mathcal N}-\{i\}}$, $\beta$
and multiply both sides
of the characteristic equation \eqref{e:chareqgenerald}  by $\alpha(i)$;
this gives a second order polynomial equation in $\alpha(i)$. 
If $\beta$ and $\alpha|_{{\mathcal N} - \{i\}}$ are
such that the discriminant of this polynomial is nonzero,
we get two distinct points $(\beta,\alpha_1)$ and
$(\beta,\alpha_2)$ on ${\mathcal H}$ which satisfy
\begin{align}\label{e:conjugacydg2}
\alpha_1|_{ {\mathcal N}-\{i\}} &= \alpha_2|_{ {\mathcal N}-\{i\}}, \\
\alpha_1(i) + \alpha_2(i) &= 
\frac{1 - \sum_{v(i) = 0} [(\beta,\alpha_1),v]}
{\sum_{v(i) = 1} [(\beta, \alpha_1\{i\}), v]}. 
\label{e:tmpconj}
\end{align}
The sum in the numerator on the right side of \eqref{e:tmpconj}
is over $v$ such that $v(i)=0$; this and \eqref{e:conjugacydg2} imply
that \eqref{e:tmpconj} remains the same if we replace 
$[(\beta,\alpha_1),v]$ in the numerator with
$[(\beta,\alpha_2),v]$ or
$[(\beta,\alpha_2\{i\}),v] =[(\beta,\alpha_1\{i\}),v]$.
Rewrite \eqref{e:tmpconj} as
\begin{align}
\alpha_2(i) &= \notag
\frac{1 - \sum_{v(i) = 0} [(\beta,\alpha_1),v]}
{\sum_{v(i) = 1} [(\beta, \alpha_1\{i\}), v]} - \alpha_1(i)
\notag \\
&=\frac{1 - \sum_{v(i) = 0} [(\beta,\alpha_1),v]-\sum_{v(i) = 1} 
[(\beta, \alpha_1), v]} 
{\sum_{v(i) = 1} [(\beta, \alpha_1\{i\}), v]}
\notag \\ 
&=\frac{\sum_{v(i) = -1} [(\beta,\alpha_1),v]} 
{\sum_{v(i) = 1} [(\beta, \alpha_1\{i\}), v]} \label{e:conja1toa2}
=
\frac{1}{\alpha_1(i)} \frac{\sum_{v(i) = -1} [(\beta,\alpha_1\{i\}),v]} 
{\sum_{v(i) = 1} [(\beta, \alpha_1\{i\}), v]} 
\intertext{Now keep $\alpha_2(i)$ on the left and repeat the same
computation to get}
\alpha_1(i) &=
\frac{1}{\alpha_2(i)} \frac{\sum_{v(i) = -1} [(\beta,\alpha_2\{i\}),v]} 
{\sum_{v(i) = 1} [(\beta, \alpha_2\{i\}), v]}.
\label{e:conja2toa1}
\end{align}
We will call $(\beta,\alpha_1) \neq (\beta,\alpha_2) \in {\mathcal H}$
$i$-{\em conjugate} if they satisfy \eqref{e:conjugacydg2} and any of
the equivalent
\eqref{e:tmpconj}, \eqref{e:conja1toa2} and \eqref{e:conja2toa1}.
Based on these, generalize the conjugator ${\boldsymbol \alpha}$ as
${\boldsymbol  \alpha}(i, (\beta,\alpha_1)) = \alpha_2$ where $\alpha_2$
is defined by \eqref{e:conjugacydg2} and \eqref{e:conja1toa2}.

Our next proposition generalizes Proposition \ref{p:harmonicYtwoterms2d}
to the current setup. In its proof the, following decomposition,
which \eqref{e:intermsofPyd0} implies, will be useful:
\begin{align}\label{e:imageofD2shortgd}
D_a( [(\beta,\alpha,\cdot)])(z) &=
\left( \sum_{j \in a} \mu_j
-\sum_{v \in {\mathcal V}_a^c} p(v) [ (\beta,\alpha),v ] \right) 
[ (\beta,\alpha), z]\\
&=
\sum_{j \in a} \left(\notag
\mu_j - \sum_{ v\in {\mathcal V}, v(j)=-1} p(v) [ (\beta,\alpha),v ] \right) 
[ (\beta,\alpha), z]\\
&= \sum_{j \in a} D_j( [(\beta,\alpha,\cdot)])(z) \label{e:Dadecomposed}
\end{align}
for $z \in \partial_a$ and $(\beta,\alpha) \in {\mathcal H}$.

\begin{proposition}\label{p:harmonicYtwoterms}
Suppose that 
$(\beta,\alpha_1)$
and $(\beta,\alpha_2)$ are $i$-conjugate and $C(i,\beta,\alpha_j)$,
$j=1,2$ are well defined.
Then
\begin{equation*}
h_\beta \doteq C(i,\beta,\alpha_2) [(\beta,\alpha_1),\cdot]
- C(i,\beta,\alpha_1) [(\beta,\alpha_2),\cdot]
\end{equation*}
is $Y$-harmonic on $\partial_i$.
\end{proposition}
\begin{proof}
The definition \eqref{e:Cjbetaalpha}  of $C$, \eqref{e:imageofD2shortgd} and 
linearity of $D_i$ imply
\begin{align*}
D_i(h_\beta) &=
C(i,\beta,\alpha_2) C(i,\beta,\alpha_1) 
[(\beta,\alpha_1),\cdot]
- 
C(i,\beta,\alpha_2) C(i,\beta,\alpha_1) 
[(\beta,\alpha_2),\cdot]
\intertext{ \eqref{e:conjugacydg2} implies
$[(\beta,\alpha_1),z] = [(\beta,\alpha_2),z]$ for $z \in \partial_i$
and therefore the last line reduces to}
&=0.
\end{align*}
Lemma \ref{l:Daeq0} now implies that $h_\beta$ is $Y$-harmonic
on $\partial_i$.
\end{proof}

To generalize the graph representation of the previous section
we will need graphs with labeled edges;
let us denote any graph by its adjacency matrix $G$.
Let 
$V_G$, a finite set, 
denote the set of vertices of $G$ .
Each edge of $G$ will have
a label taking values in a finite set $L$.
For two vertices $i \neq j$, $G(i,j) = 0$ if they are
disconnected, and $G(i,j) = l$ if an edge with label 
$l \in L$ connects them; such an edge will be called an $l$-edge. 
As usual, an edge from a vertex to itself is called a loop.
For a vertex $j \in V_G$, 
$G(j,j)$ is the set of the labels of the loops on $j$. Thus $G(j,j)\subset L$ 
is set
valued.
\begin{definition}
Let $G$ and $L$ be as above.
If each vertex $j \in V_G$ has
a unique $l$-edge (perhaps an $l$-loop) 
for all $l \in L$ we will call $G$ edge-complete with respect to
$L$.
\end{definition}

We say $G$ is
edge-complete with respect to $Y$ if it is so with respect
to 
${\mathcal N}={\mathcal N}_0 - \{0,1\}$ (remember that ${\mathcal N}$ is the
set of constrained coordinates of $Y$). If we just say ``edge-complete'' we mean
``edge-complete with respect to $Y$.''

\begin{definition}\label{d:Yharmonic}
A $Y$-harmonic system consists
of an 
edge-complete graph $G$ with respect to ${\mathcal N}$,
the variables
$ (\beta,\alpha_j)  \in {\mathbb C}^d$, ${\boldsymbol c}_j \in {\mathbb C}$,  $j \in V_G$,
and these equations/constraints:
\begin{enumerate}
\item
$(\beta,\alpha_j) \in {\mathcal H}, {\boldsymbol c}_j \in {\mathbb C}-\{0\}, j \in V_G$, 
\item
$\alpha_{i} \neq \alpha_{j}$,  if $i \neq j, i,j  \in V_G$,
\item
$\alpha_{i}, \alpha_{j}$  are $G(i,j)$-conjugate if  
$ G(i,j) \neq 0$, $i \neq j, i,j \in V_G$,
\item 
\begin{equation}
\label{e:correctmultipliers}
{\boldsymbol c}_{i}/{\boldsymbol c}_{j} = 
-\frac{C(G(i,j),\beta,\alpha_j)}{C(G(i,j),\beta,\alpha_i)}, \text{ if } G(i,j) \neq 0,
\end{equation}
\item
$(\beta,\alpha_j) \in {\mathcal H}_l \text{ for all } l \in G(j,j), j \in V_G.$
\end{enumerate}
\end{definition}

\begin{proposition}\label{p:simpleharmonicfunctions}
Suppose that 
a $Y$-harmonic system for and edge-complete $G$ has a solution;
then
\begin{equation}\label{e:defhG}
h_G \doteq \sum_{j \in V_G} {\boldsymbol c}_j [(\beta,\alpha_j),\cdot]
\end{equation}
is a harmonic function of $Y$.
\end{proposition}
\begin{proof}
All summands of $h_G$ are $Z$-harmonic and therefore
$Y$-harmonic on $\Omega_Y^o$ because 
$(\beta,\alpha_j)$, $ j \in V_G$, are all on the characteristic surface
${\mathcal H}$.
It remains to show that $h_G$ is $Y$-harmonic on all $\partial_a \cap \Omega_Y$,
$a \subset {\mathcal N}$ and $a \neq \emptyset.$ We will do this
by induction on $|a|$. Let us start with $|a|=1$, i.e., $a = \{l\}$ for some
$l \in {\mathcal N}$. 
Take any vertex $i \in V_G$; if $l \in G(i,i)$
then $(\beta,\alpha_i) \in {\mathcal H}_l$ and by Lemma \ref{l:singletermgen}
$[(\beta,\alpha_i),\cdot]$ is $Y$-harmonic on $\partial_l$.
Otherwise, the definition of a harmonic system implies that
there exists a unique vertex $j$ of $G$
such that $G(i,j) = l$. This implies, by definition, that
 $(\beta,\alpha_i)$ and $(\beta,\alpha_{j})$
are $l$-conjugate and by Proposition \ref{p:harmonicYtwoterms} 
and \eqref{e:correctmultipliers} 
\[
{\boldsymbol c}_i [(\beta,\alpha_i),\cdot] + {\boldsymbol c}_{j} [ (\beta,\alpha_{j}),\cdot]
\]
is $Y$-harmonic on $\partial_l$. 
Thus, all summands
of $h_G$ are either $Y$-harmonic on $\partial_l$ 
or form pairs which are so;
this implies that the sum $h_G$ is $Y$-harmonic on $\partial_l$.

Now assume $h_G$ is $Y$-harmonic for all
$a'$ with $|a'| = k-1$; fix $a \subset {\mathcal N}$ 
such that $|a| = k$ and $i \in a$;
by \eqref{e:Dadecomposed}
\[
D_a( h_G) = D_{a-\{i\}}(h_G) + D_i(h_G).
\]
The induction assumption and Lemma \ref{l:Daeq0} imply 
that the first term on the right is zero; the same lemma
and the previous paragraph imply the same for the second term.
Then $D_a(h_G) = 0$; this and Lemma \ref{l:Daeq0} finish the
proof of the induction step.
\end{proof}

\begin{proposition}\label{p:balayagesimpled}
Let $(\beta,\alpha_j)$, ${\boldsymbol c}_j$, $j \in V_G$,
be the solutions of a $Y$-harmonic system
for an edge-complete $G$ and let $h_\beta$ be defined as in \eqref{e:defhG}.
If 
\[
|\beta| < 1,~~~|\alpha_j(i)| \le 1,~~ i \in {\mathcal N},
\]
then $h_G$ is $\partial B$-determined.
\end{proposition}
The proof is identical to that of Proposition \ref{p:balayagesimple}.

\subsection{Simple extensions}
One can build, from the
solution of a given harmonic system for $Y$,
solutions to related harmonic systems for higher dimensional walks which are, 
in some natural sense, extensions of $Y$. 
The construction will depend on what we mean by an ``extension.''
One possibility is that of a {\em simple extension} whose
definition follows.

So far, we have taken $p$  to be a $(d+1) \times (d+1)$ real matrix, i.e.,
$p  \in {\mathbb R}_+^{(d+1) \times (d+1)}$ where $d+1 = |{\mathcal N}_0|$.
To define ``simple extensions'' it is more convenient to take the
set of nodes
${\mathcal N}_0$  to be an arbitrary set with $d+1$ elements
containing $0$
and $ p \in {\mathbb R}_+^{{\mathcal N}_0\times {\mathcal N}_0}$.
${\mathcal N}_+$ is, as before, ${\mathcal N}_0 - \{0\}.$
Suppose ${\mathcal N}_0^1 \supset {\mathcal N}_0$ and
$p_1 \in {\mathbb R}_+^{{\mathcal N}_0^1  \times {\mathcal N}_0^1}$ is
another matrix of jump probabilities.
Define
$p' \in {\mathbb R}_+^{{\mathcal N}_0  \times {\mathcal N}_0}$ as follows
\begin{align}
p'(i,j) &= p_1(i,j) \label{e:defppij}
\intertext{if $i\in {\mathcal N}_0$, $j \in {\mathcal N}_+$,}
\label{e:defppi0}
p'(i,0) &= p_1(i,0) + \sum_{j \in {\mathcal N}_0^1 -{\mathcal N}_0}
p_1(i,j).
\end{align}
\begin{definition}\label{d:simpleextension}
We say that $p_1$ is a simple extension of $p$ if 
\begin{align}
&p' =  \left(\sum_{i,j \in {\mathcal N}_0} p'(i,j)\right) p, p' \neq 0,
\label{d:pandp1}
\\
&p_1(i,j) = 0 \text{ if }i \in {\mathcal N}^1_0 - {\mathcal N}_0 \text{ and }
j \in {\mathcal N}_+.\label{e:constraintonp1}
\end{align}
\end{definition}

An example:
take ${\mathcal N}_0 = \{0,1,2\}$, ${\mathcal N}^1_0 = \{0,1,2,3,4\}$,
and
\[
p = \left( \begin{matrix}
	0 & 1/7 & 0 \\
	0 &  0      & 4/7 \\
	2/7 & 0   & 0 
\end{matrix}
\right), 
p_1 = \left( \begin{matrix}
	  0       & 0.05   & 0     & 0 &  0.02\\
	  0       & 0         & 0.2 & 0 & 0         \\
	  0.1 & 0         & 0    & 0.1 & 0  \\
	 0        & 0         & 0    & 0        & 0.25\\
	 0.1 & 0        & 0    & 0.18  & 0 
	\end{matrix}
	\right).
\]
Figure \ref{f:simpleextensionextex} shows the topologies of the networks
corresponding to $p$ and $p_1.$

\begin{figure}[h]
\begin{center}
\scalebox{0.8}{
\centerline{\input{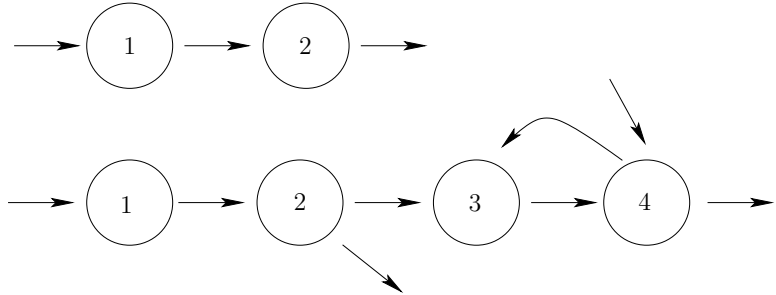}}}
\end{center}

\vspace{-0.65cm}
\caption{\hspace{0.25cm}Two networks, second is a simple extension of the first\label{f:simpleextensionextex}}
\end{figure}

Next define the ``edge-complete extension'' of a given edge-complete graph:
\begin{definition}\label{d:simpleextensionG}
Let $G$ be 
an edge-complete graph with respect to 
a finite label set $L$ . Its
edge-complete extension $G_1$ with respect to 
another set of nodes ${L}_1 \supset L$ is 
defined as follows: $V_{G_1} = V_G$
and
\begin{align*}
G_1(i,j) &= G(i,j),  i \neq j,  i,j\in V_G\\
G_1(j,j) &= G(j,j) \cup (L_1 - L), j \in V_G.
\end{align*}
\end{definition}
To get $G_1$ from $G$ one adds to each vertex of $G$ an $l$-loop for each
$l \in L_1 - L$. Then if $G$ is edge-complete with respect
to $L$, so must be
$G_1$ with respect to $L_1.$ Figure \ref{f:simpleextendGtex} gives an
example.

\begin{figure}[h]
\begin{center}
\scalebox{0.8}{
\centerline{\input{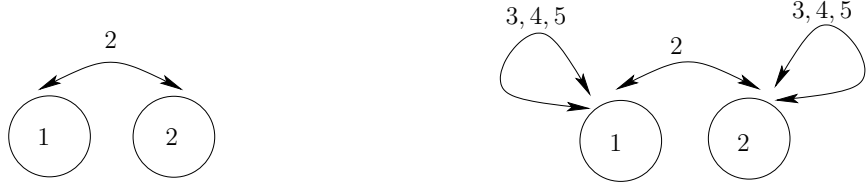}}}
\end{center}

\vspace{-0.65cm}
\caption{\hspace{0.25cm}An edge-complete graph with respect to $\{2\}$ and its edge-complete extension 
to $\{2,3,4,5\}$\label{f:simpleextendGtex}}
\end{figure}

Suppose $Y^1$ is a simple extension of $Y$; the next proposition explains
how one can construct harmonic systems (and their solutions) for $Y^1$
from those of $Y$.
\begin{proposition}\label{p:invarianceundersimpleextensions}
Let $Y^1$ be another constrained process defined using the construction
\eqref{e:defY} 
(in particular $i=1$ and only the first coordinate of $Y^1$ is unconstrained
and its remaining coordinates are constrained)
and such that its matrix of jump probabilities 
$p_1$ is a simple extension of the $p$ matrix of $Y$. 
Let $G_1$ and
$G_0$ be edge-complete graphs for $Y^1$ and $Y$ such that $G_1$
is an edge-complete extension of $G_0$. 
Suppose $(\beta,\alpha_k), {\boldsymbol c}_k, k \in V_{G_0}$,
solve the harmonic system associated with $G_0$. 
For $k \in V_{G_1} = V_{G_0}$ define
$\alpha^1_k$ as follows
\begin{align}
\label{e:defa1p1}
\alpha^1_k|_{{\mathcal N}}  &= \alpha_k, \\
\label{e:defa1p2}
\alpha^1_k|_{{\mathcal N}^1-{\mathcal N}} &= \beta,
\end{align}
where ${\mathcal N}^1 \supset {\mathcal N}$ is the set of constrained
coordinates of the process $Y^1$.
Then $(\beta,\alpha^1_k), {\boldsymbol c}_k, k \in V_{G_1}$,
solve the harmonic system
defined by  $G_1$.
\end{proposition}
The definition \eqref{e:defa1p2} assigns
the value $\beta$
to the new components of
$\alpha^1_k$ coming from the new dimensions of the simple extension; 
this corresponds to ignoring the new dimensions when we compute
the $\log$-linear function $[(\beta,\alpha^1_k),\cdot]$ at the increments
of $Y^1$
(see \eqref{e:reduction1} and \eqref{e:ba1v1p3} below).
The following proof lays down the details of this observation.
\begin{proof}
Set ${\mathcal N}^1_0 = {\mathcal N}^1 \cup \{0,1\}$ and
${\mathcal N}^1_+ = {\mathcal N}^1 \cup \{1\}$.
By assumption, $(\beta,\alpha_k)$, ${\boldsymbol c}_k$,
$k \in V_{G_0}$, satisfy the five conditions
listed under Definition \ref{d:Yharmonic} for $G=G_0$.
We want to show that this implies that
the same holds for $(\beta,\alpha^1_k)$, ${\boldsymbol c}_k$, $k \in V_{G_1}$  
for $G=G_1.$ Let ${\mathcal V}^1_0$ denote the set of increments of
$Y^1$,
$e^1_j$, $j \in {\mathcal N}^1_+$ the unit functions on
${\mathcal N}^1_+$ and let $e^1_0$ be the $0$ function on the same set.
\eqref{e:constraintonp1} implies that we can
partition ${\mathcal V}^1_0$ as follows:
\begin{align}\label{e:partitionV10}
{\mathcal V}^1_0 &= {\mathcal V}^1_1 \cup {\mathcal V}^1_2 \cup {\mathcal V}^1_3, \\
{\mathcal V}^1_1 &\doteq \{ e_1^1 + e_j^1, -e_i^1 -e_1^1, -e_i^1 + e_j^1,
i \in {\mathcal N} \cup \{0\}, j \in {\mathcal N} \},
\notag\\
{\mathcal V}^1_2 &\doteq \{ e_1^1 + e_j^1, -e_i^1 + e_j^1, i \in {\mathcal N}, 
j \in ({\mathcal N}_1 - {\mathcal N})\cup \{0\} \},
\notag
\\
{\mathcal V}^1_3 
&\doteq \{ -e^1_i + e^1_j, i,j \in ({\mathcal N}_1 - {\mathcal N})\cup\{0\}\}.
\notag
\end{align}
Parallel to this is the following partition of ${\mathcal V}_0$:
\begin{align*}
{\mathcal V}_0 &= {\mathcal V}_1 \cup {\mathcal V}_2, \\
{\mathcal V}_1 &\doteq \{ e_1 + e_j, -e_i -e_1, -e_i + e_j, 
i \in {\mathcal N} \cup \{0\}, j \in {\mathcal N} \},\\
{\mathcal V}_2 &\doteq \{e_1, -e_i, i \in {\mathcal N} \}.
\end{align*}

Fix any $k \in V_{G_1}$;
\eqref{e:defa1p1} and \eqref{e:defa1p2} imply
\begin{align}\label{e:reduction1}
[(\beta,\alpha^1_k),v^1 ] &= [(\beta,\alpha_k),v ]
\end{align}
for all $v^1 \in {\mathcal V}_1^1 \cup {\mathcal V}_2^1$ 
and where $v = v^1 |_{{\mathcal N}}.$
\eqref{e:defa1p2} implies
\begin{equation}\label{e:ba1v1p3}
[(\beta,\alpha^1_k), v^1 ] = 1
\end{equation}
for
$v^1 \in {\mathcal V}^1_3$.
Let ${\mathbf p}^1$ denote the characteristic polynomial of $Y^1$ and 
let ${\mathcal H}^1$ denote its characteristic surface; 
we would like to show $(\beta,\alpha_k^1) \in {\mathcal H}^1$, 
i.e., ${\mathbf p}^1(\beta,\alpha_k^1) = 1.$
By \eqref{e:constraintonp1} and \eqref{e:partitionV10}
\begin{align}\label{e:computePY1}
{\mathbf p}^{1}(\beta,\alpha^1_k) &= 
\sum_{v^1 \in {\mathcal V}_1^1 } 
p_1(v^1) [(\beta,\alpha^1_k), v^1 ] + 
\sum_{v^1 \in {\mathcal V}^1_2 } 
p_1(v^1) [(\beta,\alpha^1_k), v^1 ] \\
&~~~+ 
\sum_{v^1 \in {\mathcal V}^1_3 } 
p_1(v^1) [(\beta,\alpha^1_k), v^1 ].\notag
\intertext{ \eqref{e:defppij} and \eqref{e:reduction1} imply
}
&=
\sum_{v \in {\mathcal V}_1 } 
p'(v) [(\beta,\alpha_k), v ] + 
\sum_{v^1 \in {\mathcal V}^1_2 } \label{e:threetermsum}
p_1(v^1) [(\beta,\alpha^1_k), v^1 ] 
+ 
\sum_{v^1 \in {\mathcal V}^1_3 } p_1(v^1).
\end{align}
For any $v^1
\in {\mathcal V}^1_2$, 
\eqref{e:defa1p2} implies
\begin{equation}\label{e:secondpart}
[(\beta,\alpha_k^1), v^1 ] 
 = [(\beta,\alpha_k), v]
\end{equation}
where $v = v^1|_{\mathcal N} $
;
this implies that the second sum on the right side of \eqref{e:threetermsum}
equals
\[
\sum_{v \in {\mathcal V}_2} p'(v) [(\beta,\alpha_k), v].
\]
Substitute this back in \eqref{e:threetermsum} to get
\begin{align*}
{\mathbf p}^{1}(\beta,\alpha^1_k) &= 
\sum_{v \in {\mathcal V} } 
p'(v) [(\beta,\alpha_k), v ] + \sum_{v^1 \in {\mathcal V}_3^1 } p_1(v^1)
\intertext{which, by \eqref{d:pandp1}, equals}
& = \left(\sum_{v \in {\mathcal V} } p'(v) \right)
\sum_{v \in {\mathcal V}}p(v)[(\beta,\alpha_k), v ] 
+ \sum_{v^1 \in {\mathcal V}_3^1 } p_1(v^1)
\intertext{$(\beta,\alpha_k) \in {\mathcal H}$, \eqref{e:defppij} and \eqref{e:defppi0}
now give}
& = \sum_{v^1 \in {\mathcal V}_1^1 \cup {\mathcal V}_2^1 } p_1(v^1)
+ \sum_{v^1 \in {\mathcal V}^1_3 } p_1(v^1) = 1,
\end{align*}
i.e., indeed, $(\beta,\alpha_k^1) \in {\mathcal H}^1.$

Let us now show that the third part of the same definition is also satisfied.
Fix any $i\neq j$ with $G_1(i,j) =l \in {\mathcal N}.$
We want to show that $(\beta,\alpha_{i}^1)$ and
$(\beta,\alpha_{j}^1)$ are $l$-conjugate, i.e., that they satisfy
\eqref{e:conjugacydg2} and \eqref{e:conja1toa2}:
\begin{equation}\label{e:conjugacydg2Y1}
\alpha_{i}^1|_{{\mathcal N}^1 - \{l\}}
=
\alpha_{j}^1|_{{\mathcal N}^1 - \{l\}},
\end{equation}
\begin{equation}\label{e:conja1toa2Y1}
\alpha_{j}^1(l) =
\frac{1}{\alpha_i^1(l)}
\frac{\sum_{v^1(l) = -1} p_1(v) [(\beta,\alpha_i^1\{l\}),v^1]} 
{\sum_{v^1(l) = 1} p_1(v) [(\beta, \alpha_i^1\{l\}), v^1]}.
\end{equation}
By definition
$G_1(i,j)=l$
when $G(i,j) =l$; and this implies, again by definition, that
$\alpha_i$ and $\alpha_j$ satisfy \eqref{e:conjugacydg2} and \eqref{e:conja1toa2}.
\eqref{e:conjugacydg2Y1} follows from \eqref{e:conjugacydg2} for $\alpha_i$ and
$\alpha_j$, \eqref{e:defa1p1} and \eqref{e:defa1p2}.
For $l \in {\mathcal N}$ $\alpha_j^1(l) = \alpha_j(l)$ and $\alpha_i^1(l) = \alpha_i(l)$.
Then to prove \eqref{e:conja1toa2Y1} it suffices to prove
\begin{equation}\label{e:ratiosimplified}
\frac{\sum_{v^1(l) = -1} p_1(v) [(\beta,\alpha_i^1\{l\}),v^1]} 
{\sum_{v^1(l) = 1} p_1(v) [(\beta, \alpha_i^1\{l\}), v^1]}
=\frac{\sum_{v(l) = -1} p(v) [(\beta,\alpha_i\{l\}),v]} 
{\sum_{v^1(l) = 1} p_1(v) [(\beta, \alpha_i\{l\}), v]}.
\end{equation}
This follows from a decomposition parallel to the one given for the previous
part: let us first apply it to the numerator.
\begin{align}
& \sum_{v^1(l) = -1} p_1(v) [(\beta,\alpha_i^1\{l\}),v^1] \notag \\
&~~~=\sum_{v^1(l) = -1, v^1 \in {\mathcal V}^1_1 } p_1(v) [(\beta,\alpha_i^1\{l\}),v^1]
+
\sum_{v^1(l) = -1, v^1 \in {\mathcal V}^1_2 } p_1(v) [(\beta,\alpha_i^1\{l\}),v^1]
\notag
\intertext{\eqref{e:reduction1} and \eqref{e:secondpart} imply}
&~~~=\sum_{v(l) = -1, v \in {\mathcal V}_1 } p'(v) [(\beta,\alpha_i\{l\}),v]
+
\sum_{v(l) = -1, v \in {\mathcal V}_2 } p_1(v) [(\beta,\alpha_i\{l\}),v]
\notag \\
&~~~= \left(\sum_{v \in {\mathcal V} } p'(v) \right) 
\sum_{v(l) = -1} p(v) [(\beta,\alpha_i\{l\}),v].\label{e:numeratorsimplified}
\end{align}
A parallel argument for the denominator gives
\[
\sum_{v^1(l) = 1} p_1(v) [(\beta,\alpha_i^1\{l\}),v^1] 
=
 \left(\sum_{v \in {\mathcal V} } p'(v) \right) \sum_{v(l) = 1} 
p(v)[(\beta,\alpha_i\{l\}),v].
\]
Dividing \eqref{e:numeratorsimplified} by the last equality gives
\eqref{e:ratiosimplified}.

The proof that the remaining parts of the definition holds for $G=G_1$
is parallel to the arguments just given and is omitted.
\end{proof}

\section{Harmonic Systems for Tandem Queues}\label{s:hstq}
Throughout this section we will denote the dimension of the system
with ${\boldsymbol d}$; the arguments below for
${\boldsymbol d}$ dimensions require the consideration
of all walks with dimension $d \le {\boldsymbol d}$.

We will now define a specific sequence of edge-complete graphs
for tandem walks
and construct a particular solution to the harmonic system defined
by these graphs.
These particular solutions
will give us an exact formula  for $P_y(\tau < \infty)$ in terms of
the superposition of a finite number of $\log$-linear $Y$-harmonic functions.

We will assume 
\begin{equation}\label{e:separatemu}
\mu_i \neq \mu_j, i \neq j.
\end{equation}
This is the analog of \eqref{e:conjnotharmonic} for the $\boldsymbol d$ dimensional
tandem walk. 
One can treat parameter values which violate \eqref{e:separatemu} by taking limits
of the results of the present section, we give several examples in subsection \ref{ss:harmpoly}.

The $p$ matrix of the tandem walk is as given in \eqref{e:poftandem}.
Then its characteristic polynomials will be of the form
\begin{align}\label{e:charpoltandem}
{\mathbf p}(\beta,\alpha) &= 
\lambda \frac{1}{\beta} + 
\mu_1 \alpha(2) + \sum_{j=2}^{{\boldsymbol d}} \mu_j\frac{\alpha(j+1)}{\alpha(j)},\\
{\mathbf p}_i(\beta,\alpha) &= 
\lambda \frac{1}{\beta} + 
\mu_1 \alpha(2) + \mu_i +  \sum_{j=2, j \neq i}^{{\boldsymbol d}} \mu_j\frac{\alpha(j+1)}{\alpha(j)},
\notag
\end{align}
where by convention $\alpha({\boldsymbol d}+1) = \beta$ (this convention
will be used throughout this section, and in particular, 
in Lemma \ref{l:conditionforintersection}, \eqref{e:conjugationDtandem} and
\eqref{e:Cjbetaalphatandem}).
The formula \eqref{e:charpoltandem} for ${\mathbf p}$ implies 
\begin{lemma}\label{l:conditionforintersection}
$(\beta,\alpha) \in {\mathcal H} \cap {\mathcal H}_j$
$\iff$
$(\beta,\alpha) \in {\mathcal H}$,
$\mu_j \frac{\alpha(j+1)}{\alpha(j)} = \mu_j$
$\iff$
$(\beta,\alpha) \in {\mathcal H}$,
$\alpha(j+1) = \alpha(j)$, $j \in {\mathcal N}.$
\end{lemma}

The conjugators for the ${\boldsymbol d}$-tandem walk are:
\begin{align}
\boldsymbol\alpha(l, (\beta,\alpha))(l) &= 
\begin{cases}
\frac{1}{\alpha(2)} \frac{\alpha(3) \mu_2}{\mu_1}, ~~&l = 2,\\
\frac{1}{\alpha(l)} \frac{\alpha(l-1) \alpha(l+1) \mu_l}{\mu_{l-1}}, ~~ &2 < l \le \boldsymbol{d},
\end{cases}
\label{e:conjugationDtandem}
\\
\boldsymbol\alpha(l, (\beta,\alpha))|_{{\mathcal N} - \{l\}} &= \alpha|_{{\mathcal N} -\{l\}}.
\notag
\end{align}

For tandem walks, the functions $C(j, \beta,\alpha)$ of 
\eqref{e:Cjbetaalpha} reduce to 
\begin{equation}\label{e:Cjbetaalphatandem}
C(j,\beta,\alpha) = 
\mu_j - \mu_j \frac{ \alpha(j+1)}{\alpha(j)}.
\end{equation}

We define the edge-complete graphs $G_d$, $ d \in \{1,2,3,...,{\boldsymbol d}\}$:
\begin{equation}\label{d:defVGd}
V_{G_d} = \{ a\cup\{ d\}, a \subset \{1,2,3,...,d-1\}\};
\end{equation}
for $j \in (a \cup \{d\}) \cap {\mathcal N}$ define $G_d$ by
\begin{align}\label{d:defGd}
G_d( a\cup\{d\}, a\cup \{d\} \cup \{j-1\}) &= j \text{ if } j-1 \notin a 
\end{align}
and
\begin{equation}\label{e:Gdloops}
G_d( a\cup \{d\}, a \cup \{d\}) = {\mathcal N} - a \cup \{d\};
\end{equation}
these and its symmetry determine $G_d$ completely.
Figure \ref{f:G4tex} shows the graph $G_4$ for ${\boldsymbol d}=4$.

\begin{figure}[h]
\begin{center}
\scalebox{0.8}{
\centerline{\input{G4tex}}}
\end{center}

\vspace{-0.65cm}
\caption{\hspace{0.25cm}$G_4$ for ${\boldsymbol d}=4$\label{f:G4tex}}
\end{figure}

The next proposition follows directly from the above definition:
\begin{proposition}\label{p:embedding}
One can represent $G_{d+1}$ as a disjoint union of the
graphs $G_{k}, k=1,2,..,d,$ and the vertex $\{d+1\}$ as follows:
for $a \subset \{1,2,3,...,k-1\}$ map the vertex
$a \cup \{k\}$ of $G_{k}$ to vertex $a \cup \{k, d+1\}$ of $G_{d+1}.$
This maps $G_{k}$ to the subgraph $G_{d+1}^{k}$ of $G_{d+1}$ 
consisting of the vertices $\{a, k,d+1\}$, 
$a\subset \{1,2,3,...,k-1\}$. 
The same map preserves 
the edge
structure of $G_{k}$ as well except for the $d+1$-loops. These loops
on $G_{k}$ are broken and are mapped to $d+1$-edges between
$G^{k}_{d+1}$ and $G^{d}_{d+1}$. 
\end{proposition}
Figure \ref{f:G4tex} shows an example of the
decomposition described in Proposition \ref{p:embedding}.

Define
\begin{align}\label{e:solutiontandem}
{\boldsymbol c}^*_{a} &\doteq (-1)^{|a|-1} \prod_{j=1}^{|a|-1} \prod_{l = a(j) + 1}^{a(j+1)}
\frac{\mu_l - \lambda}{\mu_l-\mu_{a(j)}}\\
\alpha^*_{a}(l) &\doteq
\begin{cases}
	1 &\text{ if } l \le a(1)\\
	\rho_{a(j)}, &\text{ if } a(j) < l \le a(j+1),\\
	\rho_{a(|a|)} & \text{ if } l > a(|a|),
\end{cases}
\label{e:defalphastar}\\
\beta^*_{a}&\doteq \rho_{a(|a|)}, \notag
\end{align}
$l \in {\mathcal N}$ (remember that we assume that the elements
of sets are written in increasing order; $a(|a|)$ then denotes
the largest element in the set). Several examples with
${\boldsymbol d} = 8$:
\begin{align}\label{e:examplescstaralphastar}
{\boldsymbol c}^*_{\{5\}}&= 1,~~
\alpha^*_{\{5\}} = (1,1,1,1,\rho_5,\rho_5,\rho_5),\notag \\
{\boldsymbol c}^*_{\{3,6\}} &= - \frac{ \mu_4 -\lambda}{\mu_4 - \mu_3} \frac{\mu_5 -
\lambda}{\mu_5-\mu_3} \frac{\mu_6-\lambda}{\mu_6-\mu_3},
\alpha^*_{\{3,6\}} = (1,1,\rho_3,\rho_3,\rho_3,\rho_6,\rho_6),\notag \\
{\boldsymbol c}^*_{\{3,5,7\}} &
= (-1)^2 \frac{\mu_4 - \lambda}{\mu_4 -\mu_3} \frac{\mu_5-\lambda}{\mu_5-\mu_3}
\frac{\mu_6-\lambda}{\mu_6-\mu_5} \frac{\mu_7-\lambda}{\mu_7-\mu_5}, \\
\alpha^*_{\{3,5,7\}} &= (1,1,\rho_3,\rho_3,\rho_5,\rho_5,\rho_7),\notag \\
\alpha^*_{\{3\}} &= (1,1,\rho_3,\rho_3,\rho_3,\rho_3,\rho_3), \alpha^*_{\{8\}} = (1,1,1,1,1,1,1);\notag
\end{align}
remember that we index the components of $\alpha^*$ with ${\mathcal N}$; therefore,
e.g., the first $1$ on the right side of the last line is $\alpha^*_{\{8\}}(2).$

It follows from \eqref{e:solutiontandem} 
and \eqref{e:defalphastar} that
\begin{align*}
{\boldsymbol c}^*_{a \cup \{d_1,d_2\}} &=-{\boldsymbol c}^*_{a \cup \{d_1\}} 
\prod_{l=d_1+1}^{d_2} \frac{ \mu_l - \lambda}{\mu_l - \mu_{d_1}}
\\
\alpha^*_{a \cup \{d_1\}} &= \alpha^*_{a \cup \{d_1,{\boldsymbol d}\}}
\end{align*}
for any 
$a(|a|) < d_1 < d_2 \in {\mathcal N}$ and $a \subset {\mathcal N}$.
These and Proposition \ref{p:embedding} imply
\begin{proposition}
For $d < {\boldsymbol d}$ and $y \in \partial B$
\begin{equation}\label{e:embeddingfunc}
-\left(\prod_{l=d+1}^{{\boldsymbol d}} \frac{ \mu_l - \lambda}{\mu_l - \mu_d}\right)
\sum_{ a \in V_{G_d} } {\boldsymbol c}_a^* [ (\beta_a^*, \alpha_a^*), y ]
=
\sum_{ a \in V_{G_{\boldsymbol d}^d }} {\boldsymbol c}_a^* [ (\beta_a^*, \alpha_a^*), y ]
\end{equation}
\end{proposition}

\begin{proposition}\label{p:provesolution}
For $d \le \boldsymbol d$, let
$G_d$ be as in \eqref{d:defVGd} and \eqref{d:defGd}. Then
$(\beta^*_{a\cup\{d\}}, \alpha^*_{a\cup\{d\}})$,
${\boldsymbol c}^*_{a\cup\{d\}}$, 
 $a \subset \{1,2,3,...,d-1\}$, defined in \eqref{e:solutiontandem},
solve the harmonic system defined
by $G_d$.
\end{proposition}
\begin{proof}
The first $d$ components of a tandem walk is a simple extension
of the tandem walk consisting of its first $d-1$ components.
This and Proposition \ref{p:invarianceundersimpleextensions}
imply that it suffices to prove the current proposition only
for $d=\boldsymbol d$.

Let us begin by showing 
$\left(\rho_{\boldsymbol d}, \alpha^*_{a\cup\{\boldsymbol d\}}\right)$, 
$a \subset {\mathcal N}_+-\{{\boldsymbol d}\}$ is on the characteristic surface
${\mathcal H}$ of the tandem walk. We will write $\alpha^*$ instead
of $\alpha^*_{a\cup\{\boldsymbol d\}}$, the set $a$ will be clear from the context.

Let us first consider the case when $a(1) > 1$, i.e., when $1 \notin a$;
the opposite case is treated similarly and is left to the reader.
Then $\alpha^*(l) = 1$ for $2 \le l \le a(1)$.
By definition $\alpha^*(i) = \alpha^*(i+1)$ if $a(j) < i < a(j+1)$;
these and $\beta^*_{a\cup\{{\boldsymbol d}\}} = \rho_{\boldsymbol d}$ give
\begin{align*}
{\mathbf p}(\rho_{\boldsymbol d},\alpha^*) &= 
\mu_{\boldsymbol d} + 
\sum_{j=1}^{a(1)-1} \mu_j  + \mu_{a(1)} \rho_{a(1)}  +
 \sum_{j \in (a^c -\{1\cdots a(1)-1 \})} \mu_j\\
&~~~~~+ \sum_{j \in (a-\{a(1)\}) } \mu_j \frac{\alpha^*(j+1)}{\alpha^*(j)}
+ \rho_{\boldsymbol d} \frac{\mu_{\boldsymbol d}}{\alpha^*(\boldsymbol d)}
\intertext{(where
$a^c= ({\mathcal N}_+ -\{\boldsymbol d\}) - a$)
and in the last expression we have used the convention
$\alpha^*({\boldsymbol d} + 1) =\beta^*$;
by definition \eqref{e:defalphastar} 
$\alpha^*(a(j+1)) = \rho_{a(j)}$,
$\alpha^*(a(j)) = \rho_{a(j-1)}$ and therefore}
&= 
\mu_{\boldsymbol d} + 
\sum_{j=1}^{a(1)-1} \mu_j  + \lambda +  
 \sum_{j\in (a^c -\{1\cdots a(1)-1 \})} \mu_j
+ \sum_{j=2}^{|a|} \mu_{a(j)} \frac{\rho_{a(j)}}{\rho_{a(j-1)}}
+ \mu_{a(|a|)}
\intertext{
$\mu_{a(j)} \rho_{a(j)}/\rho_{a(j-1)} = \mu_{a(j-1)}$ implies}
&= 
\mu_{\boldsymbol d} + 
\sum_{j=1}^{a(1)-1} \mu_j  + \lambda +  
 \sum_{j\in (a^c -\{1\cdots a(1)-1\})} \mu_j
+ \sum_{j=2}^{|a|} \mu_{a(j-1)} + \mu_{a(|a|)}\\
&= 
\mu_{\boldsymbol d} + 
\sum_{j=1}^{a(1)-1} \mu_j  + \lambda +  
 \sum_{j\in (a^c -\{1\cdots a(1)-1 \})} \mu_j+ \sum_{j \in a }\mu_j = 1;
\end{align*}
i.e., $(\rho_{\boldsymbol d}, \alpha^*) \in {\mathcal H}.$

If $a_1 \neq a_2$ 
take any $ i \in a_1 -a_2$ (relabel the sets 
if necessary so that  $a_1 -a_2 \neq  \emptyset$).
Let $j$ be the index of $i$ in $a_1$, i.e., $a_1(j) = i$.
Then by definition, $\alpha^*_{a_1\cup\{\boldsymbol d\}}(j+1) = \rho_i$; but
$i \notin a_2$ and \eqref{e:separatemu} imply that no component
of $\alpha^*_{a_1\cup \{\boldsymbol d\}}$ equals $\rho_i$, and therefore
$\alpha^*_{a_1\cup\{\boldsymbol d\}} \neq \alpha^*_{a_2\cup\{\boldsymbol d\}}$. 
This shows that
$\alpha^*_{a \cup \{\boldsymbol d\}}$, $a \subset {\mathcal N}$ satisfy
the second part of Definition \ref{d:Yharmonic}.

Fix a vertex $a \cup \{\boldsymbol d\}$ of $G_{\boldsymbol d}$. 
By definition,
for each of its elements $l$,
this vertex is connected to $ a \cup \{\boldsymbol d\} \cup \{l-1\}$ if
$l-1 \notin a$ or 
or to $ a \cup \{\boldsymbol d\} - \{l-1\}$ if $l-1 \in a$.
Then to show that the $\alpha^*_{a \cup \{\boldsymbol d\}}$, $a \subset {\mathcal N}$
satisfies the third part of Definition \ref{d:Yharmonic} it suffices to 
prove that for each $a \subset {\mathcal N}$
and each $l \in a \cup \{\boldsymbol d\}$ such that $l-1 \notin a \cup \{\boldsymbol d\}$,
$\alpha^*_{a \cup \{\boldsymbol d\}}$ and
$\alpha^*_{a \cup \{\boldsymbol d\} \cup \{l-1\}}$ are $l$-conjugate 
(remember that the graphs of harmonic systems are symmetric).
For ease of notation let us denote 
$a \cup \{l-1\}$ by $a_1$,
$\alpha^*_{a\cup \{\boldsymbol d\}}$ by 
$\alpha^*$, $\alpha^*_{a_1 \cup \{\boldsymbol d\}}$ by $\alpha^*_1$ and
 $\beta^*_{a_1} =\beta^*_a$  by $\beta^*$ (because we have assumed $d={\boldsymbol d}$,
$\beta^*$ is in fact equal to $\rho_{\boldsymbol d}$).
We want to show that $(\beta^*, \alpha^*)$ and $(\beta^*, \alpha^*_1)$ are $l$-conjugate.
Let us assume $2 < l < {\boldsymbol d}$, the cases $l=2,{\boldsymbol d}$ 
are treated almost the same way
and are left to the reader.
By assumption $l \in \alpha^*$ but $l-1 \notin \alpha^*$.
If $l$ is the $j^{th}$ element of $a$, i.e.,
$l = a(j)$; then $a(k) =a_1(k)$ for $k < j$,
$a_1(j)=l-1$, $a(k-1) = a_1(k)$ for $k > j$.
This and \eqref{e:defalphastar} imply 
\begin{equation}\label{e:therestsame}
\alpha^*|_{{\mathcal N}-l} = \alpha^*_1|_{{\mathcal N} - l}
\end{equation}
i.e., $\alpha^*$ and $\alpha^*_1$ satisfy \eqref{e:conjugacydg2}
(for example, for $\boldsymbol d=8$, $\alpha^*_{\{3,6\}}$ is given in 
\eqref{e:examplescstaralphastar}; on the other hand 
$\alpha^*_{\{3,5,6\}} = (1,1,\rho_3,\rho_3,\rho_5,\rho_6,\rho_6,\rho_6)$ and
indeed 
$
\alpha^*_{\{3,6\}}|_{{\mathcal N}-\{6\}} = \alpha^*_{\{3,5,6\}}|_{{\mathcal N}-\{6\}}
$
).
Definition
\eqref{e:defalphastar}
also implies 
\begin{equation}\label{e:alphastarl}
\alpha^*_1(l) = \rho_{a_1(j)} = \rho_{l-1},~ \alpha^*_1(l+1) = 
\rho_{a_1(j+1)} = \rho_l.
\end{equation}
On the other hand,
again by \eqref{e:defalphastar}, and by $l-1 \notin a$,
we have
\begin{equation*}
\alpha^*(l) = \alpha^*(l-1) = \rho_{a(j-1)} \text{ and }  
\alpha^*(l+1) = \rho_l.
\end{equation*}
Then
\[
\frac{1}{\alpha^*(l)} \frac{\alpha^*(l-1) \alpha^*(l+1) \mu_l}{\mu_{l-1}}
= \rho_{l-1}
\]
and, by \eqref{e:alphastarl} this equals $\alpha^*_1(l)$; thus we have
seen 
$\boldsymbol \alpha(l,(\beta^*,\alpha^*))(l) = \alpha^*_1(l)$, where
the conjugator $\boldsymbol \alpha$ is defined as in
\eqref{e:conjugationDtandem}.
This and \eqref{e:therestsame} mean that 
$\alpha_1^* = \boldsymbol \alpha(l,(\beta^*,\alpha^*))$, i.e., 
$(\beta^*,\alpha_1^*)$ and $(\beta^*,\alpha)$ are
$l$-conjugate.

Now we will prove that the ${\boldsymbol c}^*_{a\cup\{\boldsymbol d\}}$, 
$a \subset {\mathcal N}$, defined in \eqref{e:solutiontandem} 
satisfy the fourth part of Definition \ref{d:Yharmonic}. 
The structure of $G_{\boldsymbol d}$ implies 
that it suffices to check that 
\begin{equation}\label{e:conditiononCs}
\frac{{\boldsymbol c}^*_{a}}{{\boldsymbol c}^*_{a_1} } = 
-\frac{C(l,\rho_{\boldsymbol d}, \alpha^*_{a_1})}{C(l,\rho_{\boldsymbol d},\alpha^*_a)}
\end{equation}
holds for any $ l \in a$ such that $l-1 \notin a$  and $a_1 = a \cup \{l-1\}.$
There are three cases to consider: $l=2$, $l=\boldsymbol d$ and $2 < l < \boldsymbol d$; 
we will only treat the last, the other cases can be
treated similarly and are left to the reader.
For $2 < l < \boldsymbol d$ one needs to further 
consider the cases $a(1) = l$ and $a(1) < l$.
For $b \subset {\mathcal N}$, ${\boldsymbol c}^*_{b\cup\{\boldsymbol d\}}$ of \eqref{e:solutiontandem}
is the product of a parity term and a running product of 
$d-b(1)$ ratios of the form $(\mu_l - \lambda)/(\mu_l - \mu_{a(j)}).$
The ratio of the parity terms of $a$ and $a_1$ is $-1$ because $a_1$ has one additional term.
If $a(1) = l$ then $a_1(1) = l-1$ the only difference between the running products
in the definitions of ${\boldsymbol c}^*$ and ${\boldsymbol c}^*_1$
is that the latter has an additional initial term
$(\mu_l-\lambda)/(\mu_l-\mu_{l-1})$ and therefore
\[
\frac{{\boldsymbol c}^*_{a}}{{\boldsymbol c}^*_{a_1}} = -\frac{ \mu_l-\mu_{l-1}}{\mu_l-\lambda}.
\]
Because $l > 2$ and $l-1 \ge 2$, \eqref{e:defalphastar} implies $\alpha^*(l) =1$, $\alpha^*(l+1) = \rho_l$,
$\alpha^*_1(l) = \rho_{l-1}$ and $\alpha^*_1(l+1) = \rho_l$.
These and \eqref{e:Cjbetaalphatandem} imply
\[
\frac{C(l,\rho_{\boldsymbol d}, \alpha^*_{a_1})}{C(l,\rho_{\boldsymbol d},\alpha^*_a)} = \frac{ \mu_l - \mu_{l-1}}{\mu_l - \lambda}.
\]
The last two displays imply \eqref{e:conditiononCs} for $a(1) = l$. 
If $l > a(1)$, let $j>1$ be the position of $l$ in $a$, i.e., $l=a(j)$.  
In this case, the definition \eqref{e:solutiontandem} implies that
the running products in the definitions of 
${\boldsymbol c}^*_{a}$ and ${\boldsymbol c}^*_{a_1}$ have the same number of ratios and they are all
equal except for the $l^{th}$ terms, which is $(\mu_l - \lambda)/(\mu_l -\mu_{a(j-1)})$ 
for the former and $(\mu_l - \lambda)/(\mu_l -\mu_{l-1})$ for the latter. $a_1$ has
one more element than $a$, therefore, the ratio of the parity terms is again
$-1$; these imply
\[
\frac{{\boldsymbol c}^*_{a}}{{\boldsymbol c}^*_{a_1}} = -\frac{ \mu_l-\mu_{l-1}}{\mu_l-\mu_{a(j-1)}}.
\]
On the other hand, $l \in a$,  $j > 1$,  $a_1 = a \cup \{l-1\}$ 
and \eqref{e:defalphastar} imply 
$\alpha^*(l) = \rho^*(a(j-1))$, $\alpha^*(l+1) = \rho_l$, $\alpha_1^*(l) = \rho_{l-1}$,
and $\alpha_1^*(l+1) = \rho_l$ and therefore 
\[
\frac{C(l,\rho_{\boldsymbol d}, \alpha^*_{a_1})}{C(l,\rho_{\boldsymbol d},\alpha^*_a)} = 
\frac{ \mu_l - \mu_{l-1}}{\mu_l - \mu_{a(j-1)}}.
\]
The last two displays once again imply \eqref{e:conditiononCs} when
$a(j) = l$ with $j > 1$.

For $a \subset {\mathcal N}$, the definition \eqref{e:defalphastar}
implies 
\[
\alpha^*_{a\cup\{d\}}(l) =  \alpha^*_{a\cup\{d\}}(l+1);
\]
we have already shown $\alpha^*_{a\cup\{d\}} \in {\mathcal H}$,
then, Lemma \ref{l:conditionforintersection} and the last display
imply $\alpha^*_{a\cup \{\boldsymbol d\}} \in {\mathcal H}_l$ for $l \notin a$.
Then by \eqref{e:Gdloops} 
$\alpha^*_{a\cup \{\boldsymbol d\}} \in {\mathcal H}_l$ for each loop on the
vertex $a\cup \{\boldsymbol d\}$ of $G_{\boldsymbol d}$, i.e., the last part of Definition
\ref{d:Yharmonic} is also satisfied. This finishes the proof of the
proposition.
\end{proof}

\begin{proposition}\label{p:defhd}
\begin{equation}\label{e:defhd}
h^*_d \doteq \sum_{ a \subset \{1,2,3,...,d-1\}} 
{\boldsymbol c}^*_{a\cup \{d\}} [(\rho_d, \alpha^*_{a\cup\{d\}}),\cdot],
\end{equation}
$d=1,2,3,...,{\boldsymbol d}$, are $\partial B$-determined $Y$-harmonic functions.
\end{proposition}
\begin{proof}
That $h^*_d$ is $Y$-harmonic follows from Propositions \ref{p:provesolution} 
and \ref{p:simpleharmonicfunctions}.
The components of $\alpha^*_{a\cup\{d\}}$, $a\subset
\{1,2,3,...,d-1\}$ and $\beta^*_d = \rho_d$ are all between $0$ and $1$.
This and Proposition \ref{p:balayagesimpled} imply 
that  $h^*_d$ are all $\partial B$-determined.
\end{proof}

With definition \eqref{e:defhd} we can rewrite \eqref{e:embeddingfunc} as
\begin{equation}\label{e:embeddingfunc1}
-\left(\prod_{l=d+1}^{{\boldsymbol d}} \frac{ \mu_l - \lambda}{\mu_l - \mu_d}\right)
h_d^*(y)
=
\sum_{ a \in V_{G_{\boldsymbol d}^d }} {\boldsymbol c}_a^* [ (\beta_a^*, \alpha_a^*), y ]
\end{equation}
for $y \in {\partial B}.$

\begin{proposition}\label{p:exactformulaDtandem}
\begin{equation}\label{e:exactformulaDtandem}
P_y(\tau < \infty) = 
\sum_{d=1}^{{\boldsymbol d}}
\left( \prod_{l=d+1}^{\boldsymbol d} \frac{\mu_l -\lambda}{\mu_l - \mu_d}  \right) 
h^*_d(y)
\end{equation}
for $y \in B.$
\end{proposition}
Display \eqref{e:nicerep} below in subsection \ref{ss:ex1} 
shows the right side of \eqref{e:exactformulaDtandem}
for $\boldsymbol d = 2.$
\begin{proof}
Let ${\boldsymbol 1} \in {\mathbb C}^{\mathcal N}$ denote the vector with
all components equal to $1$.
The decomposition of $G_{\boldsymbol d}$ into the single vertex ${\{\boldsymbol d}\}$
and $G_{\boldsymbol d}^d$, $d < {\boldsymbol d}$ implies that the right side of 
\eqref{e:exactformulaDtandem} equals
\begin{align*}
&[(\rho_{\boldsymbol d}, {\boldsymbol 1}),y] + 
\sum_{d=1}^{{\boldsymbol d}-1}
\sum_{a \in V_{G_{\boldsymbol d}^d}} {\boldsymbol c}^*_a
[ ( \beta_a^*, \alpha^*_a), y ]
+\sum_{d=1}^{{\boldsymbol d}-1}
\left( \prod_{l=d+1}^{\boldsymbol d} \frac{\mu_l -\lambda}{\mu_l - \mu_d}  \right)h^*_d(y)
\intertext{for $y \in \partial B$; \eqref{e:embeddingfunc1} implies}
&~~= 
[(\rho_{\boldsymbol d}, {\boldsymbol 1}),y]
\end{align*}
which, for $y\in \partial B$, equals $1$.
Thus, we see that
the right side of \eqref{e:exactformulaDtandem} equals $1$ on $\partial B$.
Proposition \ref{p:defhd} says that the same function is $\partial B$-determined 
and is $Y$-harmonic. Then its restriction to $B$
must be indeed equal to $y\rightarrow P_y(\tau < \infty)$, $y \in B$,
which is the unique function with those properties.
\end{proof}

\section{Convergence - initial condition set for $X$}\label{s:convergence2}
In Section \ref{s:convergence1} 
we have proved a convergence result which
takes as input 
the initial position of the $Y$ process.
One can also provide, as is done in the LD analysis,
an initial position to the
$X^n$ process as $X^n(0) = \lfloor n x \rfloor$ for  a fixed
$x \in {\mathbb R}_+^d$ with 
$\sum_{i \in {\mathcal N}_+} x(i) < 1$ and
prove a convergence result in this setting.
The initial position
$X^n(0) = \lfloor nx \rfloor$ implies that probabilities such as those
in \eqref{e:decompose1} will all decay to $0$ and therefore convergence
to $0$ no longer suffices to argue that a probability is negligible,
{\em we will now compare LD decay rates} of the probabilities which appear
in the convergence analysis.
In the current literature only some of these rates
have been computed in any generality.
We believe that, at least for the exit boundary $\partial A_n$, 
all of the necessary rates can be computed 
but forms a nontrivial task and will require an article of its own. Thus,
instead of treating the general case, for the purposes of this paper,
we will confine ourselves to the case of two tandem queues in
our convergence analysis when 
the initial position is given as $X^n(0) = \lfloor n x \rfloor$.

In the rest of the section $X$ will refer to the two dimensional tandem walk.
The possible increments of
$X$ are $v_0 \doteq (0,1)$, $v_1 \doteq (-1,1)$ and $v_2 \doteq (0,-1)$ 
with probabilities
$p(v_0) = \lambda$, $p(v_1) = \mu_1$ and $p(v_2) = \mu_2$.  
For this model the stability condition \eqref{e:stable} becomes
$\lambda < \mu_1,\mu_2.$ 
On $\{x:x(1) = 0\}$  
[$\{x:x(2)=0\}$]
the increment $(-1,1)$ [$(0,-1)$]
is replaced with $(0,0)$.
For the present proof it will be more convenient to cast the limit in terms
of the original coordinates of the $X$ process.
The $Y$ process in the coordinate space of the $X$ process
is
$\bar{X} \doteq T_n(Y)$.
$\bar{X}$ is the same process as $X$ except that it
is constrained only at the boundary
$\partial_2.$
\begin{align*}
X_{k+1} &= X_k+ \pi(X_k,I_k)\\
\bar{X}_{k+1} &= \bar{X}_k + \pi_1(\bar{X}_k,I_k).
\end{align*}
We will assume that $X$ and $\bar{X}$ start from the same initial position
\[
X_0= \bar{X}_0
\]
and whenever we specify an initial position below it will be for both processes.

As before,
$\tau =\inf\{k: Y_k \in \partial B\} = \inf\{k: \bar{X}_k \in \partial A_n\}$,
$\tau_n =  \{k: X_1(k) + X_2(k) = \partial A_n \}$ 
(by definition, $\bar{X}$ hits $\partial A_n$ exactly when $Y$ hits $\partial B$);
the subscript of $P$ will denote initial position, i.e, 
$P_{x}(\tau < \infty)$ equals $P(\tau < \infty)$ when $X_0 = \bar{X}_0 =x.$

\begin{proposition}\label{t:guzel}
Let $X$ and $\bar{X}$ be as above and assume $\lambda < \mu_1,\mu_2.$
For $x \in {\mathbb R}_+^2$, $0 < x(1) + x(2) < 1, x(1) > 0$ set
$ x_n \doteq \lfloor nx \rfloor$.
Then
\begin{equation}\label{e:relativeerror}
\frac{
|P_{x_n}(\tau_n < \tau_0) - P_{x_n}( \tau < \infty )|}
	{P_{x_n}(\tau_n < \tau_0)}
\end{equation}
decays exponentially in $n$.
\end{proposition}
The proof will require several supporting results on 
$\sigma_1 =\inf\{ k: X_k \in \partial_1 \}$ and
\begin{align*}
\sigma_{1,2} &\doteq \inf \{k: k > \sigma_1, X_k \in \partial_2 \},\\
\bar{\sigma}_{1,2} &\doteq \inf \{k: k > \sigma_1, \bar{X}_k(1) = -\bar{X}_k(2) \}.
\end{align*}
\begin{proposition}\label{p:equalityofsums}
\begin{equation}\label{e:equalityofsums}
X_k(1)+ X_k(2) = \bar{X}_k(1) + \bar{X}_k(2)
\end{equation}
for $k \le \sigma_{1,2}.$
\end{proposition}
\begin{proof}
\begin{equation}\label{e:equalityatsigma1}
X_k = \bar{X}_k
\end{equation}
for $k \le \sigma_1$ implies \eqref{e:equalityofsums} for $k  \le \sigma_1$.
If $\sigma_1 = \sigma_{1,2}$ then we are done. Otherwise 
$X_{\sigma_1}(2) = \bar{X}_{\sigma_1}(2) > 0$ and $X_k(2) > 0$ for 
$\sigma_1 < k < \sigma_{1,2}$;  
let $\sigma_1 = \nu_1 < \nu_2 < \cdots <\nu_K < \sigma_{1,2}$ 
be the times when
$X$ hits $\partial_1$ before hitting $\partial_2.$ The definitions of $\bar{X}$
and $X$ imply that these are the only times
when the increments of $X$ and $\bar{X}$ differ: $X_{\nu_j+1} - X_{\nu_j}=0$ 
and $\bar{X}_{\nu_j+1} - \bar{X}(\nu_j) = (-1,1)$
if
$I_{\nu_j}= (-1,1)$; otherwise both differences equal $I_{\nu_j}$. This and
\eqref{e:equalityatsigma1} imply
\begin{equation}\label{e:invariantinc}
X_k - \bar{X}_k = \varsigma_k \cdot ( -1,1)
\end{equation}
for $k \le \sigma_{1,2}$
where 
\[
\varsigma_k \doteq \sum_{j = 1}^K 1_{\{ \nu_j \le k\} } 1_{\left\{ I_{\nu_j} = (-1,1) 
\right\}}
\]
and $\cdot$ denotes scalar multiplication.
Summing the components of both sides of \eqref{e:invariantinc}
gives \eqref{e:equalityofsums}.
\end{proof}

Define
\[
\Gamma_n \doteq \{ \sigma_1 < \sigma_{1,2} < \tau_n < \tau_0\}.
\]
$\Gamma_n$ is one particular way for $\{\tau_n < \tau_0\}$ to occur.
In the next proposition we find an upperbound on its probability 
in terms of
\[
\gamma \doteq -(\log(\rho_1) \vee \log( \rho_2))
\]

\begin{proposition}\label{t:thm1}
For any $\epsilon > 0$
there is  $N > 0 $ such that if $n > N$
\begin{equation}\label{e:bound1}
P_{x_n} ( \Gamma_n ) \le e^{-n (\gamma-\epsilon)},
\end{equation}
where $x_n = \lfloor nx\rfloor$ and $x \in{\mathbb R}_+^2$, $x(1) + x(2) < 1.$
\end{proposition}
The proof will use the following definitions.
\begin{equation}\label{e:defHs}
H_a(q)  \doteq
-\log\left( \sum_{i \in \{0,1,2\} -a } p(v_i) e^{-\langle v_i, q \rangle} 
+ \sum_{\{ i \in a\}}p(v_i) \right),
\end{equation}
where $\langle \cdot,\cdot\rangle$ denotes the inner product in ${\mathbb R}^2.$
For $x \in {\mathbb R}_+^2$, 
set
\[
{\bf b}(x) \doteq \{i: x(i) = 0 \}.
\]
We will write $H$ rather than $H_\emptyset$. 

Let us show the gradient operator on smooth functions on ${\mathbb R}^2$ 
with $\nabla$.
The works \cite{thesis,DSW} use a smooth subsolution of 
\begin{equation}\label{e:HJB0}
H_{\bf{b}(x)}(\nabla V(x)) = 0
\end{equation}
to find an lowerbound on the decay rate of the second moment of IS estimators
for the probability $P_{x_n}(\tau_n  < \tau_0)$.
$V$ is said to be a subsolution of \eqref{e:HJB0} 
if $H_{\bf{b}(x)}(\nabla V(x)) \ge 0$.
The event $\Gamma_n$
consists of three stages: the process first hits $\partial_1$, then $\partial_2$
and then finally hits $\partial A_n$ without hitting $0$. 
To handle this, we will use a function
$(s,x)\rightarrow V(s,x)$, 
with two variables;
for the $x$ variable we will substitute 
the scaled position of the $X$ process,
and the discrete variable $s \in \{0,1,2\}$  is 
for keeping track of which of the above three stages
the process is in; $V$ will be a subsolution in the $x$ variable 
and continuous in $(s,x)$ (when $(s,x)$ is thought of as a point on the
manifold ${\mathcal M}$ 
consisting of three copies of ${\mathbb R}_+^2$ (one for each stage); the zeroth glued to the
first along $\partial_1$ and the first to the second along $\partial_2$)
 and therefore one can think of
$V$ as three subsolutions (one for each stage) glued together along the boundaries
of the state space of $X$ where transitions between the stages occur. 
We will call a function $(s,x) \rightarrow V(s,x)$ with the above properties
a subsolution of \eqref{e:HJB0} on the manifold ${\mathcal M}.$

Define
\begin{equation}\label{e:subsolpa}
\tilde{V}_i^{\varepsilon}(x) \doteq \langle {\boldsymbol r}_i, x\rangle + 2\gamma - (3-i)\varepsilon,~~~
\tilde{V}^{\varepsilon,j} \doteq \bigwedge_{i=0}^j \tilde{V}_i^{\varepsilon},
\end{equation}
where
\[
{\boldsymbol r}_0 \doteq(0,0), {\boldsymbol r}_1 = -\gamma(1,0), {\boldsymbol r}_2 \doteq -\gamma(1,1).
\]
The subsolution for stage $j$ will be a smoothed version of $\tilde{V}^{\varepsilon,j}$;
As in \cite{thesis, DSW}, 
we will need to vary $\varepsilon$ with $n$ in the convergence argument; for this
reason, $\varepsilon$ will appear as the third parameter of the constructed subsolution.
The details are as follows.

The subsolution for the zeroth stage is $\tilde{V}^{0,\varepsilon}$:
$V(0,x,\varepsilon) \doteq \gamma - 3\varepsilon$,
$\nabla V(0,\cdot) =0$ and it trivially satisfies \eqref{e:HJB0} and
is therefore a subsolution.

Define the smoothing kernel
\[
\eta_{\delta}(x) \doteq \frac{1}{\delta^2 M} \eta(x/\delta),~~\\
\eta(x) \doteq 1_{\{|x| \le 1\} } (|x|^2 - 1), M \doteq \int_{ {\mathbb R}^2} \eta(x) dx\\
\]

To construct the subsolution for the first and the second stages we will mollify
$\tilde{V}^{j,\varepsilon}$, $j=1,2$, with $\eta$:
\begin{equation}\label{e:subsolpas}
V(j,x,\varepsilon) \doteq \int_{{\mathbb R}^2} \tilde{V}^{j,\varepsilon}(y)
\eta_{c_2 \varepsilon} (x-y) dy,
\end{equation}
and $c_2$ is chosen so that 
\begin{equation}\label{e:V12nearonp2}
V(1,x,\varepsilon) = V(2,x,\varepsilon) 
\end{equation}
for $x \in
\partial_2$ and 
\begin{equation}\label{e:V01nearonp1}
V(1,x,\varepsilon) = V(0,x,\varepsilon) 
\end{equation}
for $x \in \partial_1$
(this is possible since $V(j,\varepsilon,x) \rightarrow \tilde{V}^{j,\varepsilon}$ 
as $c_2 \rightarrow 0$ and all of the involved functions are affine; see
\cite[page 38]{thesis} on how to compute $c_2$ explicitly).
That
$V(j,\cdot,\varepsilon)$, $j=1,2$ are subsolutions follow the concavity
of $H_a$ and the choices of the gradients ${\boldsymbol r}_i$; for details
we refer the reader to
\cite[Lemma 2.3.2]{thesis}; a direct computation gives
\begin{equation}\label{e:Boundon2ndderivative}
\left| \frac{ \partial^2 V(j,\cdot,\varepsilon)}{\partial x_i \partial x_j} \right| 
\le \frac{c_3}{\varepsilon},
\end{equation}
$j=1,2$, for a constant $ c_3 > 0$ (again, the proof of \cite[Lemma 2.3.2]{thesis}
gives the details of this computation).

The construction above implies
\begin{equation}\label{e:valueofVonpartialn}
V(2,x,\varepsilon ) < 0, x \in \{x: x(1) + x(2)= 1\}.
\end{equation}

Now on to the proof of Proposition \ref{t:thm1}.
\begin{proof}
$V(0,\cdot,\varepsilon)$ maps to a constant and thus
\begin{equation}\label{e:exactreplacement}
\langle \nabla W(x), v_i \rangle = W(x+v_i) - W(x) 
\end{equation}
if $W=V(0,\cdot,\varepsilon)$.
For $W= V(j,\cdot,\varepsilon)$, $j=1,2$, 
Taylor's formula and \eqref{e:Boundon2ndderivative} give
\begin{equation}\label{e:approxforV2}
\left| 
 \left\langle \nabla W(x), \frac{1}{n}v_i \right\rangle 
-\left( W\left(x+\frac{1}{n} v_i\right) - W(x) \right) \right| \le \frac{c_3}{n\varepsilon}.
\end{equation}
We will allow $\varepsilon$ to depend on $n$ so that
$\varepsilon_n \rightarrow 0$ and
$n\varepsilon_n  \rightarrow \infty.$
Define
$S_k = 1_{\{\sigma_1 > k \}} + 1_{\{ \sigma_{1,2} > k \}}$,
$M_0 \doteq 1$ and
{\small
\begin{align*}
M_{k+1} \doteq M_k
 \exp\left( -n 
\left( V \left(S_{k+1},\frac{X_{k+1}}{n}, \varepsilon_n\right)-V\left(S_k,
\frac{X_k}{n},\varepsilon_n\right)\right)
 -1_{\{n > \sigma_{1}\}} \frac{c_3}{n\varepsilon_n}\right)
\end{align*}
}
That $V(j,\cdot,\varepsilon_n)$, $j=0,1,2$ are subsolutions of \eqref{e:HJB0}, the relations
\eqref{e:exactreplacement}, \eqref{e:approxforV2}
\eqref{e:V01nearonp1} and \eqref{e:V12nearonp2}
imply that 
$M$ is a supermartingale
\eqref{e:approxforV2} 
and
(\eqref{e:exactreplacement} 
allow us to replace gradients
in \eqref{e:HJB0} and \eqref{e:defHs} with finite differences and 
\eqref{e:V12nearonp2} 
and 
\eqref{e:V01nearonp1} 
preserve the supermartingale
property of $M$ as $S$ passes from $0$ to $1$ and from $1$ to $2$).
This and $M \ge 0 $ imply (see \cite[Theorem 7.6]{MR1609153})
{\small
\[
{\mathbb E}_{x_n} 
\left[
\prod_{k=1}^{\tau_{0,n}} 
 \exp\left( -n 
\left( V \left(S_{k+1},\frac{X_{k+1}}{n},\varepsilon_n\right)-
V\left(S_k,\frac{X_k}{n},\varepsilon_n\right)\right) 
-1_{\{n > \sigma_{1}\}} \frac{c_3}{n\varepsilon_n}\right)
\right] \le 1,
\]
}
where $\tau_{0,n} \doteq \tau_n \wedge \tau_0.$
Restrict the expectation on the left to  $1_{\Gamma_n}$
and replace $1_{\{n > \sigma_{1}\} }$ with $1$ to make the expectation
smaller:
{\small
\begin{align*}
&{\mathbb E}_{x_n} 
\left[
1_{\Gamma_n}
e^{-\frac{c_3}{n\varepsilon_n} \tau_{0,n}}  
 \exp\left( -n 
\sum_{k=1}^{\tau_{0,n}} 
 V \left(S_{k+1},\frac{X_{k+1}}{n},\varepsilon_n\right)
-V\left(S_k,\frac{X_k}{n},\varepsilon_n\right)
\right)  \right] \le 1.
\end{align*}
}
Over $\Gamma_n$, $X$ first hits $\partial_1$ and then $\partial_{2}$
and finally $\partial A_n$. Furthermore, 
the sum inside the expectation is telescoping
across this whole trajectory; these imply that 
the last inequality reduces to
\[
{\mathbb E}_{x_n} 
\left[
1_{\Gamma_n}
e^{-\frac{c_3}{n\varepsilon_n} \tau_{0,n}} 
\exp( -n( V(2,X_{\tau_{0,n}},\varepsilon_n) - V(0,X_0,\varepsilon_n)) 
\right] \le 1.
\]
$\tau_{0,n} = \tau_n$ on $\Gamma_n$ and therefore on the same set
$X_{\tau_{0,n}} \in \partial_n$. This, $V(0,\cdot,\epsilon_n)= \gamma-3\epsilon_n$,
\eqref{e:valueofVonpartialn}
and the previous inequality give
\begin{equation}\label{e:basic}
{\mathbb E}_{x_n} 
\left[
1_{\Gamma_n}
e^{-\frac{c_3}{n\varepsilon_n} \tau_{0,n}} 
\right] \le e^{-n( \gamma - 3 \varepsilon_n)}.
\end{equation}
Now suppose that the statement of Theorem \ref{t:thm1} is not true,
i.e., there exists $\epsilon > 0$  and a sequence $n_k$ such that
\begin{equation}\label{e:contra}
P_{x_{n_k}}(\Gamma_{n_k}) > e^{-n_k (\gamma -\epsilon)}
\end{equation}
for all $k$.  Let us pass to this subsequence and drop the subscript
$k$. 
\cite[Theorem A.1.1]{thesis} implies that 
one can choose $ c_4 > 0$ so that $P( \tau_{0,n} > n c_4 ) \le e^{-n(\gamma + 1) }$
for $n$ large.
Then
\begin{align*}
{\mathbb E}_{x_n} 
\left[
1_{\Gamma_n}
e^{-\frac{c_3}{n\varepsilon_n} \tau_{0,n}} 
\right] 
&\ge
{\mathbb E}_{x_n} 
\left[
1_{\Gamma_n}
e^{-\frac{c_3}{n\varepsilon_n} \tau_{0,n}} 
 1_{\{\tau_{0,n} \le n c_4\}}
\right] 
\\
&\ge
e^{-\frac{c_4 c_3}{n\varepsilon_n} n }
{\mathbb E}_{x_n} 
\left[
1_{\Gamma_n}
 1_{\{\tau_{0,n} \le n c_4\}}
\right] 
\intertext{
$P(E_1 \cap E_2) \ge 
P(E_1) - P(E_2^c)$
for any two events $E_1$ and $E_2$; this and the previous line imply}
&\ge
e^{\frac{-c_3 c_4}{n\varepsilon_n} n } \left( P_{x_n}(\Gamma_n) - P_{x_n}( \tau_{0,n} > nc_4)
\right)
\\
&\ge
e^{-\frac{c_3 c_4}{n\varepsilon_n} n } \left( e^{-n(\gamma-\varepsilon)}
- e^{-(\gamma + 1)n} \right).
\end{align*}
By assumption $n\varepsilon_n \rightarrow \infty$ which implies
$c_3 c_4/n\varepsilon_n \rightarrow 0$; this and the last inequality say\\
${\mathbb E}_{x_n} 
\left[
1_{\Gamma_n}
e^{-\frac{c_3}{n\varepsilon_n} \tau_{0,n}} 
\right]$
cannot decay at an exponential rate faster than $\gamma - \epsilon$,
but this contradicts
\eqref{e:basic} because $\varepsilon_n \rightarrow 0.$
Then, there cannot be $\epsilon > 0$ and a sequence
$\{n_k\}$ for which \eqref{e:contra} holds and this implies
the statement of Proposition \ref{t:thm1}.
\end{proof}

Define 
${\boldsymbol r}_3 \doteq \log(\rho_2) (1,1)$ and
$V(x) \doteq (-\log(\rho_1) + \langle {\boldsymbol r}_1, x\rangle )\wedge (-\log(\rho_2) + \langle {\boldsymbol r}_3, x \rangle)$,
for $x \in {\mathbb R}^2$ 
\begin{proposition}\label{p:ldfortandem}
\[
\lim_{n\rightarrow \infty}-\frac{1}{n} \log  P_{x_n}( \tau_n < \tau_0) = V(x)
\]
for $x \in {\mathbb R}_+^2$, $0 < x(1) + x(2) < 1$ and  $x_n = \lfloor nx \rfloor$.
\end{proposition}
The omitted proof is a one step version of the argument used in the proof
of Proposition \ref{t:thm1} and uses a mollification of $V$ as the subsolution.

\begin{proposition}\label{p:artikforXbar}
For any $\epsilon > 0$
there is  $N > 0 $ such that if $n > N$
\begin{equation}\label{e:artikforXbar}
P_x( \sigma_1 < \sigma_{1,2} < \tau < \infty)  \le e^{-n(\gamma-\epsilon)}
\end{equation}
where $x_n = \lfloor nx\rfloor$ and $x  \in {\mathbb R}_+^2$, $x(1) + x(2) < 1.$
\end{proposition}
\begin{proof}
Write
\begin{align*}
P_x( \sigma_1 < \sigma_{1,2} < \tau < \infty) 
=P_x( \sigma_1 < \sigma_{1,2} < \bar{\sigma}_{1,2} < \tau < \infty)
+
P_x( \sigma_1 < \sigma_{1,2} < \tau <  \bar{\sigma}_{1,2}).
\end{align*}
The definitions of $X$ and $\bar{X}$ imply 
$\tau_0 \ge  \bar{\sigma}_{1,2}$. Then, if a sample path $\omega$ satisfies
$\sigma_1(\omega) < \sigma_{1,2}(\omega) < \tau(\omega) < \bar{\sigma}_{1,2}$, it must 
also satisfy $\sigma_1(\omega) < \sigma_{1,2}(\omega) < \tau_n(\omega) < \tau_0(\omega)$.
This and Proposition \ref{t:thm1} imply that there is an $N$ such that
\[
P_{x_n}( \sigma_1 < \sigma_{1,2} < \tau <  
\bar{\sigma}_{1,2}) \le e^{-n(\gamma -\epsilon)},
\]
for $n > N$.
On the other hand, Proposition \ref{p:exactformulaDtandem} 
and the Markov property of $\bar{X}$ imply
\[
P_{x_n}( \sigma_1 < \sigma_{1,2} < \bar{\sigma}_{1,2} < \tau < \infty) \le
c_5 e^{-n(\gamma - \epsilon)},
\]
for some constant $c_5 > 0$.
These imply \eqref{e:artikforXbar}. 
\end{proof}

\begin{proof}[Proof of Proposition \ref{t:guzel}]
Decompose $P_x( \tau_n < \tau_0)$ and $P_x( \bar{\tau} < \infty)$ as follows:
\begin{align}
P_{x_n} ( \tau_n < \tau_0) &= \label{e:dectaun} P_{x_n}( \tau_n < \sigma_1 < \tau_0)+
 P_{x_n}(\sigma_1 < \tau_n \le  \sigma_{1,2} \wedge \tau_0)\\
&~~~~~~~~~~~~~~~~~~~+ P_{x_n}( \sigma_1 < \sigma_{1,2} < \tau_n < \tau_0)\notag\\
P_{x_n}( \tau < \infty) &= \label{e:dectn} 
P_{x_n}( \tau < \sigma_1 ) + P_{x_n}(\sigma_1 < \tau < \sigma_{1,2} )
+ P_{x_n}( \sigma_1 < \sigma_{1,2} < \tau < \infty). 
\end{align}
By definition $X$ and $\bar{X}$ are identical until they hit $\partial_1$;
therefore $\{\tau_n < \sigma_1\} = \{\tau < \sigma_1\}$ and 
\begin{equation}\label{e:equalityoffirstterms}
P_{x_n}(\tau_n < \sigma_1) = P_{x_n}( \tau < \sigma_1).
\end{equation}

The processes $X$ and $\bar{X}$ begin to differ after they hit $\partial_1$;
but Proposition \ref{p:equalityofsums} says that the sums of their components
remain equal before time $\sigma_{1,2}$; this implies $\tau = \tau_n$ on
$\tau_n \le \sigma_{1,2}$ and therefore
\[
 P_{x_n}(\sigma_1 < \tau \le \sigma_{1,2} )=
P_{x_n}(\sigma_1 < \tau_n \le  \sigma_{1,2} \wedge \tau_0)
\]
This \eqref{e:equalityoffirstterms} and the decompositions \eqref{e:dectaun} and
\eqref{e:dectn} imply
\begin{align*}
|P_{x_n} ( \tau_n < \tau_0) - P_{x_n}( \tau < \infty) |
=
| P_{x_n}( \sigma_1 < \sigma_{1,2} < \tau_n < \tau_0)-
P_{x_n}( \sigma_1 < \sigma_{1,2} < \tau < \infty)|
\end{align*}
By Propositions \ref{t:thm1} and \ref{p:artikforXbar} for $\epsilon > 0$
arbitrarily small the right side of
the last equality is bounded above by $e^{- n(\gamma-\epsilon)}$
when $n$ is large.
On the other hand, Proposition \ref{p:ldfortandem} says for $\epsilon_0 > 0$
arbitrarily small
$P_{x_n}( \tau_n < \tau_0) \ge e^{-n(\gamma_1 +\epsilon_0)}$ for $n$ large
where $\gamma_1 \doteq V(x) < \gamma$.
Choose $\epsilon$ and $\epsilon_0$ to satisfy
\[
\gamma -\gamma_1 > \epsilon+ \epsilon_0.
\]
These imply that 
for $c_6= (\epsilon + \epsilon_0)+ \gamma_1 -\gamma < 0$
\[
\frac{|P_{x_n} ( \tau_n < \tau_0) - P_{x_n}( \tau < \infty) |}
{|P_{x_n} ( \tau_n < \tau_0)|} < e^{c_6 n}
\]
when $n$ is large; this is what we have set out to prove.
\end{proof}

It is possible to generalize Proposition \ref{t:guzel} in many directions.
In particular, one expects it to hold for any tandem walk of finite dimension 
with the same exit boundary; the proof will almost be identical
but requires a generalization of Proposition \ref{p:ldfortandem}, which, we
believe, will involve the same ideas given in its proof. We leave this
task to a future work.

There is a clear correspondence between the 
structures which appear in
the LD analysis (and the subsolution approach to IS estimation) of $p_n$
and those involved in the methods developed in this paper.
This connection is best expressed in the following equation (in
the context of two tandem walk just studied):
For $q=(q_1,q_2) \in {\mathbb R}^2$
set $\beta = e^{q_1}$ and $\alpha = e^{q_1-q_2}$; then
\[
H(q) = -\log ({\mathbf p}(\beta,\alpha)),
\]
where ${\mathbf p}$ is the characteristic polynomial 
defined in \eqref{e:hamiltonianba}. A similar relation exists
between $H_2$ and ${\mathbf p}_2.$
In the LD analysis $H$ and $H_1$ appear as two of the
Hamiltonians of the limit deterministic continuous time
control problem; the gradient of the
limit value function lies on their zero level sets.
In our approach, the counterpart of $H$ is the characteristic
polynomial
${\mathbf p}$; its $1$-level set ${\mathcal H}$ 
is the starting point of our definition of harmonic systems
whose solutions give harmonic functions for the limit unstable 
constrained random walk $Y$ of our analysis.

\subsection{How to combine multiple approximations}\label{ss:combineapp}
We have seen with Proposition \ref{t:guzel} 
that $P_{y_n}(\tau < \infty)$
approximates $P_{x_n}(\tau_n < \tau_0)$, $x_n = \lfloor nx \rfloor$
extremely well (i.e., exponentially decaying relative error)
for all $x \in A \doteq \{ x \in {\mathbb R}_+^2, x(1) > 0, x(1) + x(2) < 1\}$ when $n$ is large. In general this will not be true and to get
a good approximation across all $A$ we will have to use the
transformation $T_n^2$ as well as
$T_n^1.$ Thus, for general $X$, we will have to construct two limit processes
$Y^1$ and $Y^2$; $Y^1$ will be as above and $Y^2$ will be the limit of
$Y^{2,n} \doteq T_n^2(X).$ In words, we obtain $Y^2$ from $X$, by moving
the origin to $(0,n)$ via an affine change of coordinates and removing
the constraint on the {\em second} coordinate. In $d$, dimensions we will
have $d$ possible limit processes, one for each corner of $\partial A_n.$
A key question is how to decide for which of these limit processes
$P_{y_n}(\tau < \infty)$
best approximates $P_{x_n}( \tau_n < \tau_0)$. For the exit boundary
$\partial A_n$, we think that taking the maximum of the alternatives will 
suffice. We believe that the proof of this claim will involve arguments
similar to those given above. We hope to provide its details in a future
work, starting with the two dimensional case treated in this section.
An example is given in subsection \ref{ss:gtdw}.

\section{Examples}\label{s:examples}
We look at three examples: two tandem walk, general two dimensional walk
and $d$-tandem walk. We have used Octave \cite{octave} for the numerical computations
in this section and the rest of the paper.
\subsection{Two dimensional tandem walk}\label{ss:ex1}
Let us begin with the two tandem walk for which \eqref{e:exactformulaDtandem}
becomes
\begin{align}\label{e:nicerep}
P_y( \tau < \infty) &=
\rho_2^{y(1)-y(2)} +
\frac{\mu_2 - \lambda}{\mu_2 - \mu_1} 
\rho_1^{y(1)}
+ \frac{\mu_2-\lambda}{\mu_1 -\mu_2} \rho_2^{y(1)-y(2)} \rho_1^{y(2)}.
\end{align}
This gives the following approximation for $P_x( \tau_n < \tau_0)$:
\begin{align}\label{e:nicereptransformed}
f(x) \doteq \rho_2^{n-(x(1) + x(2))} + \frac{\mu_2 - \lambda}{\mu_2 - \mu_1} 
\rho_1^{n-(x(1) + x(2)) } \rho_1^{x(2)}
+ \frac{\mu_2-\lambda}{\mu_1 -\mu_2} \rho_2^{n-(x(1)+x(2))} \rho_1^{x(2)}.
\end{align}
Proposition \ref{t:guzel} says that for $x \in {\mathbb R}_+^2$ and $x_n =
\lfloor nx \rfloor$, 
the relative error
\[
\frac{|f(x_n) - P_{x_n}( \tau_n < \tau_0)|}{P_{x_n}(\tau_n < \tau_0)}
\]
decays exponentially in $n$. Let us see numerically how well this approximation works.
Set $\mu_1 = 0.4$, $\mu_2 = 0.5$, $\lambda = 0.1$ and $n=60$. 
In two dimensions, one can
quickly compute $P_{x_n}(\tau_n < \tau_0)$ by numerically iterating \eqref{e:linear}
and using the boundary conditions $V_{\partial A_n} = 1$ and $V(0) = 0$; we will 
call the result of this computation ``exact.''  Because both $f$ and $P_x(\tau_n 
<\tau_0)$ decay exponentially in $n$, it is visually simpler to compare
\begin{equation}\label{e:defVn}
V_n \doteq -\frac{1}{n} \log P_x( \tau_n < \tau_0),\text{ and }
W_n \doteq -\frac{1}{n} \log f(x).
\end{equation}

\ninsepsc{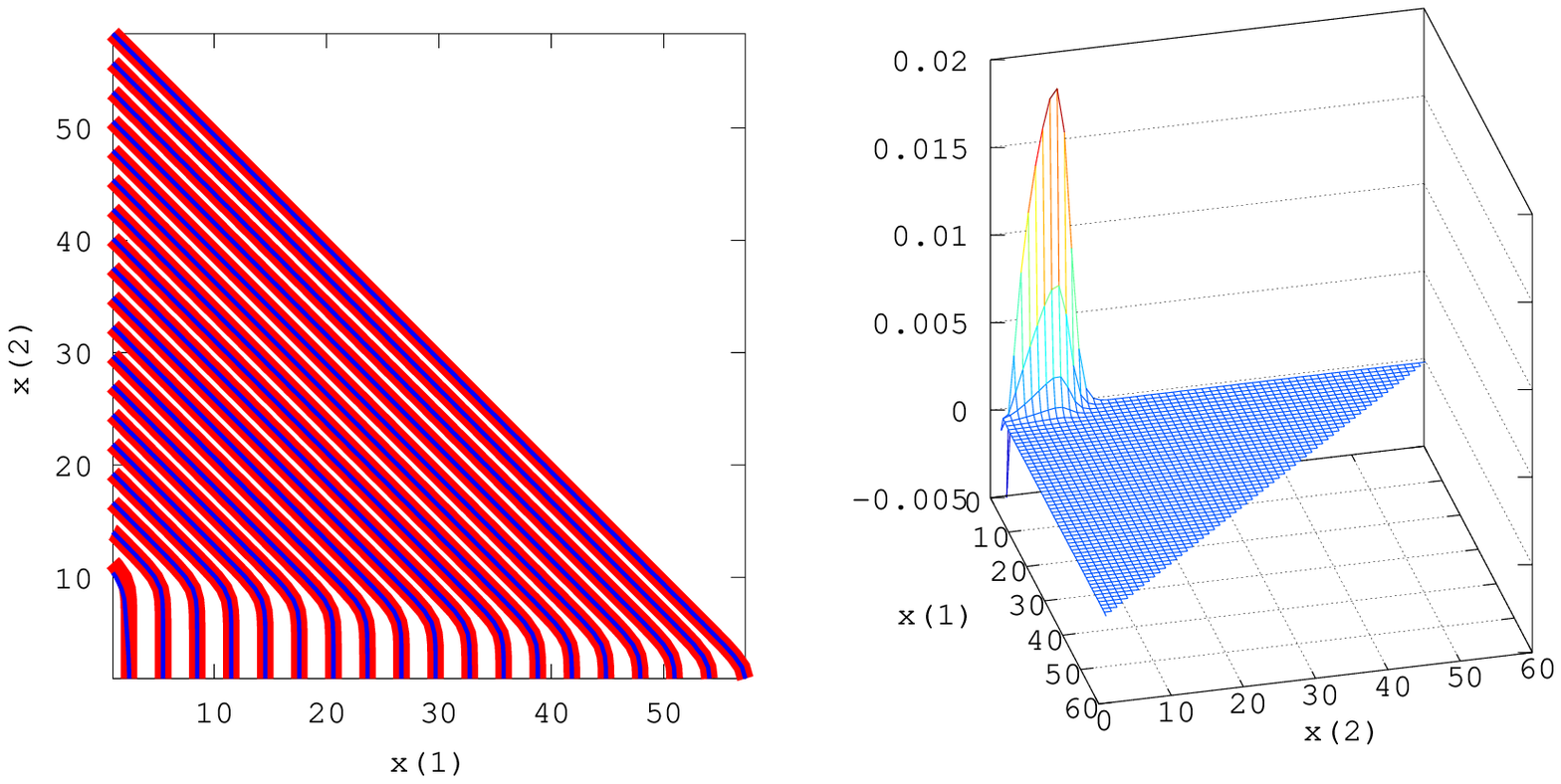}{On the left: level curves of $V_n$ (thin blue) and
$W_n$ (thick red); on the right: the graph of $(W_n-V_n)/W_n$ }{0.5}

The first graph in Figure \ref{f:relativeerror3} are the level curves
of $W_n$ of $V_n$; they all completely overlap except for the first one along the
$x(2)$ axis. The second graph shows the relative error $(W_n-V_n)/V_n$; we see that
it appears to be zero except for a narrow layer around $0$ where it is bounded by
$0.02$. 

For $x=(1,0)$,
the exact value for the
probability $P_x(\tau_{60} < \tau_0)$ is $1.1285 \cdot 10^{-35}$ and the approximate
value given by \eqref{e:nicereptransformed} equals $1.2037 \cdot 10^{-35}$.
Slightly away from the origin these quantities quickly converge to each other. For
example,
$P_x(\tau_{60} < \tau_0) = 4.8364 \cdot 10^{-35}$,
$f(x) = 4.8148 \cdot 10^{-35}$ for $x=(2,0)$ and 
$P_x(\tau_{60} < \tau_0) = 7.8888 \cdot 10^{-31}$,
$f(x) = 7.8885 \cdot 10^{-31}$ for $x = (9,0).$

For $ x \in {\mathbb Z}_+^2$ and $g: {\mathbb Z}_+^2 \rightarrow {\mathbb R}$
let ${\mathcal D} g$ denote the discrete gradient of $g$:
\ninsepsc{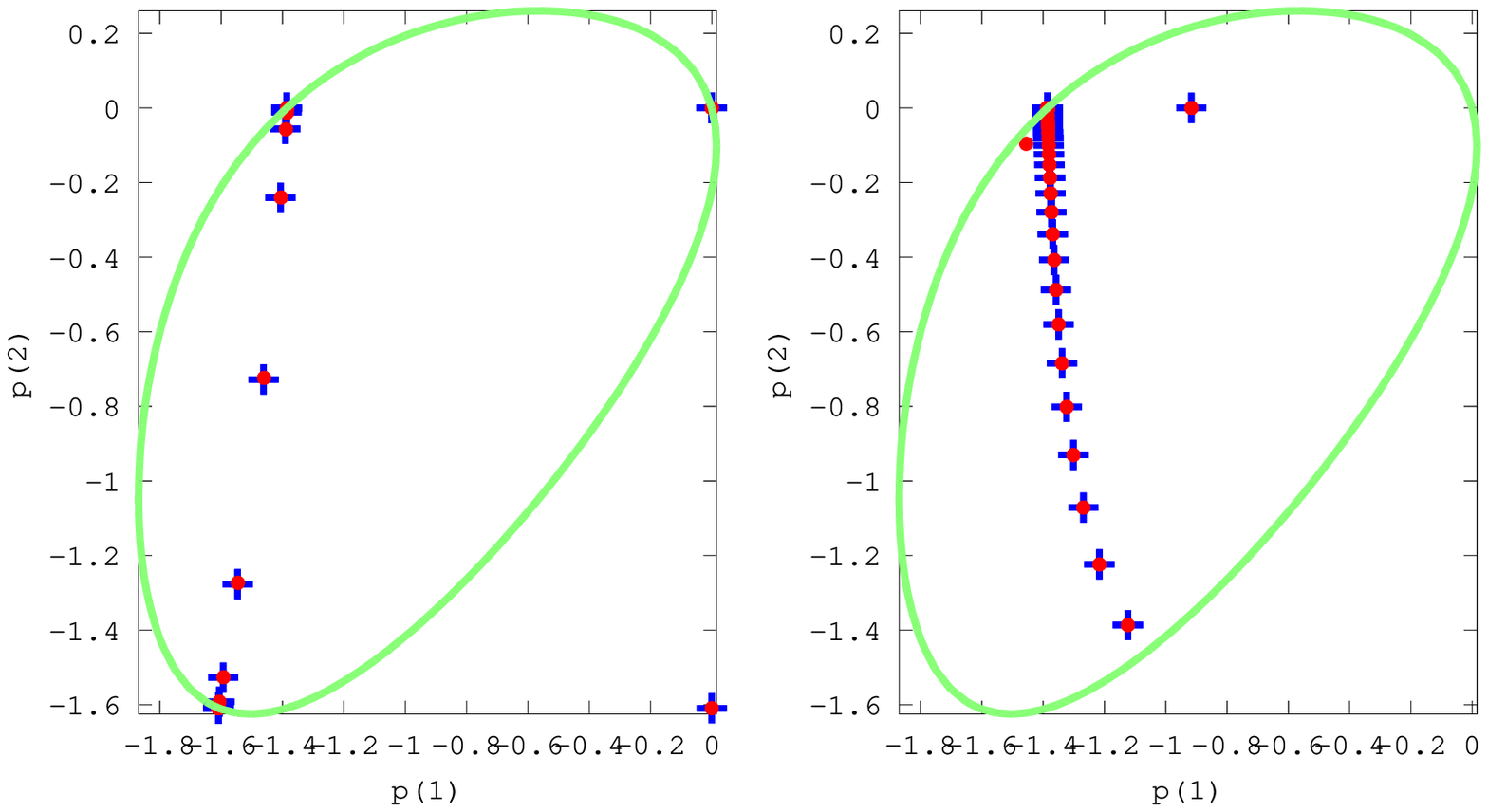}{${\mathcal D}V_n(x)$ and ${\mathcal D}W_n(x)$ , $x=(5,\cdot)$;
on the right the same functions for $x=(\cdot,1)$ }{0.5}
\[
({\mathcal D}g)(x) = (g(x + (1,0)) - g(x) , g(x+(0,1)-g(x)).
\]
The large deviations analysis of $V_n$ suggests that 
$n {\mathcal D}V_n$ approximately
equals $(\rho_1,0)$ in a region around the $x(1)$ axis and $(\rho_2,\rho_2)$
elsewhere. These discrete gradients play a key role in the importance sampling
estimation of the probability $P_x(\tau_n < \tau_0)$. Since \cite{thesis}, it has
been of interest to the author to understand how $n {\mathcal D} V_n(x)$ transitions from 
$(\rho_1,0)$ to $(\rho_2,\rho_2)$ as $x$ moves from the $x(1)$-axis to the interior
of $A_n$. The approximation of $V_n$ by $f$ also explains how this transition takes
place. As an example let us graph the values of these gradients over the line
$x(1) = 5$ (any $x(1)$ value slightly away from $0$ will give similar results).
The left panel of Figure \ref{f:discretegrad1} shows the discrete gradients 
${\mathcal D}V_n$ and ${\mathcal D} W_n$ along this line; they overlap.
The right panel of the same figure shows the same gradients over the line $x(2) = 1$.

\subsection{General two dimensional walk}\label{ss:gtdw}
Now let us consider the two dimensional 
network with the transition matrix
\begin{equation}\label{e:pofex82}
p = 
\left( \begin{matrix}
0 & 0.15 & 0.1\\
0.2 & 0 &  0.1\\
0.24 &  0.06&  0
\end{matrix}
\right).
\end{equation}
For this example, we will need to use both $T_n^i$, $i=1,2$,
to get good approximations of $P_x( \tau_n < \tau_0)$ across all of $A_n.$
These two transformations will give us two limit unstable processes $Y^1$ and
$Y^2$.
The first will give good approximations
along $\partial_2$ and the second along $\partial_1$. To combine their results into a
single function, we will take their maximum.

The limit processes $Y^1$ and $Y^2$ have the following dynamics:
\[
Y^i_{k+1} = Y^i_k + \pi_i(Y^i_k, J^i_k), ~~i=1,2,
\]
where $J^i=J$ of \eqref{e:defJ} with $i=1,2$.
\begin{remark}{\em
$T_n^{2}$ equals $T_n^{1}$ after we exchange the node labels (i.e., node $1$
becomes $2$ and $2$ becomes $1$)
This allows one to use the same computer code
to compute either of the approximations by reordering the elements of the matrix $p$.
}
\end{remark}

We want to compute
\[
P_{y}( \tau^1 < \infty), P_{y} (\tau^2 < \infty),
\]
where
$ \tau^i \doteq \inf \{ k: Y^i_k \in \partial B \}$.
We no longer have explicit finite formulas for these as we did in the tandem case.
We will instead use a linear combination (a superposition) of
basis functions of subsection \ref{ss:modifiedfourier} to approximate
the function mapping $\partial B$ to $1$; the same linear combination of the
Balayage of the basis functions (for which we have explicit formulas)
will provide an approximation for the probabilities we seek. 
One way to do this is as follows (we
give the details for $P_y ( \tau^1 < \infty)$, the procedure is identical
for $P_y(\tau^2 < \infty)$ ).
As a first order approximation we use
\[
{\boldsymbol A}_0 \doteq \frac{1}{C(r,{\boldsymbol\alpha}(r,1))} h_r =  
[(r,1),\cdot] - 
\frac{C(r,1)}
{C(r,{\boldsymbol \alpha}(r,1))}[(r,{\boldsymbol \alpha}(r,1)),\cdot].
\]
By Proposition \ref{p:harmonicYtwoterms},
 ${\boldsymbol A}_0$ is a harmonic function of $Y^1$.

For the $p$ of \eqref{e:pofex82},
$\beta_1(1)=r = 0.42373$ and
${\boldsymbol \alpha}(r,1) = 0.48123$.
Then, by Proposition \ref{p:balayagesimple}, ${\boldsymbol A}_0$ is $\partial B$ determined.
These imply
\begin{equation}\label{e:crudeapprox}
|P_y( \tau^1 < \infty) - {\boldsymbol A}_0| \le P_y(\tau^1 < \infty)
\max_{ y \in \partial B}
|{\boldsymbol A}_0(y) - 1|.
\end{equation}
Set
\[
c_7 \doteq 
- 
\frac{C(r,1)}{C(r,{\boldsymbol \alpha}(r,1))}.
\]
${\boldsymbol A}_0 -1 = c_7[(r,{\boldsymbol \alpha}(r,1)),\cdot]$ is geometrically 
decreasing on $\partial B$ and
therefore it takes its greatest value at $n=0$ where it equals, 
for the present example, $3.8418< 4$. This and \eqref{e:crudeapprox} imply
\[
\frac{1}{5} {\boldsymbol A}_0 < P_y ( \tau^1 <\infty) < {\boldsymbol A}_0.
\]
Thus, even with a single $Y$-harmonic pair of $\log$-linear functions,
we are able to approximate $P_y(\tau^1 < \infty)$
up to a constant term.
To improve, approximate 
\[
c_7 [r,{\boldsymbol \alpha}(r,1)),\cdot]
\]
by a superposition of harmonic pairs of subsection \ref{ss:modifiedfourier}
as follows. 
Consider the vector 
$\boldsymbol b \doteq 
c_7 ( 
[(r,{\boldsymbol\alpha}(r,1)),({\mathrm y},{\mathrm y}) ]
, {\mathrm y}\in \{0,1,2,...,K\}) 
\in {\mathbb C}^{K+1}$. If one thinks of 
the restriction of
$c_7[(r,{\boldsymbol\alpha}(r,1)),\cdot]$ to $\partial B$
as a sequence, one truncates it to its first $K+1$ components to 
get $\boldsymbol b$; for $K$ large (for the present example we take $K=11$), 
the remaining tail 
of
$c_7[(r,{\boldsymbol\alpha}(r,1)),\cdot]|_{\partial B}$
(its components from the $K+2^{nd}$ on)
will be almost $0$.
What we want to do is to
construct basis vectors ${\boldsymbol v}_i, i=0,1,2,...,K$, for ${\mathbb C}^{K+1}$
by truncating in the same way
the restrictions to $\partial B$
of $K+1$ log-linear $Y$-harmonic functions and write $\boldsymbol b$
as a linear combination of the members of this basis. 
To construct our first basis element 
take
the harmonic function $[(\beta(r_1),r_1),\cdot]$ of Proposition \ref{p:singletermharmY}
 and define 
${\boldsymbol v}_0 \doteq 
( [(\beta(r_1),r_1),({\mathrm y},{\mathrm y}) ], {\mathrm y} \in \{0,1,2,...,K\} )$.
This gives us a vector in ${\mathbb C}^{K+1}$; to complete it to a basis for
${\mathbb C}^{K+1}$ we need $K$ more vectors.  Set
$\alpha_j \doteq R e^{ik\frac{2\pi j}{n+1}}$, where $ R\in (r_1,1)$ is to be
specified and
consider the $Y$-harmonic functions
$h_{\beta_1(\alpha_j)}$ of Proposition \ref{p:harmonicYtwoterms}. 
We would like all of these to be $\partial B$-determined,
for which 
\begin{equation}\label{e:boundsonparam}
|\beta_1(\alpha_j)|, |\alpha_j|, |{\boldsymbol\alpha}(\beta_1(\alpha_j),\alpha_j)| < 1 
\end{equation}
suffice;
the second of these is satisfied by definition. The sufficient conditions
we have derived for the rest, listed as Proposition \ref{p:sufficientconditions},
don't cover the parameter values
of the present example (\eqref{e:simplifying0} and $p(0,2)=0$ fail). 
Then, what we will do is to compute them explicitly
(using \eqref{e:conjugator} for ${\boldsymbol \alpha}(\beta_1(\alpha_j),\alpha_j)$ 
and \eqref{e:formulaforbeta} for $\beta_1(\alpha_j)$)
and verify directly that \eqref{e:boundsonparam} holds.
Figure \ref{f:alphasbetas} shows the results of these calculations
for $R=0.7$ and $K=11$
and indeed we see that 
$|\beta_1(\alpha_j)|,|{\boldsymbol\alpha}(\beta_1(\alpha_j),\alpha_j)| < 1 $
holds for all $j$.
\ninsepsc{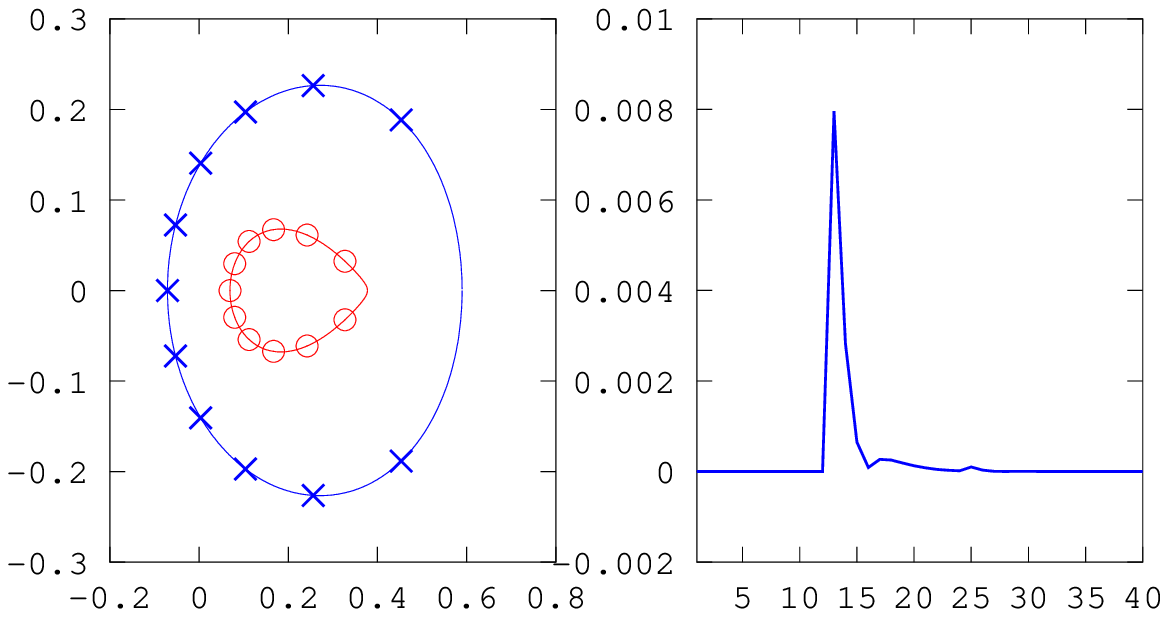}{
$\beta_1(\alpha_j)$'s
(shown with x's) 
and ${\boldsymbol\alpha}(\beta_1(\alpha_j),\alpha_j)$'s
(shown with 'o's) on the ${\mathbb C}$-plane; the graph of the
error $\Delta$ defined in \eqref{e:defDelta} }{0.8}
Thus, by Proposition \ref{p:balayagesimple} 
all $h_{\beta_1(\alpha_j)}$ are $\partial B$-determined.
Define
\[
{\boldsymbol v}_j 
\doteq ( h_{\beta_1(\alpha_j)}({\mathrm y},{\mathrm y}), {\mathrm y}= 0,1,2,...,K) \in {\mathbb C}^{K+1}.
\]
Define the change of basis matrix 
${\boldsymbol B}$ to consist of rows ${\boldsymbol v}_0$, 
${\boldsymbol v}_1$,...,${\boldsymbol v}_K$; 
directly evaluating
its determinant shows that ${\boldsymbol B}$ is invertable 
(this determinant is a polynomial in $\alpha_{j}$ and $\beta_j$, this can be used
to show that, perhaps after perturbing $\alpha_j$, we can always assume 
${\boldsymbol B}$ invertable).
Define the coefficient vector
\[
\psi \doteq {\boldsymbol b} {\boldsymbol B}^{-1}.
\]
By definition,
\[
{\boldsymbol A}_1 \doteq \psi(0)[(\beta(r_1),r_1),\cdot] +  
\sum_{j=1}^K \psi(j) h_{\beta_1({\alpha_j})}
\]
equals $c_7 [(r,{\boldsymbol \alpha}(r,1)),\cdot]$ 
over the set $\{ ({\mathrm y},{\mathrm y}), {\mathrm y} =0,1,2,...,K\}\subset \partial B$.
That
$|{\boldsymbol\alpha}(r,1)| < 1$, 
\eqref{e:boundsonparam} and $0 < r_1 < 1$ imply that 
\begin{equation}\label{e:defDelta}
\Delta({\mathrm y})\doteq |{\boldsymbol A}_1({\mathrm y},{\mathrm y}) - 
c_7 [(r,{\boldsymbol \alpha}(r,1)),({\mathrm y},{\mathrm y})]| \rightarrow 0
\end{equation}
exponentially as  $ {\mathrm y} \rightarrow \infty$. Then one can explicitly find a bounded
interval $[K+2,K']$ in which $\Delta({\mathrm y})$, $y \in {\mathbb Z}_+$,
takes its maximum value.
For $K=11$ and for the parameter values of the current example, this difference
takes its maximum value at ${\mathrm y}= 12$ (see the right panel of
Figure \ref{f:alphasbetas}) 
where it equals $0.00796 < 0.008$.
These imply
\[
0.992 ({\boldsymbol A}_0 + {\boldsymbol A}_1)< P_y( \tau^1 < \infty) < 1.008 ({\boldsymbol A}_0 + {\boldsymbol A}_1 ), y \in B.
\]
Set $g_1 = {\boldsymbol A}_0 + {\boldsymbol A}_1$.
Using exactly the same ideas we construct a function $g_2(y)$ approximating
$P_{y}( \tau^2 < \infty)$.
$g_1$ and $g_2$ give two possible approximate values for $P_x(\tau_n < \tau_0)$:
$g_1(T^1_n(x))$ and $g_2(T^2_n(x))$; as pointed out in subsection \ref{ss:combineapp}
one expects 
\[
f(x) = \max(g_1(T^1_n(x)),g_2(T^2_n(x)))
\]
to be the best approximation for $P_x(\tau_n < \tau_0)$ that one
can construct using $g_1$ and $g_2$. 
As in the previous section,
we compare $f$ and $P_x(\tau_n < \tau_0)$ 
in the logarithmic scale. Define $W_n(x) =-\frac{1}{n}\log f(x)$
($V_n(x)$ is, as before, $V_n(x) = -\frac{1}{n} \log P_x(\tau_n < \tau_0)$).
Figure \ref{f:relerr5}  shows the graph of $(V_n - W_n)/V_n$ for the present case: 
\ninsepsc{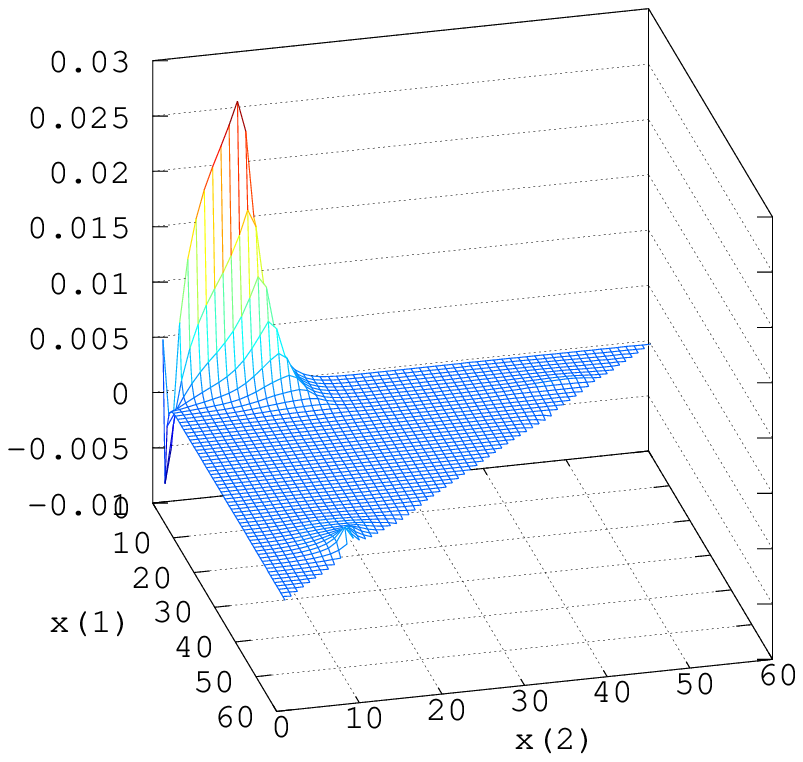}{Relative error for the nontandem two dimensional example}{0.7}
qualitatively it looks similar to the right panel of Figure \ref{f:relativeerror3};
the relative error is near zero across $A_n$, except for a short boundary layer
along the $x(2)$-axis; where it is bounded by $0.03$. One difference 
is the slight perturbation from $0$ of the relative error on $\partial A_n$ 
which comes from the approximation error 
depicted in the right panel of Figure \ref{f:alphasbetas}.

The level curves of $V_n$, $-\frac{1}{n}\log g_1(T^1_n(x))$ are shown
on the left panel of Figure \ref{f:twoapproxs}; those of $V_n$ and $
-\frac{1}{n}\log g_2(T^2_n(x))$
are given on the right panel. It is clear from
these graphs that, indeed, as discussed above, $g_i(T^i_n(x))$ 
approximates $V_n$ well away from $\partial_i$.
\ninsepsc{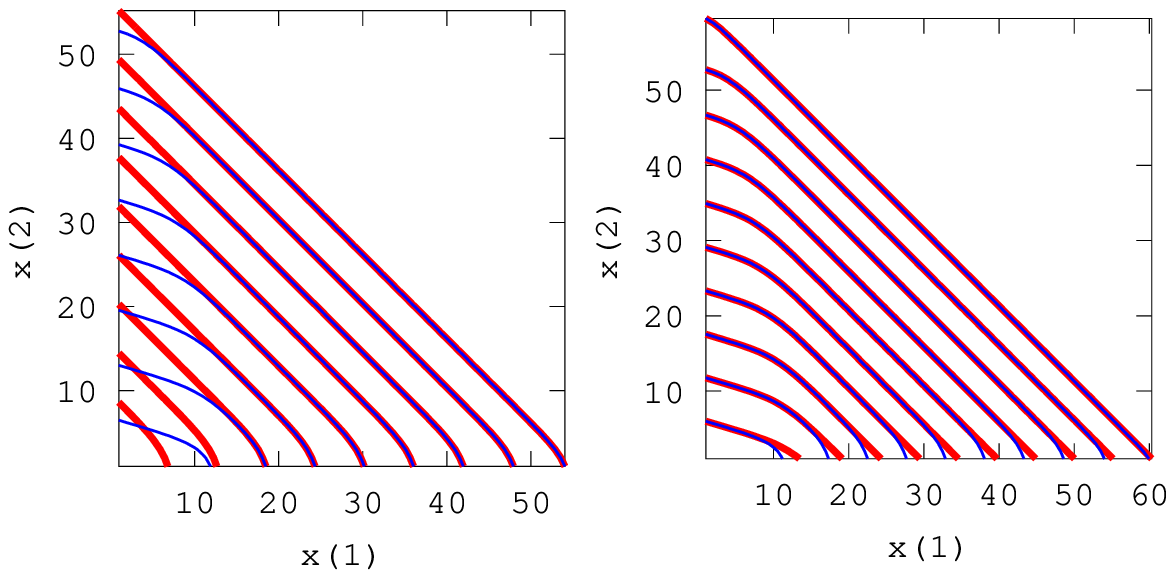}{Two approximations given by two limit processes}{0.8}

Finally, suppose we are given a function $h$ on $\partial A_n$. The algorithm above
can also be used to approximate 
${\mathbb E}_{y^i}\left[h(T^i_n(Y^i_{\tau^i}))1_{\{\tau^i < \infty\}}\right]$,
$i=1,2$, and hence
${\mathbb E}_x\left[ h(X_{\tau_n}) 1_{\{\tau_n < \tau_0\}} \right]$.
We leave the analysis of this approximation to future work.

\subsection{Tandem walk in higher dimensions}
Take a four dimensional tandem system with rates, for example,
\[
\lambda = 1/18, \mu_1 = 3/18, \mu_2 = 7/18, \mu_3 = 2/18, \mu_4 = 5/18.
\]
Let $f(y)$ denote the right side of \eqref{e:exactformulaDtandem}.
As before, define
$V_n = -\log(P_x(\tau_n< \tau_0))/n$ and $W_n = -\log f(T_n(x))/n$.
The level curves of $V_n$ and $W_n$
and the graph of the relative error $(V-W)/V$ 
for $x= (0,0,i,j)$,
$i,j \le n$, are shown in Figure \ref{f:relativeerror4}; once again, qualitatively,
these graphs show results similar to those observed for the earlier examples:
almost zero relative error across the domain selected, except for a boundary layer
along the $x(4)$-axis, where the relative error is bounded by $0.05.$

\ninsepsc{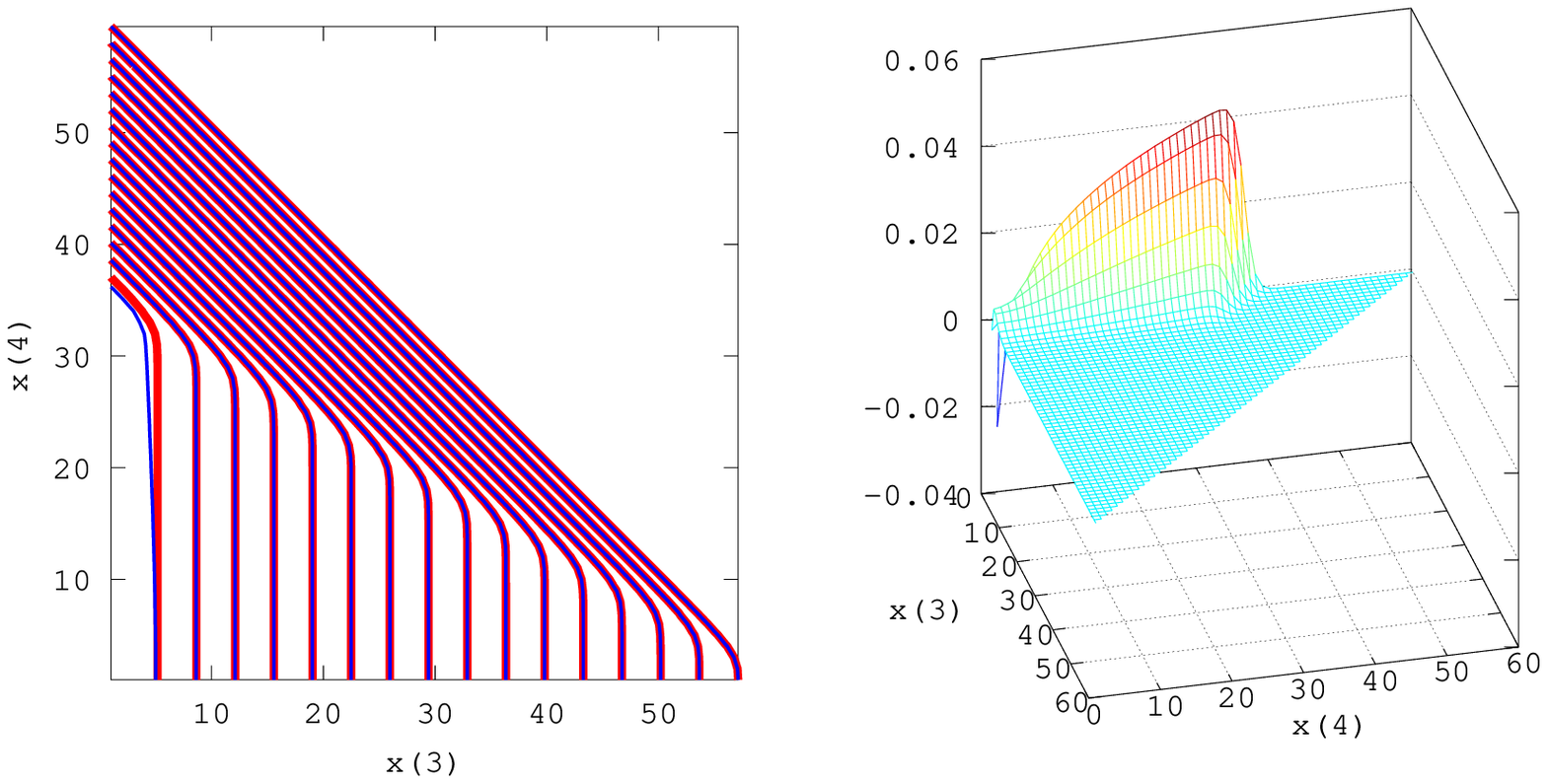}{Level curves and relative error in four dimensions}{0.65}

As our last example, consider the $14$-tandem queues with parameter values shown in
Figure \ref{f:ratesex14}.

\ninsepsc{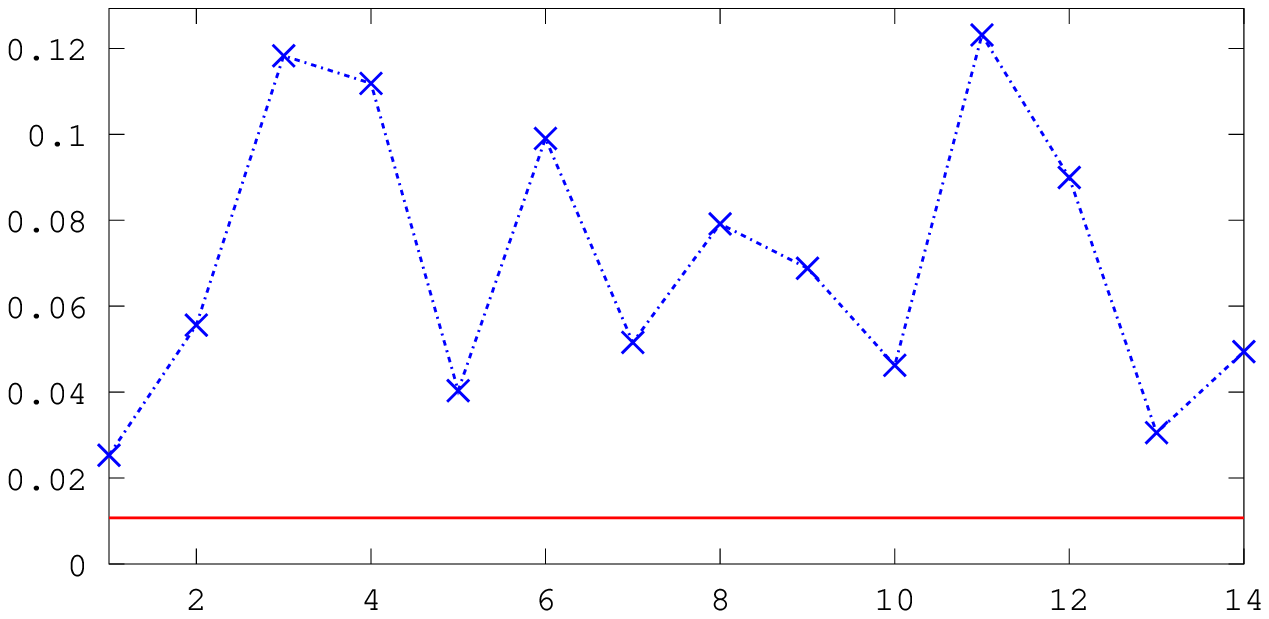}{The service rates (blue) and the arrival rate (red) 
for a $14$-dimensional tandem Jackson network}{0.8}

For $n=60$, $A_n$ contains $60^{14}/14!=  8.99 \times 10^{13}$
states which makes impractical an exact calculation via iterating \eqref{e:linear}.
On the other hand, \eqref{e:exactformulaDtandem} has $2^{14}-1=16383$ 
summands and
can be quickly calculated. Define $W_n$ as before. 
Its graph over $\{x: x(4) + x(14)= 60, x(j) = 0, j \neq 4,14\}$
is depicted in Figure \ref{f:meshex14}. 
\ninsepsc{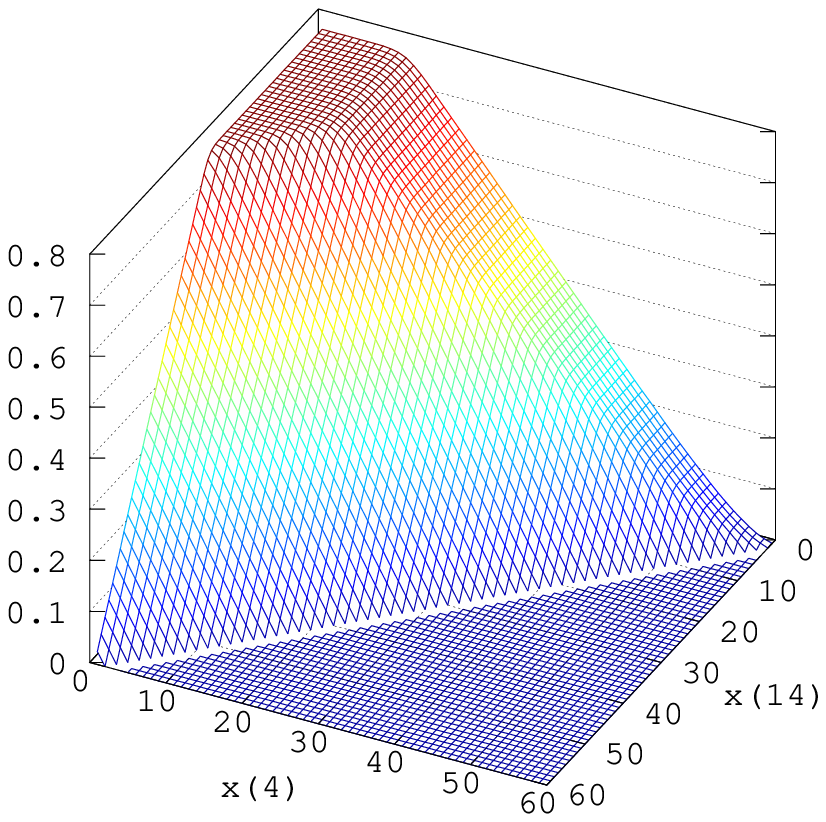}{The graph of $W_n$ over  $\{x: x(4) + x(14)= 60, x(j) = 0, j \neq 4,14\}$}{0.7}
For a finer approximation of $P_{(1,0,\cdots,0)} (\tau_n < \tau_0)$ we
use importance sampling based on $W_n$. With $12000$ samples, IS gives the
estimate $7.53 \times 10^{-20}$ with an estimated $95\%$ confidence interval  
$[6.57, 8.48] \times 10^{-20}$ (rounded to two significant figures). 
The value given by our approximation \eqref{e:exactformulaDtandem}
for the same probability is $f((1,0,\cdots,0)) = 1.77 \times 10^{-20}$ 
which is approximately $1/4^{th}$ of the estimate given
by IS. 
The LD estimate of the same probability is
$(\lambda/\min_{i=1}^{14}(\mu_i))^{60} = 4.15\times10^{-23}$.
The discrepancy between IS and \eqref{e:exactformulaDtandem} 
 quickly disappears as $x(1)$ increases. 
For example, for $x(1) = 4$, IS gives $2.47\times 10^{-19}$ and \eqref{e:exactformulaDtandem}
gives $2.32 \times 10^{-19}$.

\section{Conclusion}\label{s:conclusion}
The foregoing analysis points to a number of future directions for research. We state
some of them here.

\subsection[Constrained diffusions with drift]{Constrained diffusions with drift and
elliptic equations with Neumann boundary conditions}\label{ss:diffusionswithdrift}
Diffusion processes are weak limits of random walks. Thus, the results of
the previous sections can be used to compute/approximate Balayage and 
exit probabilities of constrained unstable diffusions. We give an example
demonstrating this possibility.

For $a,b > 0$
let $X$ be the the constrained diffusion on ${\mathbb R} \times {\mathbb R}_+$
with infinitesimal generator $L$ defined as
\[ f \rightarrow Lf,
Lf = \langle \nabla f,  ((2a+b), (a-b) ) \rangle + \frac{1}{6} \nabla^2 f \cdot  \left( \begin{matrix} 	
		2 & 1 \\
		1 & 2
\end{matrix}
\right),
\]
where $\nabla^2$
 denotes the Hessian operator, mapping $f$ to its matrix of second order
partial derivatives.
On $\{x: x(2) = 0\}$ $X$ is pushed up to remain in ${\mathbb R} \times {\mathbb R}_+$
(the precise definition involves the Skorokhod map, see, e.g., \cite{kushner2001numerical}).
$a,b > 0$ implies that, starting from $B = \{x: x(1) > x(2)\}$, $X$ has positive probability
of never hitting $\partial B = \{x:x(1) = x(2)\}.$
Let $\tau$ be the first time $X$ hits $\{x:x(1) = x(2)\}$.
Proposition \ref{p:exactformulaDtandem} for $d=2$ suggests
\begin{align}\label{e:exactforumlafordiffusions}
P_x(\tau < \infty) &= 
e^{-(a + 2b) 3(x(1)-x(2))} + \notag
\frac{a+2b}{a-b} e^{-(a + 2b) 3(x(1)-x(2))} e^{-(2a + b) 3 x(2)}\\
&~~-\frac{a+2b}{a-b}e^{-3(2a+b) x(1)},
x \in B.
\end{align}
One can check directly that the right side of the last display satisfies
\[
L V= 0,~~ \langle \nabla V, (0,1) \rangle = 0, x \in \partial_2.
\]
This and a verification argument similar to the proof of 
Proposition
\ref{p:balayagesimple} will imply \eqref{e:exactforumlafordiffusions}.
Almost the same argument for general $d$ gives an explicit formula
for the solution of 
the $d$ dimensional version of the above elliptic equation 
on ${\mathbb R} \times {\mathbb R}_+^d$,
with $2^{d-1}-1$ Neumann boundary conditions on 
$\partial({\mathbb R} \times {\mathbb R}_+^d).$
\subsection{Solutions to perturbed nonlinear PDE}\label{ss:perturbedPDE}
As indicated in the introduction, classical large deviations analysis 
leads (at least for Markov processes) 
to a deterministic first order
HJB equation. To improve the approximation provided by the solution of this
first order PDE, one can add nonlinear second order perturbation
terms to it \cite{MR758258}. 
We expect the ideas of the paper to bear on the task of computing approximate 
solutions of the perturbed second order nonlinear PDE related to the probabilities
treated in the present paper.

\subsection{Extension to other processes and domains}\label{ss:possiblegen}
In the foregoing sections, we have
computed approximations to the  Balayage operator and exit probabilities of a class of
constrained random walks in two stages: 1) use an affine
change of coordinates to move the origin to a point on the exit boundary
and take limits;
as a result, some of the constraints in the prelimit
process disappear and one obtains as a limit process an unstable constrained random walk;
2) find a class of basis functions on the exit boundary on which the 
Balayage operator of the limit process has a simple action; then 
try to approximate any other function on the exit boundary with linear
combinations of the functions in the basis class.
The type of problem we have studied here is of the following form:
there is a process $X$ with a certain law of large number limit which 
takes $X$ away from a boundary $\partial A_n$ towards a stable point 
or a region; $\tau_0$ is the first time the process gets into this
stable region. We are interested in the probability $P(\tau_n < \tau_0)$
and the associated Balayage operator. We expect the first step to be 
applicable to a range of problems that fit into this scenario.
The second stage obviously depends on the 
particular dynamics of the original 
process and the geometry of the exit boundary. It remains to be explored
for which processes and boundaries it is possible to construct classes of simple
basis functions. Even very simple changes in the dynamics or in the boundary geometry
from those covered in this paper may lead to different types of basis functions.
We hope to treat the tandem walk case with a separate boundary in an upcoming work.

\subsection{Harmonic functions with polynomial terms}\label{ss:harmpoly}
Let us confine ourselves, for the purposes of this brief comment, to tandem queues.
If $\mu_1 =\mu_2$,  \eqref{e:conjnotharmonic} no longer holds and 
indeed \eqref{e:nicerep} is not well defined (because of the
$\mu_1 -\mu_2$ in the denominator of the first ratio). One way to remedy this
is to replace the right side of \eqref{e:nicerep} by the limit of the same
expression as $\mu_1 \rightarrow \mu_2$, which gives
\begin{equation}\label{e:nicerepmue}
P_y ( \tau < \infty) = 
\rho^{y(1)} + \frac{\mu-\lambda}{\mu}(y(1) - y(2)) \rho^{y(1)-y(2)},
\end{equation}
where $\mu_1 = \mu_2 = \mu$ and $\rho = \lambda/\mu.$
Similarly, in three dimensions one gets, for example,
for $\mu_1 = \mu_2 = \mu_3$ 
\[
P_y(\tau < \infty) 
= \rho^{\bar{y}(1)}
\left(
\frac{1}{2}c_0^2 (\bar{y}(1))^2 \rho^{y(2)+y(3)}
+  \rho^{y(3)} \left( \left( \frac{c_0^2}{2} + y(3) c_0^2 \right) 
\rho^{y(2)} + c_0\right) \bar{y}(1) +1 \right),
\]
where $c_0 = (\mu - \lambda)/\mu$
and $\bar{y}(1) = y(1) -( y(2) + y(3))$.
Similar limits can be computed explicitly for the cases
$\mu_1 = \mu_2 \neq \mu_3$, $\mu_1=\mu_3 \neq \mu_2$ and
$\mu_1 \neq \mu_2  = \mu_3$.
Generalization of these results to higher dimensions and more general
topologies remain for future work.

\subsection[Boundary layers of $V_n$ and importance sampling]{The boundary layers of $V_n$ and subsolution based importance sampling algorithms}\label{ss:boundarylayers}
The works \cite{thesis,DSW} develop IS algorithms based on subsolutions
of \eqref{e:HJB0} to estimate $p_n$. These works and others which followed them
express most of their functions 
in a law of large numbers scale
(as we do in Sections \ref{s:convergence1}
and \ref{s:convergence2}).
We will express everything in unscaled coordinates in the discussion below.
Again, to be brief, we will limit ourselves to formal comments
on two tandem queues. Details and generalizations remain for future work.

For certain values of system parameters, $V_n$ of \eqref{e:defVn}
may manifest a boundary layer, where
its discrete gradient sees a rapid change near the boundaries of its state
space. This happens when $\mu_1 = \mu_2$  
for the case of two tandem queues with a boundary layer along the
$x(1)$ axis.
Here is one interpretation of the subsolution approach of \cite{DSW,thesis}
in the present context (i.e., two tandem queues and $\mu_1 = \mu_2$):
the discrete gradient of the large deviation value function 
\[
\bar{V}_n(x) \doteq -\frac{1}{n} \log \left( \rho^{n - (x(1) + x(2))}\right)
\]
approximates the discrete gradient of $V_n$
well everywhere except along a boundary layer
along the $x(1)$ axis; the subsolutions constructed in \cite{DSW,thesis}
are perturbations of $\bar{V}_n$ which attempt to approximate the discrete gradient
of $V_n$ also in this boundary layer.
The subsolutions constructed in these works involve a parameter
$\epsilon_n$ that satisfy 
\begin{equation}\label{e:vancondepsn}
n\epsilon_n\rightarrow \infty,~~\epsilon_n \rightarrow 0
\end{equation}
and have boundary
layers of constant width parallel to the $x(1)$ axis; 
$\epsilon_n$ determines the width of the boundary layer.
Although 
\cite[Theorem 3.8]{DSW} says that \eqref{e:vancondepsn} suffices for
the IS algorithm defined by the subsolutions to be asymptotically optimal,
\cite[subsection 2.3.3]{thesis}
observes that a good performance of the algorithm for finite $n$ (accurate estimation
results with bounded estimated relative error) requires
a finer specification of $\epsilon_n$ and hence of the size of the
approximating  boundary layer. 
How to measure the
size of the boundary layer of $V_n$
(and use this information to specify $\epsilon_n$)
more precisely
has remained an open problem since the inception of the subsolution approach.
With \eqref{e:nicerepmue} we now see that the correct way to add a
boundary layer to $\bar{V}_n$ (so that it has a boundary layer mimicking
that of $V_n$) is to
perturb it to
\begin{equation}\label{e:defWn}
W_n(x) = -\frac{1}{n}\log\left(
\rho^{n-(x(1) + x(2))} 
+ \frac{\mu-\lambda}{\mu} (n-(x(1)  + x(2))) \rho^{n-x(1)} \right).
\end{equation}
Proposition \ref{t:guzel} and the numerical example of subsection \ref{ss:ex1}
suggest
that 
the boundary layer of $W_n$ matches that of $V_n$ as $n$ increases.

The definitions \eqref{e:subsolpa} and \eqref{e:subsolpas} give an alternative
construction of smooth subsolutions with explicit boundary layers
(the region where $\nabla V_i^{c,\epsilon} = {\boldsymbol r}_1$). In contrast,
the boundary layer of $W_n$ is expressed implicitly in its
definition \eqref{e:defWn}.
Let us now try to quantify explicitly
the size {\em and the shape} of the boundary layer of $W_n$ and thus that of $V_n$.

Define
\[
\hat{W}(y) \doteq -\log\left( 
\rho^{y(1)-y(2)} + \frac{\mu-\lambda}{\mu}(y(1) - y(2)) \rho^{y(1)}\right).
\]
Then $W_n(x) = \frac{1}{n} \hat{W}(T_n(x)).$
It suffices to calculate
the boundary layer of $\hat{W}$; this we can transform by $T_n$ to get that
of $W_n$.
We will specify the boundary layer
by its boundary 
${\boldsymbol l}:{\mathbb R}_+ \rightarrow {\mathbb R}_+$; 
the layer will be defined as $\{y: y(2) \le {\boldsymbol l}(y(1))\}.$ 
The defining property of the
layer is that it is the region where the gradient of $\hat{W}$  rapidly changes.
Away from the boundary $\{y(2) = 0\}$
$\hat{W}$ behaves like the linear function $(y(1)-y(2)) \log(\rho)$
whose directional derivative $\nabla_{(1,1)}\hat{W}$
 in the direction $(1,1)$ is zero. And indeed
the same is true for $\hat{W}$ itself on $\{y:y(1)=y(2)\}$, i.e.,
\[
\nabla_{(1,1)}( \hat{W}) \doteq \frac{\partial \hat{W}}{\partial y(1)}(y)
+
\frac{\partial \hat{W}}{\partial y(2)}(y) = 0, y \in \{y:y(1) = y(2)\}.
\]
Furthermore, for fixed $y(1)$, $\nabla_{(1,1)}\hat{W}(y)$ is decreasing in $y(2)$ and 
the above display implies that this directional derivative hits zero on $\partial_1$.
We will define ${\boldsymbol l}(y(1))$ as the point where the value of 
$\nabla_{(1,1)}\hat{W}$ 
is half of its value at $(y(1),0)$, i.e., 
\begin{equation}\label{e:defl}
{\boldsymbol l}(y(1)) \doteq y(2)^*
\end{equation}
where $y(2)^*$ is the unique solution of
\[
(y(1) - y(2)^* ) \left( 1+ \frac{1}{2} \frac{\mu-\lambda}{\mu} y(1) \right) = 
\frac{1}{2} y(1) \rho^{-y(2)^*}.
\]
$\boldsymbol l$ is increasing $y(1)$: this implies that when transformed by 
$T_n$ it defines
a boundary layer for $W_n$ (and hence for $V_n$) that narrows down as it extends
toward the point $(n,0)$. In contrast,  the subsolutions developed in \cite{DSW,thesis}
have  boundary layers of constant size, i.e., parallel to the $x(1)$ axis.
The graph of $x(1) \rightarrow {\boldsymbol l}((n-x(1)))$,  
$x(1) \in [0,n]$, and the level sets of $V_n$
for $n=40$, $\lambda = 0.2,$ and $\mu = 0.4$ are shown in Figure \ref{f:boundarylayer}.
\ninsepsc{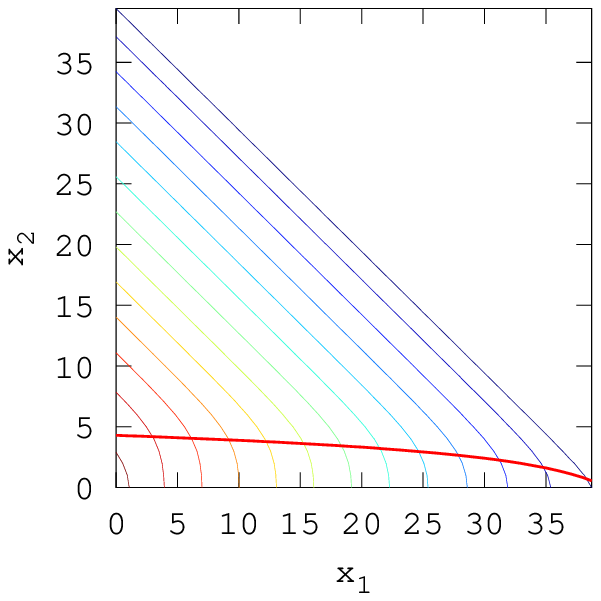}{The contours of $V_n$ for $n=40$, $\lambda = 0.2$ ,
$\mu = 0.4$ and its boundary layer computed using $\boldsymbol l$ of \eqref{e:defl}}{1}

\section*{Acknowledgement}
The great part of the research presented in this article has been made possible and funded
by the Rbuce-up European Marie Curie project, \url{http://www.rbuce-up.eu/},
and was carried out by the author at 
 L'Universit\'{e} d'Evry, Department of Mathematics, Probability and Analysis Laboratory,
\url{http://lap.maths.univ-evry.fr/},
 between November 2012 and October 2014. The author
is grateful to the Rbuce-up project and the
Probability and Analysis Laboratory of L'Universit\'{e} d'Evry, for this support.

\bibliography{balayage}

\end{document}